\documentclass[11pt]{article}
\usepackage{lmodern}
\usepackage{fullpage}
\usepackage{mathrsfs}
\usepackage{amssymb}
\usepackage{amsthm}
\usepackage{cjhebrew}
\usepackage{bbold}
\usepackage{stmaryrd}
\usepackage{setspace}
\usepackage{xspace}
\usepackage{enumerate}
\usepackage{enumitem}
\usepackage{mathtools}
\usepackage{thm-restate}
\usepackage[hidelinks]{hyperref}
\usepackage{cleveref}

\usepackage{tocloft}
\emergencystretch=\maxdimen
\allowdisplaybreaks

\SetSymbolFont{stmry}{bold}{U}{stmry}{m}{n}

\declaretheorem[sibling=basecase]{theorem, proposition, lemma, corollary, claim, conjecture, observation}
\declaretheorem[style=definition, sibling=basecase]{definition, convention, notation, example, remark, remarks, question, assumption}
\declaretheorem[name=Theorem, sibling=basecase, numbered=no]{theorem*}

\newlist{thmlist}{enumerate}{1}  
\setlist[thmlist]{label=\roman{thmlisti}., ref=\thetheorem.(\roman{thmlisti}),noitemsep} 

\addtotheorempostheadhook[theorem]{\crefalias{thmlisti}{theorem}}
\addtotheorempostheadhook[lemma]{\crefalias{thmlisti}{lemma}}
\addtotheorempostheadhook[proposition]{\crefalias{thmlisti}{proposition}}

\Crefname{theorem}{Theorem}{Theorems}
\Crefname{lemma}{Lemma}{Lemmas}
\Crefname{proposition}{Proposition}{Propositions}

\newcommand{\tf}[1]{\textnormal{\textsc{#1}}\xspace}
\newcommand{\ff}[1]{\textnormal{\textrm{#1}}\xspace}

\DeclareMathOperator{\ZFepsilon}{\tf{ZF}_{\varepsilon}}

\newcommand\cnot[1]{\mathrel{\ooalign{\hfil$#1$\hfil\cr\hfil$/$\hfil\cr}}}
\newcommand{\rlzin}{\mathop{\varepsilon}}
\newcommand{\notrlzin}{\cnot{\varepsilon}}
\newcommand{\falsity}[1]{\lVert #1 \rVert}
\newcommand{\verity}[1]{\lvert #1 \rvert}
\newcommand{\Perp}{\mathbin{\text{$\bot\mkern-11mu\bot$}}}
\newcommand{\point}{\boldsymbol{.}}

\newcommand{\stackapp}{\point}
\newcommand{\lambdaapp}{\point}

\DeclareMathOperator{\divline}{\hspace{0.05cm}\mid\hspace{0.05cm}}

\newcommand{\uminus}{\raisebox{.15\height}{\scalebox{0.75}{\ensuremath{-}}}}

\newcommand{\rlzfont}[1]{\textnormal{\textbf{#1}}}
\newcommand{\saverlz}[1]{\ff{k}_{#1}}
\newcommand{\rlzstr}{\tf{N}}
\newcommand{\rlzmodel}{\mathcal{N}}
\newcommand{\rlzset}{\mathcal{R}}
\newcommand{\fullname}[1]{\text{\Large \cjRL{r}}({#1})} 
\newcommand{\cjgimel}{\text{\Large \cjRL{g}}}
\newcommand{\identity}{\rlzfont{I}}
\newcommand{\cc}{\textnormal{\ff{cc}}}

\newcommand{\pointwise}{\textnormal{``}}
\newcommand{\inclusion}{\xhookrightarrow{}}
\newcommand{\sing}{\textnormal{\texttt{sng}}}
\newcommand{\up}{\textnormal{\texttt{up}}}
\newcommand{\op}{\textnormal{\texttt{op}}}
\newcommand{\app}[2]{#1 #2}
\newcommand{\fapp}[2]{(#1)#2}
\newcommand{\inapp}[2]{#1(#2)}
\newcommand{\twoapp}[2]{(#1)(#2)}
\newcommand{\bigfapp}[2]{\big(#1\big)\big(#2\big)}
\newcommand{\Biginapp}[2]{#1\Big(#2\Big)}
\newcommand{\Bigfapp}[2]{\Big(#1\Big)\Big(#2\Big)}
\newcommand{\superfapp}[2]{\bigg(#1\bigg)\Big(#2\Big)}
\newcommand{\lift}[1]{\tilde{#1}}
\newcommand{\imp}{\rightarrow}
\DeclareMathOperator{\downwards}{\downarrow}

\newcommand{\Booleanforce}[1]{\llbracket #1 \rrbracket}
\newcommand{\Booleanand}{\land}
\newcommand{\Booleanor}{\lor}
\newcommand{\Booleanproduct}{\bigwedge}
\newcommand{\Booleansum}{\bigvee}

\makeatletter
\providecommand*{\intcup}{
  \mathbin{
    \mathpalette\@intcup{}
  }
}
\newcommand*{\@intcup}[2]{
  \ooalign{
    $\m@th#1\cup$\cr
    \hidewidth$\m@th#1{\raisebox{0.5\height}{\scalebox{0.4}{\ensuremath \rlzmodel}}}$\hidewidth
  }
}
\makeatother

\title{A Guide to Krivine Realizability for Set Theory}
\date{}
\author{\vspace{-30pt} Richard Matthews}

\AtEndDocument{
  \medskip
  \textit{Email address: richard.matthews@u-pec.fr}
  
  \textit{URL: https://richardmatthewslogic.github.io/}
  
  \textsc{Univ. Paris Est Cr\'{e}teil, LACL, F-94010}
  }

\begin{document}
\maketitle

\vspace{-40pt}
\tableofcontents

\section{Introduction}

Realizability is a technique in mathematical logic which attempts to extract computational content from mathematical proofs. Originally invented by Kleene, \cite{Kleene1945}, to establish a correspondence between intuitionistic logic and Turing computable programs, it was then extended by Kleene to produce models of Heyting arithmetic and later by Friedman, \cite{Friedman1973}, to produce models of intuitionistic set theory. Realizability is seen as an extension of the \emph{proofs-as-programs} correspondence, which is also known as the \emph{Curry-Howard isomorphism}. Within intuitionistic logic, the analogous specification is known as the Brouwer-Heyting-Kolmogorov interpretation. The idea is that one can consider propositions as ``\emph{problems}'' which can be solved by being broken down into simpler problems. For example, the problem $A \rightarrow B$ represents the problem of producing a method which, when given a solution to the problem $A$, solves the problem $B$. From this, it follows that $p$ is a ``\emph{proof}'' of $A \rightarrow B$ if and only if whenever $q$ is a ``\emph{proof}'' of $A$ then the ``\emph{application}'' of $p$ to $q$ is a ``\emph{proof}'' of $B$. 

Intuitionistic realizability models can be built using $\lambda$-calculus which is a general method constructed by Church to create a formalisation of mathematics using a notion of computability in terms of application of functions. The notion of $\lambda$-calculus forms the basis for many functionally programming languages, for example Lisp, the Wolfram Language and Haskell. A dialect of Lisp is the programming language Scheme which extends $\lambda$-calculus with the additional control \emph{call-with-current-continuation}. Using this extended calculus, in \cite{Griffin1989}  Griffin showed how the Curry-Howard isomorphism could be extended to classical logic in a ``\emph{computationally interesting way}'' by producing a correspondence between control instructions and Peirce's law, a statement which is equivalent to the principle of Excluded Middle.

Building upon this idea, in a series of papers (notably \cite{Krivine2001, Krivine2009, Krivine2012, Krivine2015, Krivine2018}) Krivine used this extended calculus to extend the notion of realizability models to models of classical logic and then to models of \tf{ZF}. By adding the standard additional instruction \emph{quote} it was then showed how to produce realizability models of \tf{ZF +DC}. This idea was further generalised in \cite{FontanellaGeoffroy2020} to give realizability models of $\tf{DC}_\kappa$ for any fixed regular cardinal $\kappa$.

Furthermore, in \cite{Krivine2021}, Krivine went one step further to show that one can build programs which realize the Axiom of Choice and in fact any axiom which can be produced by the method of forcing. \\

\noindent The framework of Krivine's method of realizability gives a new method to produce independence results in set theory. For example, \cite{Krivine2012, Krivine2018} construct interesting new models of \tf{ZF + DC} with complicated algebraic structures on subsets of reals that are incompatible with the Axiom of Choice. It is worthwhile remarking that every known independence result using Krivine's method of realizability can also be produced using the more well-studied method of symmetric extensions (see \cite{Karagila2019} for details). However, what this machinery provides is a new technique where the ``simpler'' realizability algebras can already produce complicated sets which would involve technical results on symmetric forcing to otherwise produce. There is also the added bonus that, because we are using realizability, alongside the set theoretic structures we produce we also have interesting computational content. 

\subsection{Outline of the paper}

The main purpose of these notes is to explain what Krivine's method of realizability is and how it relates to classic constructions, in particular \emph{intuitionistic realizability} and \emph{double negation translations}. As such, most of the results we give are not new and most of the proof structures are heavily influenced by the original proofs by Krivine, notably in \cite{Krivine2012}. We will also not explicitly reproduce any of Krivine's models  for independence results over \tf{ZF}. Instead, these notes will focus on the theory of realizability itself and the general framework in which it can be used.

The principle new material is a modified formalisation of how to construct realizability models in \Cref{sec: construction of realizability models}. Instead of taking the entire ground model universe for the domain of our realizability model, we will instead recursively produce a class of names in a similar way to the class of forcing names or McCarty's construction of realizability models for intuitionistic set theory. The benefit of this construction is it allows for much cleaner and explicit proof of various concepts in realizability. We also give a reformulation of Krivine's \emph{adequacy lemma}, which is that the realizability structure is closed under classical deductions, using a formulation involving closure under rules of natural deduction without the need for types. Following the aim to give a more rigorous formalisation of realizability structures, has also led to a reformulation of the gimel operator, which we have call the \emph{reish}, or \emph{recursive}, operator. This is a general method to recursively construct a name in the realizability model from any given set in the ground model.

The second significant contribution in this work is an explicit construction of a name for the ordered pair of two sets in the realizability structure in \Cref{section:Pairing}. This method allows us to work with names for functions in a realizability model rather than having to enrich the language with function symbols. This more concrete notion gives us the ability to ``lift'' functions in the ground model to names for functions in the realizability model, from which it is possible to explicitly prove the preservation of characteristic properties of such functions. As a concrete example, we give a precise description of the Boolean algebra on the set of truth values.

The final notable piece of original work is a significant expansion of how to view any forcing model as a realizability algebra. Building on the work of Krivine and the presentation by Fontanella and Geoffroy in \cite{FontanellaGeoffroy2020} we explain a way to translate any Boolean valued model into a realizability model. We then give a precise characterisation of the relationship between the two resulting structures, allowing us to conclude that they give rise to the same theory. \\

The outline to these notes is as follows:
\begin{itemize}
    \item In \Cref{{section:IntuitionisticRealizability},{section:DoubleNegation},{section:LambdaTerms}} we give an overview of the main ingredients that combine to produce Krivine's method of realizability: intuitionistic realizability, double negation translations and lambda calculus.
    \item In \Cref{section:realizabilityintro} we discuss the intuition behind realizability models.
    \item \Cref{section:RealizabilityAlgebras} describes what a realizability algebra for Krivine's method is.
    \item \Cref{{section:ZFepsilon},{section:ZFepsilonToZF}} concerns the non-extensional theory $\ZFepsilon$ which is the underlying theory of the realizability models. We give the axioms of this theory before showing that it is a conservative extension of \tf{ZF}, and that one can restrict any model of $\ZFepsilon$ to one of \tf{ZF} in a precise way.
    \item In \Cref{sec: construction of realizability models} we give a modified formalisation of how to construct realizability models using a class of names.
    \item \Cref{{section:RealizersAndPropositionLogic},{section:Equality}} gives some initial realizers for various properties of predicate logic. We also introduce the additional symbol of \emph{non-extensional equality} which can be used to simplify many later constructions. In \Cref{section:Adequacy} we give a new formulation of Krivine's \emph{adequacy lemma} to show that any realizability model is a model of classical logic.
    \item In \Cref{section:RealzingZFepsilon} we prove that if the ground model is a model of \tf{ZF} then the realizability model satisfies every axiom of $\ZFepsilon$. Following on from this, in \Cref{section:ReishandOrdinals} we give a general discussion of ordinals in the realizability model. Using the generalisation of the gimel operator, we show how to form a proper class of ordinals in any given realizability structure. We also discuss an alternative construction of ordinals which appears in \cite{FontanellaGeoffroy2020} that is a more restrictive, but fruitful concept.
    \item \Cref{section:Pairing} concerns the concept of pairing in realizability models. Given names $a$ and $b$ in a realizability model we show how to produce names for sets which are extensionally equal to the singleton of $a$, the unordered pair of $a$ and $b$ and the ordered pair of $a$ and $b$. Using this, we show how to define the Cartesian Products of certain sets.
    \item Building upon the notion of ordered pairs, we define the concept of functions in models of $\ZFepsilon$. \Cref{section:Functions} gives the basic definitions of some of the different ways to formulate functions in non-extensional theories. In \Cref{section:LiftingFunctions} we then give a general methodology to lift certain functions from the ground model to names for functions in the realizability model and show that many of the desired properties can be conserved in this lift.
    \item \Cref{section:NEACandDC} discusses the \emph{Non-extensional Axiom of Choice}. This is an additional axiom which naturally holds in many realizability models, in particular any countable realizability model which contains the additional instruction \emph{quote}. It is shown how to derive the Axiom of Dependent Choices from this axiom.
    \item The final sections treat multiple interesting additional topics in realizability. Firstly, in \Cref{section:Reish2Size4}, we show that it is consistent for $\fullname{2}$, which is the class of truth values, to have size $4$. Then, in \Cref{section:ForcingAsRealizability}, we give an explicit translation of any Boolean valued model into a realizability model and show that the resulting model satisfies the same underlying theory. Finally, in \Cref{section:Fullname2Trivial} we give a partial converse to the previous result, which is to give a sufficient condition for a countable realizability algebra to be equivalent to a Boolean valued model; this is that the class of truth values is trivial, that is, it has size $2$.   
\end{itemize}

\subsection*{Acknowledgements}

The author is grateful to Laura Fontanella and Guillaume Geoffroy for multiple conversations on realizability and clarifying many important points on the subject and in the papers by Krivine. The work is funded by the grant DIM RFSI 21R03101S and the author is grateful for their support.

\newpage
\section{Intuitionistic Realizability} \label{section:IntuitionisticRealizability}

Krivine's method of realizability can be seen as the combination of two classical techniques from intuitionistic mathematics; intuitionistic realizability and double negation translations. Therefore, in order to better understand how the machinery works, we will briefly explain these two techniques. 

The study of intuitionistic logic was developed by Brouwer at the beginning of the twentieth century and, in its simplest description, can be considered as classical logic without the law of Excluded Middle. This statement can either be read as the claim that Excluded Middle fails (and thus the logic is incompatible with classical logic) or just that Excluded Middle is not included as an axiom of the logic (in which case classical logic is intuitionistic logic plus Excluded Middle). This then leads to the question of what it means to suggest ``Excluded Middle is incompatible with intuitionistic logic''. In the first scenario the addition of Excluded Middle would lead to a contradiction whereas in the second scenario it merely leads to classical logic which was what we were trying to avoid assuming. For ease of presentation, we will take the second of these approaches and therefore the aim will be to avoid assumptions that would lead to full classical logic.

As such, we wish to avoid logical principles such as double negation implication ($\neg \neg \varphi \rightarrow \varphi$) or equivalences between some logical symbols (for example $(\varphi \rightarrow \psi) \rightarrow (\neg \varphi \lor \psi)$). Furthermore, it is well known that the scheme of Foundation implies the law of Excluded Middle and, by Diaconescu's Theorem \cite{Diaconescu1975}, as does the Axiom of Choice (in either the form that there exist choice functions or that every set can be well-ordered). On the other hand, it is known that Foundation's classically equivalent couterpart of the $\in$-Induction Scheme is intuitionistically compatible and is often included in the formulation of intuitionistic theories. Similarly, Grayson (\cite{Grayson1975, Grayson1978}) proves that, starting from a model of \tf{ZFC}, Zorn's lemma holds in every Heyting-valued model and is thus consistent with \tf{IZF}. We refer the interested reader to the last section of \cite{Bell2011} for a proof of this. It is worth observing here that when we formulate the theory $\ZFepsilon$ we shall do so with Induction. Moreover, for any cardinal $\kappa$ it is possible to realize a version of Zorn's lemma of length $\kappa$ in Krivine's realizability models, see \cite{FontanellaGeoffroy2020} for more details. \\

\noindent A general guiding principle for intuitionistic logic is the \emph{Brouwer-Heyting-Kolmogorov Interpretation} which is an analogous statement to the Curry-Howard correspondence. This is an informal idea as to how one should interpret logical connectives in terms of what they can prove. The interpretation is purposely stated in a vague way without a formal definition of any of the concepts such as ``proof'' or ``program''. 

\begin{definition}[The Brouwer-Heyting-Kolmogorov Interpretation] \,
    \begin{itemize}
        \item There is no proof of $\perp$.
        \item $p$ is a proof of $\varphi \land \psi$ iff $p$ is a pair $( q, r )$ where $q$ proves $\varphi$ and $r$ proves $\psi$.
        \item $p$ is a proof of $\varphi \lor \psi$ iff $p$ is a pair $( n, q )$ where $n = 0$ and $q$ proves $\varphi$ or $n = 1$ and $q$ proves $\psi$.
        \item $p$ proves $\varphi \rightarrow \psi$ iff $p$ is a program which transforms any proof of $\varphi$ into a proof for $\psi$.
        \item $p$ proves $\neg \varphi$ iff $p$ proves $\varphi \rightarrow \perp$.
        \item $p$ proves $\exists x \varphi(x)$ iff $p$ is a pair $( a, q )$ where $q$ is a proof of $\varphi(a)$.
        \item $p$ proves $\forall x \varphi(x)$ iff $p$ is a program such that for any set $a$, $p(a)$ is a proof of $\varphi(a)$. 
    \end{itemize}
\end{definition}

\noindent Kleene's original method of realizability is now seen as a realisation of this interpretation\footnote{However it is worth pointing out that while Kleene was aware of the BHK interpretation he considered it more of a hindrance than help in formalising realizability \cite{Kleene1971}.} and gives an explicit connection between intuitionistic logic and computable functions. In the original paper, \cite{Kleene1945}, S.C. Kleene developed the idea of realizability in order to show that Church's thesis\footnote{which is the claim that if $\forall n \in \mathbb{N} \exists m \in \mathbb{N} \varphi(x, y)$ then there exists a recursive function $f$ such that $\forall n \in \mathbb{N} \varphi(n, f(n))$} can be consistently added to Heyting arithmetic. This was done by recursively defining what it means for a natural number $n$ to realize a formula $\varphi$ in arithmetic, written $n \Vdash \varphi$, as follows:
\begin{align*}
    n & \Vdash t = s && \quad \text{ iff } \qquad  t = s, \\
    n & \Vdash \varphi \land \psi && \quad \text{ iff } \qquad \mathbf{1^{st}}(n) \Vdash \varphi \text{ and } \mathbf{2^{nd}}(n) \Vdash \psi, \\
    n & \Vdash \varphi \rightarrow \psi && \quad \text{ iff } \qquad \text{for every } m \in \mathbb{N}, \text{ if } m \Vdash \varphi \text{ then } \{n\}(m) \downwards \text{ and } \{n\}(m) \Vdash \psi, \\
    n & \Vdash \varphi \lor \psi && \quad \text{ iff } \qquad \mathbf{1^{st}}(n) = 0 \text{ and } \mathbf{2^{nd}}(n) \Vdash \varphi \text{ or } \mathbf{1^{st}}(n) \neq 0 \text{ and } \mathbf{2^{nd}}(n) \Vdash \psi, \\
    n & \Vdash \forall x \varphi(x) && \quad \text{ iff } \qquad \text{for all } m \in \mathbb{N}, \, \{n\}(m) \downwards \text{ and } \{n\}(m) \Vdash \varphi(m), \\
    n & \Vdash \exists x \varphi && \quad \text{ iff } \qquad \mathbf{2^{nd}}(n) \Vdash \varphi(\mathbf{1^{st}}(n)),
\end{align*}
where $( \cdot ) \colon \mathbb{N} \rightarrow \mathbb{N} \times \mathbb{N}$ is some primitive recursive bijection whose projections are $\mathbf{1^{st}}$ and $\mathbf{2^{nd}}$, $\{n\}$ is the $n^{th}$ Turing machine (according to some fixed enumeration) and $\{n\}(m) \downwards$ is the assertion that the $n^{th}$ Turing machine halts on input $m$.

Kleene extended this idea to replace the natural numbers with other realizers such as recursive functions or more general algebras. The notion of realizability was further generalised from Heyting arithmetic to all Intuitionistic set theory by Myhill and Friedman \cite{Myhill1971} and McCarty \cite{McCartyPhD} and then to Constructive set theory by Aczel and Rathjen \cite{Rathjen2003}. For further details on the history of realizability we refer the reader to the survey article \cite{Oosten2002}. \\

\noindent It is worthwhile to finish this section with a few more details on McCarty's definition of the realizability structure because this will influence the presentation of Krivine realizability that we shall give. To fit with Kleene's original number realizability we will continue to use natural numbers, noting that $\omega$ can be replaced with any appropriately defined algebra. For simplicity, we will also use \tf{ZFC} as our background universe even though this is not necessary. 

We begin by defining a hierarchy $\tf{V}(Kl)_\alpha$ by recursion on the ordinals as: 
\[
\tf{V}(Kl)_\alpha \coloneqq \bigcup_{\beta \in \alpha} \mathcal{P}(\omega \times \tf{V}(Kl)_\beta)
\]
and set the universe of realizable sets to be $\tf{V}(Kl) \coloneqq \bigcup_{\alpha \in \tf{Ord}} \tf{V}(Kl)_\alpha$. It then follows that each element, $a$, of the realizable universe is a collection of pairs $( n, b )$ where $n \in \omega$ can be seen as a witness to the claim that $b$ is in $a$. For $a, b \in \tf{V}(Kl)$ one then defines what it means for $n \in \omega$ to realize the atomic formulas $a \in b$ and $a = b$ by transfinite recursion as
\begin{align*}
    n & \Vdash a \in b && \quad \text{ iff } \qquad  \text{there exists } c \in \tf{V}(Kl) \text{ such that } ( \mathbf{1^{st}}(n), c ) \in b \text{ and } \mathbf{2^{nd}}(n) \Vdash a = c, \\
    n & \Vdash a = b && \quad \text{ iff } \qquad \text{for all } m \in \omega \text{ and } c \in \tf{V}(Kl), \text{ if } ( m, c ) \in a \text{ then }  \{\mathbf{1^{st}}(n)\}(m) \Vdash c \in b, \\
    & && \phantom{ \quad \text{ iff } \qquad} \text{ and if } ( m, c ) \in b \text{ then } \{\mathbf{2^{nd}}(n)\}(m) \Vdash c \in a. 
\end{align*}
\pagebreak
Finally, one extends Kleene's notion of realizability to all first-order formulas as
\begin{align*}
    n & \Vdash \varphi \land \psi && \quad \text{ iff } \qquad \mathbf{1^{st}}(n) \Vdash \varphi \text{ and } \mathbf{2^{nd}}(n) \Vdash \psi, \\
    n & \Vdash \varphi \rightarrow \psi && \quad \text{ iff } \qquad \text{for every } m \in \omega, \text{ if } m \Vdash \varphi \text{ then } \{n\}(m) \Vdash \psi, \\
    n & \Vdash \varphi \lor \psi && \quad \text{ iff } \qquad \mathbf{1^{st}}(n) = 0 \text{ and } \mathbf{2^{nd}}(n) \Vdash \varphi \text{ or } \mathbf{1^{st}}(n) \neq 0 \text{ and } \mathbf{2^{nd}}(n) \Vdash \psi, \\
    n & \Vdash \forall x \varphi(x) && \quad \text{ iff } \qquad \text{for all } a \in \tf{V}(Kl), \, n \Vdash \varphi(a), \\
    n & \Vdash \exists x \varphi && \quad \text{ iff } \qquad \text{there exists } a \in \tf{V}(Kl) \text{ such that } n \Vdash \varphi(a).
\end{align*}
McCarty proves that in this way one can build a theory which is closed under deductions in intuitionistic logic and satisfies every axiom of \tf{IZF}.

\section{Double Negation Translations} \label{section:DoubleNegation}

The second ingredient that influenced Krivine's method of realizability is the \emph{double negation translation}. This is a technique developed by Kleene \cite{Kleene1952} and further extended to set theory by Friedman \cite{Friedman1973} in order to interpret classical mathematics in intuitionistic mathematics. 

The essential idea is, for given first-order formulas $\varphi$ and $\psi$, to produce two translations $\varphi^\star$ and $\psi^{\uminus}$ such that
\begin{itemize}
    \item If $\tf{IZF} \vdash \varphi$ (respectively $\tf{ZF} \vdash \varphi$) then $\tf{IZF} \setminus \ff{Ext.} \vdash \varphi^\star$ (respectively $\tf{ZF} \setminus \ff{Ext.} \vdash \varphi^\star$),
    \item If $\tf{ZF} \setminus \ff{Ext.} \vdash \psi$ then $\tf{IZF} \setminus \ff{Ext.} \vdash \psi^{\uminus}$.
\end{itemize}

From this, one can conclude that the four theories above are all equiconsistent. This is because if \tf{ZF} were to derive a contradiction then, by the first interpretation, so would $\tf{ZF} \setminus \ff{Ext.}$ and so, by the second interpretation, $\tf{IZF} \setminus \ff{Ext.}$ would also derive a contradiction. Finally, since $\tf{IZF} \setminus \ff{Ext.}$ is a subtheory of \tf{IZF}, we must also be able to derive a contradiction in \tf{IZF}. 

\begin{remark}
    In the axiomatisation of $\tf{ZF} \setminus \ff{Ext.}$, Friedman replaces the Replacement Scheme with the Collection Scheme and the Power Set axiom with the Axiom of Weak Power Set and this will also be how we formulate the theory $\ZFepsilon$ in \Cref{section:ZFepsilon}. The reason for replacing Power Set is because the double negation of this axiom does not appear to be provable in the theory \tf{IZF}, however the intuitionistically weaker (but under \tf{ZF} without Power Set equivalent) notion of weak power set is provable. 

    The reason for replacing the Replacement Scheme with the Collection Scheme is much clearer. This is because, by work of Scott \cite{Scott1961}, it is known that \tf{ZF} without Extensionality formulated with only the Replacement Scheme is equiconsistent with the much weaker theory of Zermelo Set Theory.
\end{remark}

\noindent The first interpretation, $^{\star}$, works by ``\emph{simulating}'' extensionality through a new definable relation, $\sim$. Essentially, $a \sim b$ indicates that there is some equivalence relation $E$ with $( a, b ) \in E$ and for any $x \in \ff{trcl}(a)$ there is some $y \in b$ such that $( x, y ) \in E$ and vice versa. Therefore, two sets are related if there is a single equivalence relation witnessing that; $a$ is related to $b$, every element of $\ff{trcl}(a)$ is related to an element of $\ff{trcl}(b)$ and every element of $\ff{trcl}(b)$ is related to an element of $\ff{trcl}(a)$.

We can then define a new relation $\in^\star$ by
\[
a \in^\star b \quad \Longleftrightarrow \quad \exists x \in b (x \sim a).
\]
Using this, the $^\star$ interpretation is defined by:

\begin{definition}
    For $\varphi$ a first-order formula, let $\varphi^\star$ be the formula which is abbreviated by the result of replacing each instance of $\in$ in $\varphi$ with $\in^\star$ and $=$ by $\sim$.
\end{definition}

\noindent For the second interpretation, we take a $\neg \neg$\emph{-translation}, the details of which can be found in Section 81 of \cite{Kleene1952}. The idea is to define a translation that transforms classically valid propositions into ones that are classically equivalent but still intuitionistically valid. In particular, we will have that if classical logic deduces $\varphi$ then intuitionistic logic deduces $\varphi^{\uminus}$ and classical logic deduces $\varphi$ if and only if it deduces $\varphi^{\uminus}$.

\begin{definition}
    The translation $\varphi^{\uminus}$ is defined recursively for $\varphi$ a first-order formula by the following specifications:
    
    \noindent \begin{tabular}{l l l l} \setlength\itemsep{0pt}
    a. & $(a \in b)^{\uminus}$ & \quad $\equiv$ \qquad & $\neg \neg (a \in b)$, \\
    ~b. & $(\varphi \land \psi)^{\uminus}$ & \quad $\equiv$ \qquad & $\varphi^{\uminus} \land \psi^{\uminus}$, \\
    ~c. & $(\varphi \lor \psi)^{\uminus}$ & \quad $\equiv$ \qquad & $\neg \neg (\varphi^{\uminus} \lor \psi^{\uminus})$, \\
    ~d. & $(\varphi \rightarrow \psi)^{\uminus}$ & \quad $\equiv$ \qquad & $\varphi^{\uminus} \rightarrow \psi^{\uminus}$, \\
    ~e. & $(\neg \varphi)^{\uminus}$ & \quad $\equiv$ \qquad & $\neg (\varphi^{\uminus})$, \\
    ~f. & $\forall x ~ \varphi(x, u)$ & \quad $\equiv$ \qquad & $\forall x ~ \varphi^{\uminus}(x, u)$, \\
    ~g. & $\exists x ~ \varphi(x, u)$ & \quad $\equiv$ \qquad & $\neg \neg \exists x ~ \varphi^{\uminus}(x, u)$.
\end{tabular}
\end{definition}

\noindent In \Cref{section:realizabilityintro} we will discuss how intuitionistic realizability and the double negation translation naturally led to Krivine's notion of realizability. However, before that we will briefly introduce the notion of \emph{lambda terms} which are an important part of the theory of realizability.

\section{Lambda terms} \label{section:LambdaTerms}

The final preliminary information we will need is the idea of lambda terms. We will only explain the basic concepts of this extensive area of research that are needed for the rest of this paper. We then refer the reader to \cite{Barendregt1985} for a more comprehensive study of $\lambda$ calculus.

$\lambda$-calculus was devised by Church in the 1930s as a simple semantics for defining computations, with the idea being to produce a formalisation of mathematics using functions as the basic concepts. One main principle of $\lambda$-calculus is that it treats functions ``\emph{anonymously}''. By this we mean that instead of defining the function $f$ such that $f(x) = x^2 + 1$, we write this as $\lambda u \lambdaapp u^2 + 1$. The class of terms will be defined recursively using two rules: \emph{application} and \emph{abstraction}.

\begin{definition}
    The class of $\lambda$-terms is defined to be the smallest class $X$ such that
    \begin{itemize} \setlength \itemsep{0pt}
        \item Every variable is in $X$,
        \item If both $s$ and $t$ are in $X$ then so is $\app{s}{t}$, \hfill (\emph{application})
        \item If $s$ is in $X$ and $u$ is a variable, then $\lambda u \lambdaapp s$ is in $X$. \hfill (\emph{abstraction})
    \end{itemize}
\end{definition}

\begin{remark} \,
    \begin{itemize} \setlength \itemsep{0pt}
        \item Application will be a left associative operation. So $tsr$ is $\fapp{\app{t}{s}}{r}$,
        \item Abstraction will have higher priority than application. So $\lambda u \lambdaapp \app{t}{s}$ is $\lambda u \lambdaapp (\app{t}{s})$.
    \end{itemize}
\end{remark}

\begin{definition} \label{definition:closedterms}
    The free variables of a term $t$, $\ff{FV}(t)$, are defined inductively as follows:
    \begin{itemize}
        \item If $u$ is a variable, then $\ff{FV}(u) = \{ u\}$,
        \item $\ff{FV}(\app{t}{s}) = \ff{FV}(t) \cup \ff{FV}(s)$,
        \item $\ff{FV}(\lambda u \lambdaapp t) = \ff{FV}(t) \setminus \{ u \}$.
    \end{itemize}
    A term $t$ is said to be \emph{closed} if $\ff{FV}(t) = \emptyset$ and \emph{open} otherwise. 
\end{definition}

\noindent Two terms are said to be $\alpha$\emph{-equivalent}, written $t \equiv_\alpha s$, if they are equal up-to a change in variables. For example, $\lambda u \lambdaapp \app{u}{w} \equiv_{\alpha} \lambda v \lambdaapp \app{v}{w}$, however $\lambda u \lambdaapp \app{u}{w}$ it is not equivalent to $\lambda w \lambdaapp \app{w}{w}$ because $w$ occurs freely in the first expression but becomes bound in the second one. We will consider terms to be equal up-to $\alpha$-equivalence and freely change between such equivalences when it aids readability. 

The main way we shall compare two different terms is through $\beta$\emph{-reduction}:
\[
(\lambda u \lambdaapp t)s \rightarrow_{\beta} t[u \coloneqq s],
\]
where $t[u \coloneqq s]$ is the substitution of every free occurrence of $u$ in the term $t$ with the term $s$.

We end this section by noting a few $\lambda$-terms that will be used frequently in this work: the \emph{identity term} and the \emph{Church numerals}.

\begin{definition} \label{ChurchNumerals} \,
    \begin{itemize}
        \item $\identity$ denotes the $\lambda$-term $\lambda u \lambdaapp u$,
        \item $\underline{0}$ denotes the $\lambda$-term $\lambda u \lambdaapp \lambda v \lambdaapp v$,
        \item For $n \in \omega$, $\underline{n+1}$ denotes the $\lambda$-term $\lambda u \lambdaapp \lambda v \lambdaapp \twoapp{\app{\underline{n}}{u}}{\app{u}{v}}$. In particular, $\underline{1}$ denotes the $\lambda$-term $\lambda u \lambdaapp \lambda v \lambdaapp \app{u}{v}$.
    \end{itemize}
\end{definition}

\begin{remark} \label{ChruchNumeralSuccessor}
Let $\rlzfont{s} = \lambda n \lambdaapp \lambda u \lambdaapp \lambda v \lambdaapp \twoapp{\app{n}{u}}{\app{u}{v}}$. Then, for every $n \in \omega$ we have $\app{\rlzfont{s}}{\underline{n}} \rightarrow_\beta \underline{n+1}$.
\end{remark}

\section{Introduction to Realizability} \label{section:realizabilityintro}

Bringing all of the threads together, we now discuss the intuition behind the realizability models before formally constructing them in the next sections. Since Kleene's original method results in intuitionistic models and we wish to produce classical models, the idea is to combine realizability with a double negation translation. Therefore, instead of working with realizers providing evidence for the ``truth'' of an assertion we will want to provide evidence that an assertion is not false.

Our algebra will consist of two sets: the \emph{terms}, $\Lambda$, and the \emph{stacks}, $\Pi$. There will also be a special subset of the terms, $\mathcal{R} \subseteq \Lambda$, called the \emph{realizers}, which will be used to indicate the formulas that are true in our resulting model.

To each formula, $\varphi$ we will then assign two values: its ``\emph{truth}'' value, $\verity{\varphi} \subseteq \Lambda$, which can be seen as providing evidence for the formula and its ``\emph{falsity}'' value, $\falsity{\varphi} \subseteq \Pi$, which can be seen as providing evidence against the formula.

Finally, a \emph{process} will consist of a pair $( t, \pi )$ where $t \in \Lambda$ and $\pi \in \Pi$. We will specify some set of processes $\Perp \subseteq \Lambda \times \Pi$, called the \emph{pole}, which can be thought of as a specification of when a truth value is incompatible with a falsity value. \\

\noindent Now, to say that $t$ provides evidence that a statement is not false, or alternative $t$ \emph{realizes} a statement, is to say that $t$ is incompatible with anything that falsifies the statement. In other words, $t \in \verity{\varphi}$ if and only if $\forall \pi \in \falsity{\varphi}$, $( t, \pi ) \in \Perp$. This leads to the idea that a formula $\varphi$ is ``\emph{true}'' if there is nothing falsifying it, i.e. $\falsity{\varphi} = \emptyset$ and ``\emph{false}'' if everything falsifies it, i.e. $\falsity{\varphi} = \Pi$. Thus if we let $\top$ and $\perp$ symbolise true and false this leads to the definitions $\falsity{\top} = \emptyset$ and $\falsity{\perp} = \Pi$.

We next consider what it means to realize an implication. Firstly, $t \in \verity{\varphi \rightarrow \psi}$ if for any $\pi \in \falsity{\varphi \rightarrow \psi}$, $( t, \pi ) \in \Perp$. So what does $\pi$ look like? Well evidence against $\varphi \rightarrow \psi$ would be a combination of evidence for $\varphi$, say $s \in \verity{\varphi}$, and evidence against $\psi$, say $\sigma \in \falsity{\psi}$. This is written $\pi = s \stackapp \sigma$, which leads to the following ``definition'',
\[
\falsity{\varphi \rightarrow \psi} = \{ s \stackapp \sigma \divline s \in \verity{\varphi}, \sigma \in \falsity{\psi} \}.
\]
This leads to the conclusion that it may in general be very difficult for an implication to be ``\emph{true}'' in our realizability model. This is because we have no control over what $t$ is. As an example, consider the sentence $\perp \rightarrow \perp$, i.e. the sentence false implies false (which should be ``true'' in the realizability model). Now suppose we had picked $t$ such that for any $\pi \in \Pi$, $( t, \pi ) \in \Perp$.\footnote{This is indeed possible in any realizability models whose pole in non-empty. For example, using the formalised notation we will introduce in the next section, suppose that $\Perp \neq \emptyset$ and fix $t_0 \star \pi_0 \in \Perp$. Then if $s \coloneqq \app{\saverlz{t_0 \stackapp \pi_0}}{\identity}$, for any $\pi \in \Perp$ we can prove that $s \star \pi \succ \saverlz{t_0 \stackapp \pi_0} \star \identity \stackapp \pi \succ \identity \star t_0 \stackapp \pi_0 \succ t_0 \star \pi_0 \in \Perp$.} Then, in particular, for any $\pi \in \falsity{\perp} = \Pi$, $( t, \pi ) \in \Perp$ so $t \in \verity{\perp}$. Thus $t \stackapp \pi \in \falsity{\perp \rightarrow \perp}$ for any $\pi \in \Pi$ and so the implication is not fully ``true'' in the model. Instead we will need to restrict our witnesses to the truth of a formula to the ``\emph{well-behaved}'' terms, which is the set of \emph{realizers}, $\mathcal{R}$. \\

\noindent The other logical symbol we consider is the universal quantifier. Since all other logical symbols can be deduced from $\rightarrow$ and $\forall$ (alongside $\top$, $\perp$ and the atomic formulas) it will suffice to only have explicit realizers for these two things. Again to do this we consider the falsity value of a universal quantifier, $\falsity{\forall x \varphi(x)}$. Evidence against $\forall x \varphi(x)$ would be some witness $a$ and evidence against $\varphi(a)$. This leads us to the following ``definition'' where the union should be taken over the ``\emph{universe of discourse}'' which we shall define later
\[
\falsity{\forall x \varphi(x)} = \bigcup \falsity{\varphi(a)}.
\]

\noindent It only remains to discuss the ``universe of discourse'' and atomic formulas, and for this we will follow McCarty's construction of realizability structures. Recall that we defined a hierarchy $\tf{V}(Kl)_\alpha$ recursively using $\tf{V}(Kl)_{\alpha + 1} = \mathcal{P}(\omega \times \tf{V}(Kl)_\alpha)$ where $\omega$ was the set of realizers. Since in our construction the falsity values are more fundamental than the truth values we will instead use $\Pi$ and take $\tf{N}_{\alpha + 1} \coloneqq \mathcal{P}(\tf{N}_\alpha \times \Pi)$\footnote{we note the unfortunate feature that the ordering of the pairs switches between the two different notions of realizability, so for Kleene's method the realizers come first and for Krivine's method the stacks come second} and also take unions at limit stages. Then the universe of discourse will be $\tf{N} \coloneqq \bigcup_{\alpha \in \tf{Ord}} \tf{N}_\alpha$. It will be shown, using a combination of \Cref{{theorem:rlzmodelModelsZFepsilon},{theorem:ZFepsilonToZF}}, that we can find binary relations $\in$ and $\simeq$ for which $(\tf{N}, \in, \simeq) \models \tf{ZF}$.

We are now left with the most technical task, which is to define the atomic formulas. In the double negation translation, one worked in a model of $\tf{ZF} \setminus \ff{Ext.}$ which contained the original elementhood relation $\in$ and used this to simulate a new elementhood relation $\in^\star$ which was compatible with extensionality. An important thing to note about this simulated relation is that it extends the real relation in that if $a \in b$ then $a \in^\star b$. We will do the same here, however for notational reasons the ``original'' non-extensional relation will be denoted by $\rlzin$ while the ``simulated'' relation will be denoted by $\in$. Moreover, since we are working primarily with negations it will turn out that it is much simpler to realize $\notrlzin$ and $\not\in$. 

For $\rlzin$, evidence that a set $b$ is not in a set $a$ is simply some $\pi \in \Pi$ for which $(b, \pi) \in a$. Therefore, $\falsity{b \notrlzin a} = \{ \pi \in \Pi \divline (b, \pi) \in a \}$. However, to see why this is not an extensional relation, let $\theta$ be any fixed realizer and consider the names $1_A = \{ (\emptyset, \pi) \divline \pi \in \Pi \}$ and $1_B = \{ (\emptyset, \theta \stackapp \pi) \divline \pi \in \Pi \}$. It is possible to realize that $0 \rlzin 1_A$ and $0 \rlzin 1_B$ and $0$ is the only set in either $1_A$ or $1_B$. However, these two sets are not equal (this is similar to how, in forcing, we can construct two names for the same object). In fact, it will turn out that the realizability model does not satisfy the statement $\forall x (1_A \notrlzin x \leftrightarrow 1_B \notrlzin x)$. Instead we shall use $\rlzin$ to simulate a relation which is compatible with extensionality. \\

\noindent The extensional relation $\in$ will be defined simultaneously alongside the relation $\subseteq$ which will give the extensional subset relation. Extensional equality, $\simeq$, can then be defined as a conjunction of $a \subseteq b$ and $b \subseteq a$. We observe here that the constructions of these atomic formulas are influenced by their intuitionistic counterparts that were defined in \Cref{section:IntuitionisticRealizability}. There we defined $n \Vdash a \in b$ whenever there was some $c$ such that $( \mathbf{1^{st}}(n), c ) \in b$ and $\mathbf{2^{nd}}(n) \Vdash a = c$ and $n \Vdash a \subseteq b$ whenever it was the case that for any realizer $m \in \omega$ and set $c$, if $( m, c ) \in a$ then $\{\mathbf{1^{st}}(n)\}(m) \Vdash c \in b$. Mirroring this, $a \not\in b$ can be written as $\forall c (a \simeq c \rightarrow c \notrlzin b)$. Thus, using the falsity values for implications, evidence against the assertion $a \not\in b$ would be to find some name $c$, stack $\pi \in \Pi$ and terms $t, t' \in \Lambda$ for which $(c, \pi) \in b$, $t \Vdash c \subseteq a$ and $t' \Vdash a \subseteq c$. This leads to the definition,
\[
\falsity{a \not\in b} = \bigcup_{c \in \ff{dom}(b)} \{ t \stackapp t' \stackapp \pi \divline (c, \pi) \in b \land t \Vdash c \subseteq a \land t' \Vdash a \subseteq c \}. 
\]
Similarly, evidence against the assertion $a \subseteq b$ would be some name $c$, stack $\pi \in \Pi$ and realizer $s$ for which $(c, \pi) \in a$ and $t \Vdash c \not\in b$. This leads to the final definition,
\[
\falsity{a \subseteq b} = \bigcup_{c \in \ff{dom}(a)} \{ (c, \pi) \in a \land t \Vdash c \not\in b \}. 
\]

\noindent We end this section by remarking on how much of the Brouwer-Heyting-Kolmogorov interpretation we should expect to get in a model which satisfies classical logic. Firstly, both $\forall$ and $\rightarrow$ will follow the interpretation. For example, it will be shown in \Cref{theorem:ImplicationandApplication} that if $t \Vdash \varphi \rightarrow \psi$ and $s \Vdash \varphi$ then $\app{t}{s} \Vdash \psi$ where $\app{t}{s} \in \Lambda$ is a term that represents the \emph{application} of $t$ and $s$. Secondly, while we will not do it in these notes, it is possible to explicitly realize conjunctions in such a way that a realizer for $\varphi \land \psi$ is a term of two parts, the first of which realizes $\varphi$ while the second realizes $\psi$. However, this is as far as we should expect the interpretation to go. Being able to realize a disjunction gives a theory the \emph{disjunction property}. This is a very useful tool in intuitionistic mathematics because it tells us that if our theory can deduce $\varphi \lor \psi$ then either it can deduce $\varphi$ or it can deduce $\psi$. However, in classical mathematics this is not possible. For example, by Excluded Middle we can always prove $\ff{Con}(T) \lor \neg \ff{Con}(T)$ for any theory $T$. On the other hand, we cannot in general say which of the two is actually true. 

In a similar way, the specification for existential statements leads to the \emph{existence property} which essentially says that if $\exists x \varphi(x)$ is deducible in a theory then we can explicitly find a witness $a$ for which $\varphi(a)$ holds. As before, in classical logic this cannot be possible. For example, let $A$ be equal to $\{1\}$ is $\varphi$ holds and $\{0\}$ otherwise. Then it is clear that we have $\exists x (x \in A)$ but to explicitly determine the value of $x$ would be to determine $\varphi \lor \neg \varphi$. 

\section{Realizability Algebras} \label{section:RealizabilityAlgebras}

In order to construct a realizability algebra that can handle classical logic it will be necessary to extend the $\lambda$-calculus to also include what is known as $\lambda_c$-terms. This is an expanded calculus which also includes the term \cc, which is known as \emph{call-with-current-continuation}. We refer the reader to \cite{FelleisenFriedmanKohlbeckerDuba1987} for a presentation of this calculus. In order to state realizability algebras in full generality, we will also include two sets; the set of \emph{special instructions} $A$, and the set of \emph{stack bottoms} $B$, both of which may be empty.

We will define three sets; $\Lambda^{\ff{open}}_{(A, B)}$, which consists of all (possibly open) $\lambda_c$-terms; $\Lambda_{(A, B)}$, which consists of all \emph{closed} $\lambda_c$-terms; and $\Pi_{(A, B)}$, which consists of all stacks. Recall that a closed term was defined in \Cref{definition:closedterms} to be a term in which every variable occurs within the scope of a $\lambda$-abstraction and we restrict our attention to the closed terms to have a well-defined notion of computation. To ease notation, we shall refer to elements of $\Lambda_{(A, B)}$ simply as \emph{terms} and therefore assume that all terms are closed unless explicitly stated. 

\begin{definition}
    The classes $\Lambda^{\ff{open}}_{(A, B)}$, $\Lambda_{(A, B)}$ and $\Pi_{(A, B)}$ are defined simultaneously by induction as the smallest classes closed under the following rules:
    
    \noindent $\Lambda^{\ff{open}}_{(A, B)}$:
    \begin{itemize}
        \item Any variable is in $\Lambda^{\ff{open}}_{(A, B)}$,
        \item If $t$ and $s$ are in $\Lambda^{\ff{open}}_{(A, B)}$ then so is $\app{t}{s}$, \hfill (\emph{application})
        \item If $u$ is a variable and $t$ is in $\Lambda^{\ff{open}}_{(A, B)}$ then $\lambda u \lambdaapp t$ is in $\Lambda^{\ff{open}}_{(A, B)}$, \hfill ($\lambda$\emph{-abstraction}),
        \item $\cc$ is in $\Lambda^{\ff{open}}_{(A, B)}$ \hfill (\emph{call-with-current-continuation}),
        \item If $\pi$ is in $\Pi_{(A, B)}$ then $\saverlz{\pi}$ is in $\Lambda^{\ff{open}}_{(A, B)}$, \hfill (\emph{continuation constant}),
        \item $t_a$ is in $\Lambda^{\ff{open}}_{(A, B)}$ for $a < A$. \hfill (\emph{special instructions})
    \end{itemize} 
    \noindent $\Lambda_{(A, B)}$:
    \begin{itemize}
        \item Any element of $\Lambda^{\ff{open}}_{(A, B)}$ which is \emph{closed} is in $\Lambda_{(A, B)}$.
    \end{itemize}
    \noindent $\Pi_{(A, B)}$:
    \begin{itemize}
        \item If $t$ is in $\Lambda_{(A, B)}$\footnote{so, in particular, $t$ is \emph{closed}} and $\pi$ is in $\Pi_{(A, B)}$ then $t \stackapp \pi$ is in $\Pi_{(A, B)}$, \hfill (\emph{push}),
        \item $\omega_b$ is in $\Pi_{(A, B)}$ for $b \in B$. \hfill (\emph{stack bottoms})
    \end{itemize}
\end{definition}

\begin{definition}
    A term $t$ is called a \emph{realizer} if it contains no occurrence of a continuation constant. We denote by $\rlzset_{(A, B)}$ the collection of all realizers.
\end{definition}

\begin{definition}
    A \emph{process} is an ordered pair $(t, \pi)$ where $t \in \Lambda_{(A, B)}$ and $\pi \in \Pi_{(A, B)}$ and will be denoted $t \star \pi$. We denote by $\Lambda_{(A, B)} \star \Pi_{(A, B)}$ the set of all processes. 
    
    Given such a process, $t$ will be called the \emph{head} and $\pi$ the \emph{tail} of the process.
\end{definition}

\noindent We now give a pre-order relation, $\succ$, on the collection of process which is defined as the smallest ordering satisfying four simple rules. Often the ordering is further extended by additional specifications, possibly using the special instructions or stack bottoms. 

\begin{definition}
A \emph{one-step evaluation} of the set of processes $\Lambda_{(A, B)} \star \Pi_{(A, B)}$ is a relation $\succ_1$ which satisfies the four rules
\begin{align*}
    \app{t}{s} \star \pi & \quad \succ_1 \quad t \star s \stackapp \pi && \text{(push)}, \\
    \lambda u \lambdaapp t \star s \stackapp \pi & \quad \succ_1 \quad t[u \coloneqq s] \star \pi && \text{(grab)}, \\
    \cc \star t \stackapp \pi & \quad \succ_1 \quad t \star \saverlz{\pi} \stackapp \pi && \text{(save)},\\
    \saverlz{\sigma} \star t \stackapp \pi & \quad \succ_1 \quad t \star \sigma && \text{(restore)}.
\end{align*}
The relation $\succ$ denotes the reflexive-transitive closure of $\succ_1$.
\end{definition}

\begin{remark}
    An example of an additional instruction that is regularly added to a countable realizability algebra is the instruction \emph{quote}, $\rlzfont{q}$, which allows one to compute the ``\emph{code}'' of a given stack. The instruction comes with the additional rule 
    \[
    \rlzfont{q} \star t \stackapp s \stackapp \pi \succ_1 t \star \underline{\eta}_s \stackapp \pi,
    \]
    where $s \mapsto \eta_s$ is some fixed enumeration of $\Lambda$ in order-type $\omega$.  
\end{remark}

\noindent Finally we fix a final segment of the processes, known as the \emph{pole} and denoted $\Perp$. The pole will be used to determine which terms realize a given formula. Namely, given a formula $\varphi$ we shall define its falsity value $\falsity{\varphi} \subseteq \Pi_{(A, B)}$ and then say that a term $t$ realizes $\varphi$ if and only if for any $\pi \in \falsity{\varphi}$, $t \star \pi \in \Perp$.

\begin{definition}
    A \emph{pole} of the relation $\succ$ over the set of processes $\Lambda_{(A, B)} \star \Pi_{(A, B)}$ is a set $\Perp \subseteq \Lambda_{(A, B)} \star \Pi_{(A, B)}$ such that
    \[
    ((s \star \sigma \succ t \star \pi) \; \land \; (t \star \pi \in \Perp)) \; \rightarrow \; s \star \sigma \in \Perp.
    \]
\end{definition}

\begin{definition}
    A \emph{realizability algebra} is a tuple $\mathcal{A} \coloneqq (\Lambda_{(A, B)}, \Pi_{(A, B)}, \prec, \Perp)$.
\end{definition}

\begin{notation}
    Brackets around terms will mainly be used to assist in comprehension and application will be a left associative operation. 

    To see an example, we will use the intuition behind realizability from \Cref{section:realizabilityintro}. So, suppose that $t_i \Vdash \varphi_i$ and $\pi \in \Pi$. Then $t_1 \stackapp \pi$ would be a falsifier for the statement $\varphi_1 \rightarrow \perp$. Next, a realizer for the statement $\varphi_2 \rightarrow (\varphi_1 \rightarrow \perp)$ would be the process $t_2 \star t_1 \stackapp \pi$ and $\app{t_2}{t_1} \star \pi \succ t_2 \star t_1 \stackapp \pi$. Taking another step, a realizer for $\varphi_3 \rightarrow (\varphi_2 \rightarrow (\varphi_1 \rightarrow \perp))$ would correspond to $t_3 \star t_2 \stackapp t_1 \stackapp \pi$ and $\fapp{\app{t_3}{t_2}}{t_1} \star \pi \succ \app{t_3}{t_2} \star t_1 \stackapp \pi \succ t_3 \star t_2 \stackapp t_1 \stackapp \pi$.
\end{notation}

\begin{remark}
    When it is either clear from the context or not important to the discussion, we will drop mention of $A$ and $B$ and instead refer to the realizability algebra as $(\Lambda, \Pi, \prec, \Perp)$.
\end{remark}

\section{The Theory \texorpdfstring{$\ZFepsilon$}{ZFepsilon}} \label{section:ZFepsilon}

Throughout these notes, we will work in first-order logic \emph{without equality}. In addition, the only primitive logical symbols we will include in are language are $\rightarrow, \top, \perp$ and $\forall$.

\begin{definition}    
    Let $\mathcal{L}_{\in}$ be the first-order language with non-logical symbols the binary relations $\in$ and $\simeq$ (which will be interpreted as the usual equality symbol, $=$). 

    Let $\mathcal{L}_{\rlzin}$ be the first-order language with non-logical symbols the binary relations $\notrlzin$, $\not\in$ and $\subseteq$. By abusing notation, this will also be written as the language with non-logical symbols the binary relations $\rlzin$, $\in$ and $\subseteq$ for readability.

    Let $Fml_\in$ and $Fml_{\rlzin}$ denote the collection of all $\mathcal{L}_\in$ and $\mathcal{L}_{\rlzin}$ formulas respectively.
\end{definition}

\noindent Using the three binary relation symbols, $\notrlzin$, $\not\in$ and $\subseteq$, in $\mathcal{L}_{\rlzin}$, we then have the following abbreviations: \\

\begin{tabular}{@{$\bullet$ \;}lll}
    $a \in b$ & \quad should be read as \quad & $a \not\in b \rightarrow \perp$, \\[2pt]
    $a \rlzin b$ & \quad should be read as \quad & $a \notrlzin b \rightarrow \perp$, \\[2pt]
    $\varphi \land \psi$ & \quad should be read as \quad & $(\varphi \rightarrow (\psi \rightarrow \perp)) \rightarrow \perp$, \\[2pt]
    $\varphi \lor \psi$ & \quad should be read as \quad & $(\varphi \rightarrow \perp) \rightarrow ((\psi \rightarrow \perp) \rightarrow \perp)$, \\[2pt]
    $\forall x \rlzin a \, \varphi(x)$ & \quad should be read as \quad & $\forall x (\neg \varphi(x) \rightarrow x \notrlzin a)$, \\[2pt]
    $\exists x \rlzin a \, \varphi(x)$ & \quad should be read as \quad & $\neg \forall x (\varphi(x) \rightarrow x \rlzin a)$, \\[2pt]
    $\exists x \varphi(x)$ & \quad should be read as \quad & $\forall x (\varphi(x) \rightarrow \perp) \rightarrow \perp$, \\[2pt]
    $a \simeq b$ & \quad should be read as \quad & $(a \subseteq b) \land (b \subseteq a)$,\\[2pt]
    $a \simeq b \rightarrow \varphi$ & \quad should be read as \quad & $(a \subseteq b) \rightarrow ((b \subseteq a) \rightarrow \varphi)$.
\end{tabular}

\begin{remark}
    Technically there is a clash in notation for $a \simeq b \rightarrow \varphi$ because $(a \subseteq b) \land (b \subseteq a) \rightarrow \varphi$ is not the same as $(a \subseteq b) \rightarrow ((b \subseteq a) \rightarrow \varphi)$. However, it can easily be seen that $(\varphi \rightarrow (\psi \rightarrow \theta))$ and $(((\varphi \rightarrow (\psi \rightarrow \perp)) \rightarrow \perp) \rightarrow \theta)$ are logically equivalent statements. Using the facts that $\ZFepsilon$ is built using classical logic and that the realizability interpretation will be closed under natural deductions (\Cref{theorem:adequacy}) will allow us to free move between logically equivalent statements and therefore we can take the abbreviation to be the ``simplest'' logically equivalent form of what we want.
\end{remark}

\noindent We now give the axioms of $\ZFepsilon$. The first axiom defines $\subseteq$ and gives the relationship between $\rlzin$ and $\in$, which intuitively says that $\in$ is obtained as the ``extensional collapse'' of $\rlzin$. The other axioms essentially state that every axiom of Zermelo-Fraenkel set theory holds with respect to the relation $\rlzin$. In the rest of these notes, we will differentiate between the standard definition of a concept and the version using $\rlzin$ using the prefixes $\in$ and $\rlzin$-. However, when it is clear from the context which theory we are working in, these prefixes will be dropped.

\pagebreak
\begin{enumerate}
    \item $\rlzin$-Extensionality axioms.
    \[
    \forall x \forall y \Big( x \in y \leftrightarrow \exists z \rlzin y (x \simeq z) \Big); \quad \forall x \forall y \Big( x \subseteq y \leftrightarrow \forall z \rlzin x (z \in y) \Big).
    \]
    \item $\rlzin$-Induction Scheme.
    \[
    \forall u_1 \dots \forall u_n \Big( \forall x \big( \forall y \rlzin x \, \varphi(y, u_1, \dots, u_n) \rightarrow \varphi(x, u_1, \dots, u_n) \big) \rightarrow \forall x \varphi(x, u_1, \dots, u_n) \Big)
    \]
    for every formula $\varphi(u, u_1, \dots, u_n)$ in $Fml_{\rlzin}$.
    \item $\rlzin$-Separation Scheme.
    \[
    \forall u_1 \dots \forall u_n \forall a \exists b \forall x ( x \rlzin b \leftrightarrow (x \rlzin a \land \varphi(x, u_1, \dots, u_n)))
    \]
    for every formula $\varphi(u, u_1, \dots, u_n)$ in $Fml_{\rlzin}$.
    \item Axiom of $\rlzin$-Pairing.
    \[ 
    \forall a \forall b \exists c ( a \rlzin c \land b \rlzin c).
    \] 
    \item Axiom of $\rlzin$-Unions.
    \[
    \forall a \exists b \, \forall x \rlzin a \, \forall y \rlzin x (y \rlzin b).
    \]
    \item Axiom of $\rlzin$-Weak Power Sets.
    \[
    \forall a \exists b \forall x \, \exists y \rlzin b \, \forall z (z \rlzin y \leftrightarrow (z \rlzin a \land z \rlzin x)).
    \]
    \item $\rlzin$-Collection Scheme.
    \[
    \forall u_1 \dots \forall u_n \forall a \exists b \, \forall x \rlzin a \Big( \exists y \varphi(x, y, u_1, \dots, u_n) \rightarrow \exists y \rlzin b \, \varphi(x, y, u_1, \dots, u_n) \Big)
    \]
    for every formula $\varphi(u, v, u_1, \dots, u_n)$ in $Fml_{\rlzin}$.
    \item $\rlzin$-Infinity Axiom.
    \[
    \forall a \, \exists b \, \Big( a \rlzin b \land \forall x (x \rlzin b \rightarrow \exists y (y \rlzin b \land x \rlzin y) \Big).
    \]
\end{enumerate}

\begin{remark}
    Note that in the formulation of the axioms of $\ZFepsilon$ we are taking the Axiom of Weak Power Sets instead of the standard Power Set Axiom. The same is also done by Friedman in his proof that \tf{ZF} is equiconsistent with \tf{IZF}, \cite{Friedman1973}. The reason for this is that when one does the translation to set theories without extensionality, this is a much easier axiom to prove, primarily because proving that one set is a subset of another is normally very difficult to see. In fact, it is unclear if the $\rlzin$ version of the Power Set Axiom can in general be realized. 

    However, in a theory with extensionality it is very easy to see that the above formulation is equivalent to the standard axiomatisation of Power Set. The Axiom of Weak Power Sets essentially says that for any set $a$ there is a set $b$ such that for all $x$, $x \cap a \in b$. In particular, if $x \subseteq a$ then $x = x \cap a \in b$. On the other hand, if the Power Set Axiom holds, then $\mathcal{P}(a)$ trivially satisfies this condition.
\end{remark}

\section{Going from \texorpdfstring{$\ZFepsilon$}{ZFepsilon} to \texorpdfstring{\tf{ZF}}{ZF} } \label{section:ZFepsilonToZF}

As proven in \cite{Krivine2001} and \cite{Krivine2012}, every axiom of \tf{ZF} is provable from the axioms of $\ZFepsilon$. By reproducing this proof, in this section we shall make something explicit which has been implicitly assumed in the literature. Namely, given a realizability model $\rlzmodel = (\rlzstr, \rlzin, \in, \subseteq)$, we obtain a model of \tf{ZF} by ``dropping'' $\rlzin$.

\begin{theorem} \label{theorem:ZFepsilonToZF}
    Suppose that $\rlzmodel = (\rlzstr, \rlzin, \in, \subseteq)$ is a model of $\ZFepsilon$. Then $(\rlzstr, \in, \simeq) \models \tf{ZF}$, where $\simeq$ interprets equality of sets.
\end{theorem}

\begin{notation}
    To ease notation, whenever $\rlzmodel = (\rlzstr, \rlzin, \in, \subseteq)$ is a model of $\ZFepsilon$, we will denote by $\rlzmodel_\in$ the reduced structure $(\rlzstr, \in, \simeq)$.
\end{notation}

\noindent Since $\ZFepsilon$ is a classical theory which is closed under natural deductions, the proof will consist of two parts; first we show that $\rlzmodel_\in$ satisfies every axiom of \tf{ZF}, then we show that $\rlzmodel$ and $\rlzmodel_\in$ prove the same sentences in the langauge $\mathcal{L}_\in$. Before proving the theorem, we will give a list of some of the many simple statements provable in $\ZFepsilon$. All of the proofs can be found in Section 3 of \cite{Krivine2012}. Note that, by \Cref{ZFepsilonProperty:SubsetDefinition}, we do not need to consider $\subseteq$ as a primitive symbol in the structure $\rlzmodel_\in$.

\begin{theorem} \label{theorem:ZFepsilonStatements} \, 
    \begin{thmlist}
        \item \label{ZFepsilonProperty:SubsetIdentity} $\ZFepsilon \vdash \forall a (a \subseteq a)$;
        \item \label{ZFepsilonProperty:RlzinImpliesIn} $\ZFepsilon \vdash \forall a \forall x (x \rlzin a \rightarrow x \in a)$;
        \item \label{ZFepsilonProperty:SimeqIdentity} $\ZFepsilon \vdash \forall a (a \simeq a)$;
        \item \label{ZFepsilonProperty:SimeqReflective} $\ZFepsilon \vdash \forall a \forall b (a \simeq b \rightarrow b \simeq a)$;
        \item \label{ZFepsilonProperty:ExtensionalityA} $\ZFepsilon \vdash \forall a \forall x \forall y (x \simeq a \land a \in y \rightarrow x \in y)$;
        \item \label{ZFepsilonProperty:SimeqTransitive} $\ZFepsilon \vdash \forall a \forall b \forall c (a \simeq b \land a \simeq c \rightarrow b \simeq c)$;
        \item \label{ZFepsilonProperty:ExtensionalityB} $\ZFepsilon \vdash \forall a \forall x \forall y (a \subseteq y \land x \in a \rightarrow x \in y)$;
        \item \label{ZFepsilonProperty:SubsetDefinition} $\ZFepsilon \vdash \forall a \forall b \big( a \subseteq b \leftrightarrow \forall x (x \in a \rightarrow x \in b) \big)$.
    \end{thmlist}
\end{theorem}

\begin{proof}
    $(i)$. This is proved by $\rlzin$-Induction. So suppose that the claim holds for all $x \rlzin a$ and fix $x \rlzin a$. Then, since $x \simeq x$, we have that $\exists y \rlzin a (y \simeq x)$ which, by the $\rlzin$-Extensionality is equivalent to saying that $x \in a$. 
    
    $(ii)$ and $(iii)$ then follow immediately from $(i)$, while $(iv)$ is true by definition. \\

    \noindent Next, let $\varphi_1(a) \equiv \forall x \forall y (x \simeq a \land a \in y \rightarrow x \in y)$ and $\varphi_2(a) \equiv \forall x \forall y (a \simeq x \land a \simeq y \rightarrow x \simeq y)$. We shall prove that $\forall b \rlzin a \, \varphi_1(b)$ implies $\varphi_2(a)$ and $\varphi_1(a)$. Then $(v)$ and $(vi)$ will both follow by $\rlzin$-Induction. So suppose that $\forall b \rlzin a \, \varphi(b)$.

    To prove $\varphi_2(a)$, fix $x$ and $y$ and suppose that $a \simeq x$ and $a \simeq y$. We shall prove that $x \subseteq y$, with the argument for $y \subseteq x$ being symmetric. To do this, fix $z \rlzin x$. Since $x \subseteq a$ we have that $z \in a$ which means that we can fix some $b \rlzin a$ for which $z \simeq b$. Next, since $a \subseteq y$, $b \in y$. Now, by the inductive hypothesis, $\varphi_1(b)$ holds, from which it follows that $z \in y$. 

    To prove $\varphi_1(a)$, fix $x$ and $y$ and suppose that $x \simeq a$ and $a \in y$. Next, take $z \rlzin y$ such that $a \simeq z$. By $\varphi_3(a)$, $x \simeq z$ from which it follows that $x \in y$. \\

    \noindent $(vii)$. Fix $x$ and $y$ and suppose that $a \subseteq y$ and $x \in a$. Since $x \in a$, we can fix $b \rlzin a$ such that $b \simeq x$. Then, since $a \subseteq y$ and $\varphi_1(b)$ holds by $(v)$, $b \in y$. Thus $x \in y$, as desired.

    $(viii)$. For the first direction, suppose that $\forall x (x \in a \rightarrow x \in b)$ and fix $x \rlzin a$. Then, by $(ii)$, $x \in a$ and therefore $x \in b$ by the assumption. Thus $a \subseteq b$. For the reverse direction, suppose that $a \subseteq b$ and $x \in a$. Then we can fix $y \rlzin a$ such that $y \simeq x$. Now, since $a \subseteq b$, $y \in b$ and therefore, by $(v)$, we obtain that $x \in b$.

\end{proof}

\begin{theorem} \label{theorem:EquivalentSetsProveSameThings}
    Let $\varphi(u)$ be a formula in the language of $\{ \in, \simeq \}$. If $a \simeq b$, then $\ZFepsilon \vdash \varphi(a) \leftrightarrow \varphi(b)$.
\end{theorem}

\begin{proof}
    Formally, this is proven by induction on the complexity of formulas. However, it is clear that the only non-trivial case is for atomic formulas, as the proof for connectives and quantifiers is immediate. For the atomic case there are five formulas to check: $a \in x$; $x \in a$; $a \simeq x$; $x \simeq a$ and $a \simeq a$. The proof of the first follows from \ref{ZFepsilonProperty:ExtensionalityA}, the second from \ref{ZFepsilonProperty:ExtensionalityB} and the others from a combination of \ref{ZFepsilonProperty:SimeqIdentity}, \ref{ZFepsilonProperty:SimeqReflective} and \ref{ZFepsilonProperty:SimeqTransitive}.
\end{proof}

\noindent We are now in a position to prove that every axiom of $\tf{ZF}$ is derivable from $\ZFepsilon$.

\medskip

\noindent \textbf{Extensionality} $(\forall a, b \, \forall x (x \in a \leftrightarrow x \in b) \rightarrow a \simeq b)$

Fix $a$ and $b$ in $\rlzstr$ and suppose that $\forall x (x \in a \leftrightarrow x \in b)$. By Property \ref{ZFepsilonProperty:SubsetDefinition}, $a \subseteq b$ and $b \subseteq a$. Therefore, $a \simeq b$.

\bigskip

\noindent \textbf{Induction} $\forall \overrightarrow{u} \big( \forall x \big( \forall y \in x \, \varphi(y, \overrightarrow{u}) \rightarrow \varphi(x, \overrightarrow{u}) \big) \rightarrow \forall z \varphi(z, \overrightarrow{u}) \big)$

Let $\varphi(x, \overrightarrow{u})$ be a formula in the language of $\{ \in, \simeq \}$. For notational simplicity, we will drop mention of $\overrightarrow{u}$. Let $\psi(x)$ be the formula $\forall z (z \simeq x \rightarrow \varphi(z))$. Observe that for any $b$ in $\rlzstr$, we have $\ZFepsilon \vdash \varphi(b) \leftrightarrow \psi(b)$. To see this, suppose that $\varphi(b)$ holds and take $z \simeq b$. By \Cref{theorem:EquivalentSetsProveSameThings} we have that $\varphi(z)$ holds and thus $\psi(b)$ is true. On the other hand, if $\psi(b)$ holds then, since $b \simeq b$, $\varphi(b)$ must also hold.

Next, we shall prove that
\[
\forall x (\forall y \in x \, \varphi(y) \rightarrow \varphi(x)) \longrightarrow \forall v (\forall w \rlzin v \, \psi(w) \rightarrow \psi(v)).
\]
To do this, suppose that this premise holds, fix $v$ and suppose that $\forall w \rlzin v \, \psi(w)$. Take $y \in v$. Then we can find some $w \rlzin v$ such that $w \simeq y$. Since $\psi(w)$ holds, $\varphi(y)$ also holds. Thus, $\forall y \in v \, \varphi(y)$ so, by the assumed premise, $\varphi(v)$ holds which means that $\psi(v)$ holds. 

Now, by Induction in $\ZFepsilon$, we know that
\[
\forall v (\forall w \rlzin v \, \psi(w) \rightarrow \psi(v)) \rightarrow \forall z \psi(z)
\]
and therefore $\forall x (\forall y \in x \, \varphi(y) \rightarrow \varphi(x)) \rightarrow \forall z \psi(z) \rightarrow \forall z \varphi(z)$
as required.

\bigskip

\noindent \textbf{Separation} $(\forall \overrightarrow{u} \forall a \exists b \forall x (x \in b \leftrightarrow (x \in a \land \varphi(x)))$

Let $\varphi(x, \overrightarrow{u})$ be a formula in the language of $\{ \in, \simeq \}$ and let $a$ and $\overrightarrow{u}$ be in $\rlzstr$. By Separation in $\ZFepsilon$, there is some $b$ in $\rlzstr$ such that $\forall z (z \rlzin b \leftrightarrow (z \rlzin a \land \varphi(z, \overrightarrow{u}))$. Now, if $x \in b$ then we can fix some $z \rlzin b$ such that $x \simeq z$. Therefore, $z \rlzin a$ and $\varphi(z, \overrightarrow{u})$. Next, since $x \simeq z$ and $z \rlzin a$, Properties \ref{ZFepsilonProperty:RlzinImpliesIn} and \ref{ZFepsilonProperty:ExtensionalityA} give us that $z \in a$ and $x \in a$. Finally, by \Cref{theorem:EquivalentSetsProveSameThings}, $\varphi(x, \overrightarrow{u})$.

For the reverse direction, suppose that $\varphi(x, \overrightarrow{u})$ and $x \in a$. Then we can fix some $z \rlzin a$ such that $x \simeq z$. Again by \Cref{theorem:EquivalentSetsProveSameThings}, $\varphi(z, \overrightarrow{u})$, and thus $z \rlzin b$ from which we can deduce that $x \in b$.

\bigskip
\pagebreak
\noindent \textbf{Pairing} $(\forall a \forall b \exists c (a \in c \land b \in c))$

Fix $a$ and $b$ in $\rlzstr$. By Pairing in $\ZFepsilon$, there exists some $c$ in $\rlzstr$ such that $a \rlzin c$ and $b \rlzin c$. Therefore, by Property \ref{ZFepsilonProperty:RlzinImpliesIn}, $a \in c$ and $b \in c$.

\bigskip

\noindent \textbf{Unions} $(\forall a \exists b \forall x \in a \forall y \in x (y \in b))$

Fix $a$ in $\rlzstr$. By Unions in $\ZFepsilon$, there exists some $b$ in $\rlzstr$ such that $\forall x \rlzin a \forall y \rlzin x (y \rlzin b)$. Take $x \in a$ and $y \in x$. First, we can fix $u \rlzin a$ such that $x \simeq u$. Next, since $y \in x$, by Property \ref{ZFepsilonProperty:ExtensionalityB}, $y \in u$ and therefore we can fix some $v \rlzin u$ such that $y \simeq v$. By definition, $v \rlzin b$ so, by Property \ref{ZFepsilonProperty:RlzinImpliesIn}, $v \in b$ from which Property \ref{ZFepsilonProperty:ExtensionalityA} gives us that $y \in b$.

\bigskip

\noindent \textbf{Weak Power Set} $(\forall a \exists b \forall x \exists y \in b \forall z (z \in y \leftrightarrow (z \in a \land z \in x)))$

Fix $a$ in $\rlzstr$. By Weak Power Set in $\ZFepsilon$, there exists some $b$ in $\rlzstr$ witnessing this axiom. Given $x$ we define $x'$ by the specification
\[
\forall z (z \rlzin x' \leftrightarrow (z \rlzin a \land z \in x)),
\]
noting that such as $x'$ exists by the Separation Scheme in $\ZFepsilon$. By definition of $b$, we can find some $y \rlzin b$ such that $\forall z (z \rlzin y \leftrightarrow (z \rlzin a \land z \rlzin x'))$. Then, by construction, $\forall z (z \rlzin y \leftrightarrow (z \rlzin a \land z \in x))$. We shall conclude by showing that $\forall z (z \in y \leftrightarrow (z \in a \land z \in x))$.

Firstly, suppose that $z \in y$. Then we can fix some $u \rlzin y$ such that $z \simeq u$. By definition of $y$, this means that $u \rlzin a \land u \in x$, from which Property \ref{ZFepsilonProperty:RlzinImpliesIn} gives us that $u \in a$. Finally, two uses of Property \ref{ZFepsilonProperty:ExtensionalityA} gives us that $z \in a \land z \in x$.

For the reverse direction, suppose that $z \in a \land z \in x$. Since $z \in a$ we can fix some $u \rlzin a$ such that $u \simeq z$. By Property \ref{ZFepsilonProperty:ExtensionalityA}, $u \in x$ from which we can conclude that $u \rlzin y$. By the Extensionality axiom of $\ZFepsilon$ this means that $z \in y$.
\bigskip

\noindent \textbf{Collection} $(\forall a \exists b \forall x \in a (\exists y \varphi(x, y, \overrightarrow{u}) \rightarrow \exists y \in b \varphi(x, y, \overrightarrow{u})))$

Let $\varphi(x, \overrightarrow{u})$ be a formula in the language of $\{ \in, \simeq \}$ and let $a$ be in $\rlzstr$. For notational simplicity, we will drop mention of $\overrightarrow{u}$. By the Collection Scheme in $\ZFepsilon$ we can find some $b$ in $\rlzstr$ such that $\forall z \rlzin a (\exists y \varphi(z, y) \rightarrow \exists y \rlzin b \varphi(z, y))$.

Fix $x \in a$ and suppose that $\exists y \varphi(x, y)$. Then we can find some $z \rlzin a$ such that $z \simeq a$. Since $\varphi$ is a formula in the language of $\{ \in, \simeq \}$, by \Cref{theorem:EquivalentSetsProveSameThings}, $\exists y \varphi(z, y)$. Therefore, by definition of $b$, $\exists y \rlzin b \varphi(z, y)$. Hence, by Property \ref{ZFepsilonProperty:RlzinImpliesIn} and \Cref{theorem:EquivalentSetsProveSameThings}, $\exists y \in b \varphi(x, y)$.

\bigskip

\noindent \textbf{Infinity} $(\forall a \exists b (a \in b \land \forall x (x \in b \rightarrow \exists y (y \in b \land x \in y))))$

Fix $a$ in $\rlzstr$. By the axiom of Infinity in $\ZFepsilon$, there exists some $b$ in $\rlzstr$ such that $a \rlzin b$ and $\forall x (x \rlzin b \rightarrow \exists y (y \rlzin b \land x \rlzin y))$. Firstly, by Property \ref{ZFepsilonProperty:RlzinImpliesIn}, $a \in b$. So, suppose $x \in b$. Then we can find some $u \rlzin b$ such that $x \simeq u$. Then by the definition of $b$, there is some $y$ such that $y \rlzin b \land u \rlzin y$. Using Property \ref{ZFepsilonProperty:RlzinImpliesIn} again, we get that $y \in b \land u \in y$ and, finally, Property \ref{ZFepsilonProperty:ExtensionalityA} gives us that $y \in b \land x \in y$. 

\bigskip

\noindent To finish the argument, we show that any model $(\rlzstr, \rlzin, \in, \subseteq)$ of $\ZFepsilon$ is in fact a \emph{conservative} extension of the model $(\rlzstr, \in, \simeq)$. Namely, if the $\ZFepsilon$ model proves a statement in the language $\mathcal{L}_\in$ then the statement is already true in the restricted \tf{ZF} model.

\begin{theorem}
    Suppose that $\rlzmodel = (\rlzstr, \rlzin, \in, \subseteq)$ is a model of $\ZFepsilon$ and $\varphi(\overrightarrow{x})$ is a formula in $Fml_\in$. Then, for any $\overrightarrow{a} \subset \rlzstr$, $\rlzmodel \models \varphi(\overrightarrow{a}) \Longleftrightarrow (\rlzstr, \in, \simeq) \models \varphi(\overrightarrow{a})$.
\end{theorem}

\begin{proof}
    We prove the theorem by induction on the complexity of the class $Fml_\in$. For the atomic cases, $\in$ is a subclasses of $\rlzstr \times \rlzstr$ which is the same class for both the models $\rlzmodel$ and $\rlzmodel_\in$. Therefore, 
    \[
    \rlzstr \models a \in b \quad \Longleftrightarrow \quad ( a, b ) \in \, ``\in" \quad \Longleftrightarrow \quad \rlzstr_\in \models a \in b
    \]
    The same argument will also hold for the relation $\subseteq$ from which we will obtain that $\rlzmodel \models a \simeq b$ if and only if $\rlzmodel_\in \models a \simeq b$ by combining this with the proof for implications.
    
    For the case $\varphi \equiv \psi \rightarrow \theta$, $\rlzmodel \models \psi \rightarrow \theta$ if and only if whenever $\rlzmodel \models \psi$, $\rlzmodel \models \theta$. Using the inductive hypothesis, 
    \[
    \rlzmodel_\in \models \psi \quad \Longrightarrow \quad \rlzmodel \models \psi \quad \Longrightarrow \quad \rlzmodel \models \theta \quad \Longrightarrow \quad \rlzmodel_\in \models \theta
    \]
    and therefore $\rlzmodel_\in \models \psi \rightarrow \theta$. The same argument also shows the reverse implication.

    For the final case, $\varphi \equiv \forall x \psi(x)$, we have that $\rlzmodel \models \forall x \psi(x)$ if and only if for every $a \in \rlzstr$, $\rlzmodel \models \psi(a)$. Then
    \[
    \rlzmodel \models \forall x \psi(x) \quad \Longrightarrow \quad \forall a \in \rlzstr \, \rlzmodel \models \psi(a) \quad \Longrightarrow \quad  \forall a \in \rlzstr \, \rlzmodel_\in \models \psi(a) \quad \Longrightarrow \quad \rlzmodel_\in \models \forall x \psi(x).
    \]
    As before, the same argument also works for the reverse direction.
\end{proof}

\section{Construction of Realizability Models}\label{sec: construction of realizability models}

Suppose that \tf{V} is a model of \tf{ZF} and $\mathcal{A}$ is a realizability algebra $(\Lambda, \Pi, \prec, \Perp)$ which is defined in \tf{V}. We shall define a realizability model, $\rlzmodel = \rlzmodel^{\mathcal{A}, \tf{V}}$, over the universe as follows:
\begin{eqnarray*}
\text{For $\alpha$ an ordinal, let } \rlzstr_\alpha \coloneqq \bigcup_{\beta < \alpha} \mathcal{P}(\rlzstr_\beta \times \Pi) \\
\rlzstr \coloneqq \bigcup_{\alpha \in \tf{Ord}} \rlzstr_\alpha.
\end{eqnarray*}
Given an element $a \in \rlzstr$ we will denote by $\ff{dom}(a)$ the set of first co-ordinates of $a$. That is, 
\[
\ff{dom}(a) \coloneqq \{ b \divline \exists \pi \in \Pi \, (b, \pi) \in a \}.
\]

\begin{definition}
    Given a closed formula $\varphi$ in $Fml_{\rlzin}$ with parameters in $\rlzstr$, we define two truth values: $\falsity{\varphi} \subseteq \Pi$ and $\verity{\varphi} \subseteq \Lambda$. $\falsity{\varphi}$ will be defined recursively, and then $\verity{\varphi}$ is defined from it as:
    \[
    t \in \verity{\varphi} \Longleftrightarrow \forall \pi \in \falsity{\varphi} (t \star \pi \in \Perp).
    \]
\end{definition}

\begin{notation}
    We shall write $t \Vdash \varphi$ and say ``$t$ realizes $\varphi$'' when $y \in \verity{\varphi}$.
\end{notation}

\begin{definition}[Definition of $\falsity{\varphi}$]
    $\falsity{\varphi}$ is defined by recursion on the complexity of $\varphi$ for $\varphi$ in $Fml_{\rlzin}$:
    \begin{itemize}
        \item If $\varphi$ is atomic, then case $\varphi$ is of one of the following forms: $\top$, $\perp$, $a \notrlzin b$, $a \not\in b$ and $a \subseteq b$, where $a, b$ are in $\rlzstr$. Then,
        \begin{itemize}
            \item $\falsity{\top} \coloneqq \emptyset$,
            \item $\falsity{\perp} \coloneqq \Pi$,
            \item $\falsity{a \notrlzin b} \coloneqq \{ \pi \in \Pi \divline (a, \pi) \in b \}$,
            \item $\falsity{a \not\in b} \coloneqq \displaystyle{\bigcup}_{c \in \ff{dom}(b)} \{ t \stackapp t' \stackapp \pi \divline t, t' \in \Lambda \land \pi \in \Pi \land (c, \pi) \in b \land t \Vdash a \subseteq c \land t' \Vdash c \subseteq a \}$,
            \item $\falsity{a \subseteq b} \coloneqq \displaystyle{\bigcup}_{c \in \ff{dom}(a)} \{ t \stackapp \pi \divline t \in \Lambda \land \pi \in \Pi \land (c, \pi) \in a \land t \Vdash c \not\in b \}$.
        \end{itemize}
        Where the definitions of $\falsity{a \not\in b}$ and $\falsity{a \subseteq b}$ are formally defined simultaneously by induction on $(\ff{rk}(a), \ff{rk}(b))$ under the lexicographical ordering.
        \item If $\varphi \equiv \psi \rightarrow \theta$, then $\falsity{\varphi} \coloneqq \{ t \stackapp \pi \divline t \Vdash \psi \land \pi \in \falsity{\theta} \}$.
        \item If $\varphi \equiv \forall x \, \varphi(x)$, then $\falsity{\varphi} \coloneqq {\displaystyle \bigcup}_{a \in \rlzstr} \falsity{\varphi[a / x]}$.
    \end{itemize}
\end{definition}

\begin{definition}
    Given a formula $\varphi \in Fml_{\rlzin}$, we say that $\rlzmodel \Vdash \varphi$ if there exists a realizer $t \in \mathcal{R}$ such that $t \Vdash \varphi$.

    Given a set of formulas $\Gamma$, we say that $\rlzmodel \Vdash \Gamma$ if for every $\varphi \in \Gamma$, $\rlzmodel \Vdash \varphi$.
\end{definition}

\begin{remark}
    We remark here that in general the theory $\tf{T}^{\mathcal{A}, \tf{V}} \coloneqq \{ \varphi \in Fml_{\rlzin} \divline \exists t \in \mathcal{R} \; t \Vdash \varphi \}$ is not complete. That is, there will exist formulas $\varphi$ of $\ZFepsilon$ such that $\rlzmodel \not\Vdash \varphi$ and $\rlzmodel \not\Vdash \neg \varphi$. In particular, it is possible for there to be formulas which are true in some model of $\tf{T}^{\mathcal{A}, \tf{V}}$ but which will not be realizable by any realizer.
    
    Therefore what we want to do is consider the theory $\tf{T}^{\mathcal{A}, \tf{V}}$. We shall show that $\tf{T}^{\mathcal{A}, \tf{V}}$ is closed under classical predicate calculus (the \emph{adequacy lemma}, \Cref{theorem:adequacy}), that is if $\varphi$ is in $\tf{T}^{\mathcal{A}, \tf{V}}$ and $\varphi$ entails $\psi$ in classical logic then $\psi$ is in $\tf{T}^{\mathcal{A}, \tf{V}}$. Moreover, under a small assumption of coherency, starting from a model of $\tf{ZF}$, $\tf{T}^{\mathcal{A}, \tf{V}}$ contains every axiom of $\ZFepsilon$ and $\perp \not\in \tf{T}^{\mathcal{A}, \tf{V}}$. Therefore $\tf{T}^{\mathcal{A}, \tf{V}}$ is a coherent theory extending $\ZFepsilon$. 

    In a similar way, we can also consider the theory $\tf{T}_{\in}^{\mathcal{A}, \tf{V}} \coloneqq \{ \varphi \in Fml_{\in} \divline \exists t \in \mathcal{R} \; t \Vdash \varphi \}$. Using \Cref{section:ZFepsilonToZF}, it will be seen that, again under the same assumptions, $\tf{T}_{\in}^{\mathcal{A}, \tf{V}}$ is a coherent theory extending $\tf{ZF}$.

    Finally, on occasion we will want to work internally within our realizability model. To do this, we will work in an arbitrary model of the theory $\tf{T}^{\mathcal{A}, \tf{V}}$. Since this theory extends $\ZFepsilon$ we will use the same notation for an arbitrary model and the structure that has been defined in this section. That is, we will also use the notation $\rlzmodel = (\tf{N}, \rlzin, \in, \subseteq)$ for an arbitrary model of the realizable theory. 

\end{remark}

\noindent The first principle that we shall realize is \emph{Peirce's Law}. This can be seen to be equivalent to the law of Excluded Middle and from this it will follow that $\rlzmodel$ satisfies a classical theory.

\begin{proposition} \label{theorem:SaveCommandandNegation}
    Suppose that $\pi \in \falsity{\varphi}$. Then for any formula $\psi$, $\saverlz{\pi} \Vdash \varphi \rightarrow \psi$. In particular, $\saverlz{\pi} \Vdash \neg \varphi$.
\end{proposition}

\begin{proof}
    Suppose that $t \Vdash \varphi$ and $\pi \in \falsity{\psi}$. Then $\saverlz{\pi} \star t \stackapp \pi \succ t \star \pi$ so, since $t \star \pi \in \Perp$ by the assumption that $t \Vdash \varphi$, $\saverlz{\pi} \star t \stackapp \pi \in \Perp$.
\end{proof}

\begin{proposition}[Peirce's Law] \label{theorem:Peirce}
    For any formulas $\varphi$ and $\psi \in Fml_{\rlzin}$, $\cc \Vdash ((\varphi \rightarrow \psi) \rightarrow \varphi) \rightarrow \varphi$.
\end{proposition}

\begin{proof}
    Suppose that $t \Vdash (\varphi \rightarrow \psi) \rightarrow \varphi$ and $\pi \in \falsity{\varphi}$. By \Cref{theorem:SaveCommandandNegation}, $\saverlz{\pi} \Vdash \varphi \rightarrow \psi$ from which it follows that
    \[
    \cc \star t \stackapp \pi \succ t \star \saverlz{\pi} \stackapp \pi \in \Perp.
    \]
\end{proof}

\begin{definition}
    A realizability algebra is said to be coherent (consistent) if for every realizer $t \in \mathcal{R}$ there is a stack $\pi$ such that $t \star \pi \not\in \Perp$.
\end{definition}

\begin{theorem}
    A realizability algebra is consistent if and only if there is no realizer $t \in \mathcal{R}$ which realizes the formula $\perp$.
\end{theorem}

\subsection{Adding Defined Functions} \label{section:AddingDefinedFunctions}

Often when working directly in the theory $\ZFepsilon$ we want to use predefined notions in order to simplify the definitions we want to realize. For example, that a function consists of ordered pairs. However, without Extensionality, \emph{the} ordered pair of $a$ and $b$ is not a well-defined notion. That is, we can be in the situation where both $z$ and $z'$ satisfy ``\emph{I am an ordered pair of $a$ and $b$}'' but $z$ and $z'$ are not equal. 

In the realizability model we can regularly circumvent issues such as these by defining a \emph{canonical} definition of the function we want. For example, in \Cref{section:Pairing} we will introduce canonical interpretations of singletons, unordered pairs and ordered pairs. In order to use such objects in definitions directly in $\ZFepsilon$, it is useful to extend the language of $\ZFepsilon$ to include those operations which are uniformly realized.

\begin{definition}
    We say that a formula $\varphi(u_1, \dots, u_n, v) \in Fml_{\rlzin}$ is $\mathcal{A}$\emph{-definable by} $f \colon \tf{N}^n \rightarrow \tf{N}$ if there exists a \emph{realizer} $t \in \mathcal{R}$ such that for any $a_1, \dots, a_n, b \in \rlzstr$,
    \[
    b = f(a_1, \dots, a_n) \qquad \Longrightarrow \qquad t \Vdash \varphi(a_1, \dots, a_n, b).
    \]
$\varphi$ is said to be $\mathcal{A}$\emph{-definable} if there exists such a function $f$ and in this case we say $f$ $\mathcal{A}$\emph{-defines} $\varphi$.
\end{definition}

For example, we shall see that the function $\up \colon \rlzstr^2 \rightarrow \rlzstr$, 
\[
(u_1, u_2) \mapsto \{ (u_1, \underline{0} \stackapp \pi ) \divline \pi \in \Pi \} \cup \{ (u_2, \underline{1} \stackapp \pi) \divline \pi \in \Pi\}
\]
$\mathcal{A}$\emph{-defines} an unordered pair. That is, there exists some realizer $t$ such that for any $a_1, a_2 \in \rlzstr$
\[
t \Vdash \up(a_1, a_2) \text{ is (extensionally) equal to the unordered pair of } a_1 \text{ and } a_2.
\]
We note here that for the construction to go through in the realizability model, we need \emph{uniformity}: there is a \emph{single} realizer $t$ which realizes the formula for \emph{any} selection of parameters. \\

\noindent Given $\mathcal{A}$, we can then expand the language $\mathcal{L}_{\rlzin}$ by adding a $n$-ary function symbol $\lift{f}$ whenever $f$ $\mathcal{A}$-defines some formula $\varphi \in Fml_{\rlzin}$. Then in the realizability model we will interpret $\lift{f}$ as a ``\emph{universal lift}'' (this will formally be defined in \Cref{section:LiftingClassFunctions} where we also prove it satisfies the desired properties, but this is essentially just a way to internalise the construction within the realizability model). It will then be the case that the universal lift of a function will extend the structural property that $f$ satisfied. For example, we will have that
\[
\rlzmodel \Vdash \forall a_1 \forall a_2 (\up(a_1, a_2) \text{ is (extensionally) equal to the unordered pair of } a_1 \text{ and } a_2).
\]

From this it follows that we can add $\mathcal{A}$-definable functions to $\mathcal{L}_{\rlzin}$ and reason with them as ``canonical'' witnesses to the properties we are trying to prove. In order to simplify notation, we will still refer to this extension as $\ZFepsilon$ and implicitly assume that whenever $f$ $\mathcal{A}$-defines some formula then we have added a symbol to the language $\mathcal{L}_{\rlzin}$ which will be interpreted as $\lift{f}$.

\section{Realizers and Predicate Logic} \label{section:RealizersAndPropositionLogic}

\subsection{Realizers for Logical Connectives}

\noindent We note here some useful propositions which we shall frequently use in later analysis. \Cref{theorem:falsitysubsets} gives a  
sufficient criterion to prove that the identity term, $\lambda u \lambdaapp u$, witnesses an implication. \Cref{{theorem:ImplicationandApplication},{theorem:realizinguniversals}} prove that the realizability model satisfies the BHK interpretations for implications and universal quantifiers. \Cref{theorem:realizingexistentials} gives a sufficient condition to realizer an existential quantifier. \Cref{theorem:RealizingConjunction} gives a realizer for the conjunction of two formulas. \Cref{theorem:RealizingNotImplication} tells us that if we can realize $\varphi$ and $\neg \psi$ then we can realize $\neg (\varphi \rightarrow \psi)$. \Cref{theorem:negatingimplications} gives us information on the relationship between realizers for $\varphi \rightarrow \psi$ and $\neg \psi \rightarrow \neg \varphi$. \Cref{theorem:RealizingBoundedUniversals} gives an alternative way to express bounded universal quantifiers.

\begin{proposition} \label{theorem:falsitysubsets}
    If $\falsity{\varphi} \supseteq \falsity{\psi}$ then $\identity \Vdash \varphi \rightarrow \psi$.
\end{proposition}

\begin{proof}
    Take $t \stackapp \pi \in \falsity{\varphi \rightarrow \psi} = \{ s \stackapp \sigma \divline s \Vdash \varphi, \sigma \in \falsity{\psi} \}$. Since $\pi \in \falsity{\psi} \subseteq \falsity{\varphi}$, $t \star \pi \in \Perp$. Therefore, since $\identity \star t \stackapp \pi \succ t \star \pi$ and the latter is in $\Perp$, $\identity \star t \stackapp \pi \in \Perp$. Hence, by definition, $ \identity \Vdash \varphi \rightarrow \psi$.
\end{proof}

\begin{proposition} \label{theorem:ImplicationandApplication}
    If $t \Vdash \varphi \rightarrow \psi$ and $s \Vdash \varphi$, then $\app{t}{s} \Vdash \psi$.
\end{proposition}

\begin{proof}
    We need to prove that for every $\pi \in \falsity{\psi}$, $\app{t}{s} \star \pi \in \Perp$. So suppose that $t$ and $s$ satisfy the assumptions of the proposition and fix $\pi \in \falsity{\psi}$. Observe that $\falsity{\varphi \rightarrow \psi} = \{ r \stackapp \sigma \divline r \Vdash \varphi, \sigma \in \falsity{\psi} \}$ and therefore $s \stackapp \pi \in \falsity{\varphi \rightarrow \psi}$. This means that $t \star s \stackapp \pi \in \Perp$. Therefore, the claim follows from the fact that $\app{t}{s} \star \pi \succ t \star s \stackapp \pi$.
\end{proof}

\begin{proposition} \label{theorem:realizinguniversals}
    $t \Vdash \forall x \varphi(x) \Longleftrightarrow \forall a \in \rlzstr \, t \Vdash \varphi(a).$
\end{proposition}

\begin{proof}
    First, assume that $t \Vdash \forall x \, \varphi(x)$. Fix $a \in \rlzstr$ and $\pi \in \falsity{\varphi(a)} \subseteq \bigcup_{b \in \rlzstr} \falsity{\varphi(b)} = \falsity{\forall x \varphi(x)}$. Then, $t \star \pi \in \Perp$ and therefore $t \Vdash \varphi(a)$.

    For the other direction, suppose that for every $a \in \rlzstr$, $t \Vdash \varphi(a)$. Fix $\pi \in \falsity{\forall x \varphi(x)} = \bigcup_{b \in \rlzstr} \falsity{\varphi(b)}$. Then, by definition, we can fix some $b \in \rlzstr$ such that $\pi \in \falsity{\varphi(b)}$. Since, $t \Vdash \varphi(b)$ by assumption, $t \star \pi \in \Perp$ and therefore $t \Vdash \forall x \varphi(x)$.
\end{proof}

\begin{proposition} \label{theorem:realizingexistentials}
    If $t \Vdash \varphi(a)$ for some $a \in \rlzstr$ then $\lambda u \lambdaapp \app{u}{t} \Vdash \exists x \, \varphi(x)$.
\end{proposition}

\begin{proof}
    Note that $\exists x \, \varphi(x)$ is an abbreviation for $(\forall x \, (\varphi(x) \rightarrow \perp)) \rightarrow \perp$. So suppose that $t \Vdash \varphi(a)$ for some $a \in \rlzstr$ and suppose that $s \Vdash \forall x \, (\varphi(x) \rightarrow \perp)$ while $\pi \in \Pi$. This means that for any $\sigma \in \Pi$, $b \in \rlzstr$ and term $r$, if $r \Vdash \varphi(b)$ then $s \star r \stackapp \sigma \in \Perp$. In particular, $s \star t \stackapp \pi \in \Perp$. Thus, since $\lambda u \lambdaapp \app{u}{t} \star s \stackapp \pi \succ \app{s}{t} \star \pi \succ s \star t \stackapp \pi$, $\lambda u \lambdaapp \app{u}{t} \star s \stackapp \pi \in \Perp$, proving the claim.
\end{proof}

\begin{proposition} \label{theorem:RealizingConjunction}
    If $t \Vdash \varphi$ and $s \Vdash \psi$, then $\lambda u \lambdaapp \fapp{\app{u}{t}}{s} \Vdash \varphi \land \psi$.
\end{proposition}

\begin{proof}
    First, note that $\varphi \land \psi \equiv (\varphi \rightarrow (\psi \rightarrow \perp)) \rightarrow \perp$. Now suppose that $r \Vdash \varphi \rightarrow (\psi \rightarrow \perp)$ and $\pi \in \Pi$. By construction we then have that $s \stackapp \pi \in \falsity{\psi \rightarrow \perp}$ and $t \stackapp s \stackapp \pi \in \falsity{\varphi \rightarrow (\psi \rightarrow \perp)}$. Therefore,
    \[
    \lambda u \lambdaapp \fapp{\app{u}{t}}{s} \star r \stackapp \pi \succ \fapp{\app{r}{t}}{s} \star \pi \succ r \star t \stackapp s \stackapp \pi \in \Perp.
    \]
\end{proof}

\begin{proposition} \label{theorem:RealizingNotImplication}
    If $t \Vdash \varphi$ and $s \Vdash \psi \rightarrow \perp$ then $\lambda u \lambdaapp \inapp{s}{\app{u}{t}} \Vdash (\varphi \rightarrow \psi) \rightarrow \perp$.
\end{proposition}

\begin{proof}
    Suppose that $r \Vdash \varphi \rightarrow \psi$ and $\pi \in \Pi$. Then, by multiple uses of \Cref{theorem:ImplicationandApplication}, $\app{r}{t} \Vdash \psi$ and thus $\inapp{s}{\app{r}{t}} \Vdash \perp$, which means that $\inapp{s}{\app{r}{t}} \star \pi \in \Perp$. Thus, since $\lambda u \lambdaapp \inapp{s}{\app{u}{t}} \star r \stackapp \pi \succ \inapp{s}{\app{r}{t}} \star \pi$, $\lambda u \lambdaapp \inapp{s}{\app{u}{t}} \star r \stackapp \pi \in \Perp$, proving the claim.
\end{proof}

\begin{proposition} \label{theorem:negatingimplications}
    If $t \Vdash \varphi \rightarrow \psi$ then $\lambda u \lambdaapp \lambda v \lambdaapp \inapp{u}{\app{t}{v}} \Vdash ((\psi \rightarrow \perp) \rightarrow (\varphi \rightarrow \perp))$. Moreover, if \hbox{$t \Vdash ((\psi \rightarrow \perp) \rightarrow (\varphi \rightarrow \perp))$} then $\lambda u \lambdaapp \inapp{\cc}{\lambda k \lambdaapp \fapp{\app{t}{k}}{u}} \Vdash \varphi \rightarrow \psi$.
\end{proposition}

\begin{proof}
    Firstly, suppose that $t \Vdash \varphi \rightarrow \psi$, $s \Vdash \psi \rightarrow \perp$, $r \Vdash \varphi$ and $\pi \in \Pi$. Then, by repeated use of \Cref{theorem:ImplicationandApplication}, $\app{t}{r} \Vdash \psi$ and $\inapp{s}{\app{t}{r}} \Vdash \perp$. Thus $\lambda u \lambdaapp \lambda v \lambdaapp \inapp{u}{\app{t}{v}} \star s \stackapp r \stackapp \pi \succ \inapp{s}{\app{t}{r}} \star \pi \in \Perp$.

    For the second claim, suppose that $t \Vdash (\psi \rightarrow \perp) \rightarrow (\varphi \rightarrow \perp)$, $s \Vdash \varphi$ and $\pi \in \falsity{\psi}$. Then, by \Cref{theorem:SaveCommandandNegation}, $\saverlz{\pi} \Vdash \psi \rightarrow \perp$. Thus $\app{t}{\saverlz{\pi}} \Vdash \varphi \rightarrow \perp$ and hence $\app{t}{\saverlz{\pi}} \star s \stackapp \pi \in \Perp$. Finally,
    \begin{align*}
    \lambda u \lambdaapp \inapp{\cc}{\lambda k \lambdaapp \fapp{\app{t}{k}}{u}} \star s \stackapp \pi & \succ \inapp{\cc}{\lambda k \stackapp \fapp{\app{t}{k}}{s}} \star \pi \succ \cc \star (\lambda k \lambdaapp \fapp{\app{t}{k}}{s}) \stackapp \pi \\ 
    & \succ \lambda k \lambdaapp \fapp{\app{t}{k}}{s} \star \saverlz{\pi} \stackapp \pi \succ \fapp{\app{t}{\saverlz{\pi}}}{s} \star \pi \succ \app{t}{\saverlz{\pi}} \star s \stackapp \pi \in \Perp.
    \end{align*}
\end{proof}

\noindent Recall that we use $\forall x \rlzin a \, \varphi(x)$ as an abbreviation for $\forall x (\neg \varphi(x) \rightarrow x \notrlzin a)$. The next proposition gives an easier, alternative way to view realizers of bounded universal quantification. We observe here that the reverse direction requires an essential instance of $\cc$ along with an additional assumption on the falsity values.

\begin{definition} \label{definition:RestrictedQuantifier}
    For a set $a \in \rlzstr$, we define the restricted quantifier $\forall x^a$ to have the following meaning:
    \[
    \falsity{\forall x^a \varphi(x)} = \{ \pi \divline \exists b \in \ff{dom}(a) \: \pi \in \falsity{\varphi(b)} \} = \bigcup_{b \in \ff{dom}(a)} \falsity{\varphi(b)}.
    \]
\end{definition}

\begin{proposition} \label{theorem:RealizingBoundedUniversals}\,
    \begin{thmlist}
        \item \label{item:RealizingBoundedUniversals1} $\lambda u \lambdaapp \lambda v \lambdaapp \app{v}{u} \Vdash \forall x^a \, \varphi(x) \rightarrow \forall x (\neg \varphi(x) \rightarrow x \notrlzin a)$;
        \item \label{item:RealizingBoundedUniversals2} If $\falsity{\varphi(b)} \subseteq \falsity{b \notrlzin a}$ for every $b \in \ff{dom}(a)$, then 
        
        $\lambda u \lambdaapp \inapp{\cc}{\lambda k \lambdaapp \app{u}{k}} \Vdash \forall x (\neg \varphi(x) \rightarrow x \notrlzin a) \rightarrow \forall x^a \, \varphi(x)$.
    \end{thmlist}
\end{proposition}

\begin{proof}
    $(i)$. Suppose that $t \Vdash \forall x^a \, \varphi(x)$, $s \Vdash \neg \varphi(b)$ for some $b \in \rlzstr$ and $\pi \in \falsity{b \notrlzin a}$. Since $\falsity{b \notrlzin a} = \{ \sigma \divline (b, \sigma) \in a\}$ and this is non-empty by assumption, we must have that $b \in \ff{dom}(a)$. We need to prove that $\lambda u \lambdaapp \lambda v \lambdaapp \app{v}{u} \star t \stackapp s \stackapp \pi \in \Perp$ or, in other words, that $s \star t \stackapp \pi \in \Perp$. By the hypothesis on $t$, since $b \in \ff{dom}(a)$, $t \star \sigma \in \Perp$ for any $\sigma \in \falsity{\varphi(b)}$. Therefore, this gives us that $t \Vdash \varphi(b)$. Finally, since $s \Vdash \varphi(b) \rightarrow \perp$, $s \star t \stackapp \pi \in \Perp$ as desired.

    $(ii)$. Suppose that $t \Vdash \forall x (\neg \varphi(x) \rightarrow x \notrlzin a)$ and $\pi \in \falsity{\forall x^a \, \varphi(x)}$. First, fix $b \in \ff{dom}(a)$ such that $\pi \in \falsity{\varphi(b)}$. Now 
    \[
    \falsity{\forall x (\neg \varphi(x) \rightarrow x \notrlzin a)} = \bigcup_{c \in \rlzstr} \falsity{\neg \varphi(c) \rightarrow c \notrlzin a} = \bigcup_{c \in \rlzstr} \{ s \stackapp \sigma \divline s \Vdash \neg \varphi(c), \sigma \in \falsity{c \notrlzin a} \}.
    \]
    Therefore, for any $s \stackapp \sigma$ in the above set, $t \star s \stackapp \sigma \in \Perp$. Next, since $\pi \in \falsity{\varphi(b)}$, $\saverlz{\pi}\Vdash \neg \varphi(b)$ by \Cref{theorem:SaveCommandandNegation}. Also, by our initial assumptions, $\pi \in \falsity{b \notrlzin a}$. Thus, $\saverlz{\pi} \stackapp \pi \in \falsity{\forall x (\neg \varphi(x) \rightarrow x \notrlzin a)}$ and $t \star \saverlz{\pi} \stackapp \pi \in \Perp$. Finally, the result follows by evaluating the realizer in the proposition:
    \[
    \lambda u \lambdaapp \inapp{\cc}{\lambda k \lambdaapp \app{u}{k}} \star t \stackapp \pi \succ \inapp{\cc}{\lambda k \lambdaapp \app{t}{k}} \star \pi \succ \cc \star (\lambda k \lambdaapp \app{t}{k}) \stackapp \pi \succ \lambda k \lambdaapp \app{t}{k} \star \saverlz{\pi} \stackapp \pi \succ \app{t}{\saverlz{\pi}} \star \pi \succ t \star \saverlz{\pi} \stackapp \pi. 
    \]
\end{proof}

\noindent Using the same proof as in \Cref{theorem:realizinguniversals}, we immediately get the corollary that the quantifier $\forall x^a \varphi(x)$ behaves as we would expect a bounded universal quantifier would.

\begin{corollary} \label{theorem:RealizingBoundedQuantifier}
    For any term $t$, $t \Vdash \forall x^a \varphi(x) \Longleftrightarrow \forall b \in \ff{dom}(a) \: t \Vdash \varphi(b)$. Therefore, if $\falsity{\varphi(b)} \subseteq \falsity{b \notrlzin a}$ for every $b \in \ff{dom}(a)$ then $t \Vdash \forall x \rlzin a \, \varphi(x) \Longleftrightarrow \forall b \in \ff{dom}(a) \: t \Vdash \varphi(b)$.
\end{corollary}

\subsection{Adequacy} \label{section:Adequacy}

We shall conclude this section by showing that any realizability model is closed under classical deductions, which is called the \emph{adequacy lemma} by Krivine. From this and \Cref{section:RealzingZFepsilon} it will follow that any realizability model satisfies $\ZFepsilon$ and thus the extensional part satsifes \tf{ZF}. This will be done by proving that the structure is closed under a certain collection of logical axioms and inference rules, for more details on this we refer the reader to Chapter V of \cite{Kleene1952} or Section 1.4 and 2.3 of \cite{MendelsonIntro}.

In \cite{MendelsonIntro}, Mendelson gives a definition of First Order Logic using only the symbols $\rightarrow$, $\neg$ and $\forall$ and proves that this is sufficient to build a system of natural deduction for first order logic. It is then shown that this is a semantically complete calculus in which first order classical logic can be expressed. The system is built from the following axioms, where we have replaced $\neg \varphi$ with $\varphi \rightarrow \perp$, which hold for any well-founded formulas $\varphi$, $\psi$ and $\theta$ built from the symbols $\rightarrow$, $\perp$, variables and $\forall$. \\

\noindent \textbf{Propositional Axioms}
\begin{itemize}
    \item $\varphi \rightarrow (\psi \rightarrow \varphi)$,
    \item $(\varphi \rightarrow (\psi \rightarrow \theta)) \rightarrow ((\varphi \rightarrow \psi) \rightarrow (\varphi \rightarrow \theta))$,
    \item $((\varphi \rightarrow \perp) \rightarrow (\psi \rightarrow \perp)) \rightarrow (((\varphi \rightarrow \perp) \rightarrow \psi) \rightarrow \varphi)$.
\end{itemize}
\textbf{First-order Axioms}
\begin{itemize}
    \item $\forall x \varphi \rightarrow \varphi[y / x]$ where $y$ is any term which may be substituted for $x$ in $\varphi$,
    \item $\forall x (\varphi \rightarrow \psi(x)) \rightarrow (\varphi \rightarrow \forall x \psi)$ where $x$ does not occur freely in $\varphi$.
\end{itemize}
\textbf{Rules of Inference}
\begin{itemize}
    \item Modus Ponens: $\psi$ follows from $\varphi \rightarrow \psi$ and $\varphi$,
    \item Universal Generalisation: $\forall x \varphi(x)$ follows from $\varphi[y / x]$, where $x$ appears free in $\varphi$ and $y$ is any term which may be substituted for $x$ in $\varphi$.
\end{itemize}

\begin{lemma} \label{theorem:adequacy}
    Let $\varphi$ be a formula in $Fml_{\rlzin}$. If $\tf{CPC} \vdash \varphi$, where \tf{CPC} denotes classical predicate calculus, then there exists a realizer $t$ such that $t \Vdash \varphi$.
\end{lemma}

\begin{proof}
    Since the above set of axioms forms a semantically complete calculus for first order logic, it suffices to prove that each of the axioms can be realized for any formulas $\varphi$, $\psi$ and $\theta$ in $Fml_{\rlzin}$.

    Firstly, we see that $\lambda u \lambdaapp \lambda v \lambdaapp u \Vdash \varphi \rightarrow (\psi \rightarrow \varphi)$. This is because any element of $\falsity{\psi \rightarrow \varphi}$ is of the form $s \stackapp \pi$ where $s \Vdash \psi$ and $\pi \in \falsity{\varphi}$. Therefore, if $t \Vdash \varphi$ then $\lambda u \lambdaapp \lambda v \lambdaapp u \star t \stackapp s \stackapp \pi \succ t \star \pi \in \Perp$.

    For the second propositional axiom, a realizer for $(\varphi \rightarrow (\psi \rightarrow \theta)) \rightarrow ((\varphi \rightarrow \psi) \rightarrow (\varphi \rightarrow \theta))$ is $\lambda u \lambdaapp \lambda v \lambdaapp \lambda w \lambdaapp \twoapp{\app{u}{w}}{\app{v}{w}}$. Suppose that $t \Vdash \varphi \rightarrow (\psi \rightarrow \theta)$, $s \Vdash \varphi \rightarrow \psi$, $r \Vdash \varphi$ and $\pi \in \falsity{\theta}$. Then, by \Cref{theorem:ImplicationandApplication}, $\app{s}{r} \Vdash \psi$ from which it follows that $r \stackapp \app{s}{r} \stackapp \pi \in \falsity{\varphi \rightarrow (\psi \rightarrow \theta)}$ and $\lambda u \lambdaapp \lambda v \lambdaapp \lambda w \lambdaapp \twoapp{\app{u}{w}}{\app{v}{w}} \star t \stackapp s \stackapp r \stackapp \pi \succ \twoapp{\app{t}{r}}{\app{s}{r}} \star \pi \succ t \star r \stackapp \app{s}{r} \stackapp \pi \in \Perp$. 
    
    For the final propositional axiom, a realizer for $((\varphi \rightarrow \perp) \rightarrow (\psi \rightarrow \perp)) \rightarrow (((\varphi \rightarrow \perp) \rightarrow \psi) \rightarrow \varphi)$ will be $\lambda u \lambdaapp \lambda v \lambdaapp \inapp{\cc}{\lambda k \lambdaapp \twoapp{\app{u}{k}}{\app{v}{k}}}$. Suppose that $t \Vdash (\varphi \rightarrow \perp) \rightarrow (\psi \rightarrow \perp)$, $s \Vdash (\varphi \rightarrow \perp) \rightarrow \psi$ and $\pi \in \falsity{\varphi}$. Then, by \Cref{theorem:SaveCommandandNegation}, $\saverlz{\pi} \Vdash \varphi \rightarrow \perp$ and thus, by \Cref{theorem:ImplicationandApplication}, $\app{t}{\saverlz{\pi}} \Vdash \psi \rightarrow \perp$ and $\app{s}{\saverlz{\pi}} \Vdash \psi$. From this it follows that $\lambda u \lambdaapp \lambda v \lambdaapp \inapp{\cc}{\lambda k \lambdaapp \twoapp{\app{u}{k}}{\app{v}{k}}} \star t \stackapp s \stackapp \pi \succ \inapp{\cc}{\lambda k \lambdaapp \twoapp{\app{t}{k}}{\app{s}{k}}} \star \pi \succ \cc \star (\lambda k \lambdaapp \twoapp{\app{t}{k}}{\app{s}{k}}) \stackapp \pi \succ \lambda k \lambdaapp \twoapp{\app{t}{k}}{\app{s}{k}} \star \saverlz{\pi} \stackapp \pi \succ \twoapp{\app{t}{\saverlz{\pi}}}{\app{s}{\saverlz{\pi}}} \star \pi \succ \app{t}{\saverlz{\pi}} \star \app{s}{\saverlz{\pi}} \stackapp \pi \in \Perp$. \\

    \noindent Moving on to the axioms which discuss universal quantifiers, it is easy to see that the identity, $\identity$, is a realizer for both of the statements. For the second axiom, a realizer for $(\forall x \varphi \rightarrow \psi(x)) \rightarrow (\varphi \rightarrow \forall x \psi)$ is $\lambda u \lambdaapp \lambda v \lambdaapp \app{u}{v}$. To see this, suppose that $t \Vdash \forall x (\varphi \rightarrow \psi(x))$, $s \Vdash \varphi$ and $\pi \in \falsity{\forall x \psi(x)}$. Then we can fix $a \in \rlzstr$ such that $\pi \in \falsity{\psi(a)}$. By definition of $t$ and \Cref{theorem:realizinguniversals}, $t \Vdash \varphi \rightarrow \psi(a)$ and thus $\identity \star t \stackapp s \stackapp \pi \succ t \star s \stackapp \pi \in \Perp$. \\

    \noindent Finally, for the rules of inference, to prove Modus Ponens, one shows that whenever $t \Vdash \varphi$ and $s \Vdash \varphi \rightarrow \psi$ then there exists a term realizing $\psi$, which is the term $\app{t}{s}$ by \Cref{theorem:ImplicationandApplication}. For Universal Generalizations, suppose that $t \Vdash \varphi[y / x]$ for some $y$ which does not appear freely in $\varphi$. Then, for any $a \in \rlzstr$ we must have that $t \Vdash \varphi(a)$ and so, by \Cref{theorem:realizinguniversals}, $t \Vdash \forall x \varphi(x)$. 
\end{proof}

\section{Equality and Elementhood} \label{section:Equality}

\subsection{The relations \texorpdfstring{$\subseteq$, $\rlzin$ and $\in$}{subset, epsilon and elementhood}}

Here we show how to realize a few of the simple statements from \Cref{theorem:ZFepsilonStatements} which can be derived in $\ZFepsilon$.

\begin{proposition} \label{theorem:RealizerofSubseteq}
    Let $\theta = \lambda u \lambdaapp \lambda v \lambdaapp \bigfapp{\inapp{v}{\app{u}{u}}}{\app{u}{u}}$. Then $\app{\theta}{\theta} \Vdash \forall x \, (x \subseteq x)$. 
\end{proposition}

\begin{proof}
    We shall prove this by induction on the rank of elements (in $\tf{V}$) of $\rlzstr$. So fix $a \in \rlzstr$ and assume that the claim holds for every set of rank less than that of $a$. Since $\falsity{z \not\in a \rightarrow z \notrlzin a} = \{ t \stackapp \pi \divline t \Vdash z \not\in a, (z, \pi) \in a \}$, we only need to consider the $z$ in $\ff{dom}(a)$. So fix $b \in \ff{dom}(a)$ and $\pi$ such that $(b, \pi) \in a$. By the inductive hypothesis, $\app{\theta}{\theta} \Vdash b \subseteq b$. Now suppose that $t \Vdash b \not\in a$. Then, by definition of $\falsity{b \not\in a}$, $t \star \app{\theta}{\theta} \stackapp \app{\theta}{\theta} \stackapp \pi \in \Perp$. Finally, 
    \[
    \app{\theta}{\theta}[t] = \theta[u \coloneqq X, v \coloneqq t] = \bigfapp{\inapp{t}{\app{\theta}{\theta}}}{\app{\theta}{\theta}}
    \]
    and therefore, 
    \[
    \app{\theta}{\theta} \star t \stackapp \pi \succ \bigfapp{\inapp{t}{\app{\theta}{\theta}}}{\app{\theta}{\theta}} \star \pi \succ t \star (\app{\theta}{\theta}) \stackapp (\app{\theta}{\theta}) \stackapp \pi \in \Perp.
    \] 
\end{proof}

Using \Cref{theorem:RealizingConjunction} and the fact that $a \simeq b$ is an abbreviation for $a \subseteq b \land b \subseteq a$, we can now immediately produce a realizer for $a \simeq a$,

\begin{corollary} \label{theorem:RealizerofSimeq}
    Suppose that $t \Vdash \forall x (x \subseteq x)$, then $\lambda u \lambdaapp \fapp{\app{u}{t}}{t} \Vdash \forall x (x \simeq x)$.
\end{corollary}

\begin{proposition} \label{theorem:RealizerNotinImpliesNotrlzin}
    Suppose that $t \Vdash \forall x (x \subseteq x)$, then $\lambda u \lambdaapp \fapp{\app{u}{t}}{t} \Vdash \forall x \forall y \, (x \not\in y \rightarrow x \notrlzin y)$.
\end{proposition}

\begin{proof}
    Fix $a, b \in \rlzstr$ and suppose that $s \Vdash a \not\in b$ and $\pi \in \falsity{a \notrlzin b}$. Then $(a, \pi) \in b$ while $t \Vdash a \subseteq a$. By definition, this means that $t \stackapp t \stackapp \pi \in \falsity{a \not\in b}$ and thus $\lambda u \lambdaapp \fapp{\app{u}{t}}{t} \star s \stackapp \pi \succ s \star t \stackapp t \stackapp \pi \in \Perp$.
\end{proof}

\subsection{Non-extensional Equality}

We next introduce two further symbols, $=$ and $\inclusion$, whose additions will simplify many of the later arguments. The symbol $=$ will be used to denote non-extensional \emph{Leibniz} equality. Informally, we will have $a = b$ if and only if for any property $P$, $P(a) \Leftrightarrow P(b)$, which is stronger than asserting that $a$ and $b$ have the same $\rlzin$ elements. The symbol $\inclusion$ can be seen as an equational version of implication which will give a much simpler way to realize implications when the premise involves the symbol $=$.

\begin{definition}
    Suppose that $a, b$ are closed terms and $\varphi$ is a formula in $Fml_{\rlzin}$. Then
    \begin{equation*}
        \falsity{a \neq b} = \begin{cases}
            \falsity{\perp} \; ( = \Pi) & \text{if } a = b \\
            \falsity{\top} \; ( = \emptyset) & \text{otherwise}
        \end{cases}
    \end{equation*}
    and
    \begin{equation*}
        \falsity{a = b \inclusion \varphi} = \begin{cases}
            \falsity{\varphi} & \text{if } a = b \\
            \falsity{\top} & \text{otherwise}
        \end{cases}
    \end{equation*}  
\end{definition}

\begin{remark}
    As per usual, $a = b$ should be read as an abbreviation for $a \neq b \rightarrow \perp$.
\end{remark}

\noindent We shall now see that both of these symbols have an equivalent expression in the realizability model using the language $\mathcal{L}_{\rlzin}$.

\begin{proposition} \label{theorem:InclusionImplicationEquivalence} \,
    \begin{thmlist}
        \item $\lambda u \lambdaapp \app{u}{\identity} \Vdash (a = b \rightarrow \varphi) \rightarrow (a = b \inclusion \varphi)$;
        \item $ \lambda u \lambdaapp \lambda v \lambdaapp \inapp{\cc}{\lambda k \lambdaapp \inapp{v}{\app{k}{u}}} \Vdash (a = b \inclusion \varphi) \rightarrow (a = b \rightarrow \varphi)$. 
    \end{thmlist}
\end{proposition}

\begin{proof}
    For the first claim, suppose that $t \Vdash a = b \rightarrow \varphi$ while $\pi \in \falsity{a = b \inclusion \varphi}$. Since $\falsity{\top} = \emptyset$, we must have that $a = b$ and $\pi \in \falsity{\varphi}$. Next, since $a = b$, $\falsity{a \neq b} = \emptyset$ and therefore $\identity \Vdash a = b$ (any other fixed term would also do). Thus, $t \star \identity \stackapp \pi \in \Perp$. The conclusion then follows from the observation that $\lambda u \lambdaapp \app{u}{\identity} \star t \stackapp \pi \succ \app{t}{\identity} \star \pi \succ t \star \identity \stackapp \pi$.

    For the second claim, suppose that $t \Vdash a = b \inclusion \varphi$, $s \Vdash a = b$ and $\pi \in \falsity{\varphi}$. We have that 
    \begin{align*}
    \lambda u \lambdaapp \lambda v \lambdaapp \inapp{\cc}{\lambda k \lambdaapp \inapp{v}{\app{k}{u}}} \star t \stackapp s \stackapp \pi & \succ \inapp{\cc}{\lambda k \lambdaapp \inapp{s}{\app{k}{t}}} \star \pi \succ \cc \star (\lambda k \lambdaapp \inapp{s}{\app{k}{t}}) \stackapp \pi \\ & \succ \lambda k \lambdaapp \inapp{s}{\app{k}{t}} \star \saverlz{\pi} \stackapp \pi \succ \inapp{s}{\app{\saverlz{\pi}}{t}} \star \pi \succ s \star (\app{\saverlz{\pi}}{t}) \stackapp \pi
    \end{align*}
    and therefore it suffices to prove that $s \star (\app{\saverlz{\pi}}{t}) \stackapp \pi \in \Perp$. To this end, it will suffice to prove that $\app{\saverlz{\pi}}{t} \Vdash a \neq b$. Now, if $a = b$ then $\falsity{a = b \inclusion \varphi} = \falsity{\varphi}$, so $t \Vdash \varphi$ and therefore $\app{\saverlz{\pi}}{t} \Vdash \perp$ (since for any $\sigma \in \Pi$, $\app{\saverlz{\pi}}{t} \star \sigma \succ t \star \pi$) and thus $\app{\saverlz{\pi}}{t} \Vdash a \neq b$. On the other hand, if $a \neq b$ then $\falsity{a \neq b} = \falsity{\top} = \emptyset$, for which it follows that $\app{\saverlz{\pi}}{t} \Vdash a \neq b$. 
\end{proof}

\begin{proposition}[Leibniz Equality] \,
    \begin{thmlist}
        \item $\identity \Vdash a = b \inclusion \forall x (b \notrlzin x \rightarrow a \notrlzin x)$;
        \item $\identity \Vdash \forall x (b \notrlzin x \rightarrow a \notrlzin x) \rightarrow a = b$.
    \end{thmlist}
\end{proposition}

\begin{proof}
    If $a \neq b$ then $\falsity{a = b \inclusion \forall x (b \notrlzin x \rightarrow a \notrlzin x)} = \emptyset$, therefore it suffices to show that if $a = b$ then $\identity \Vdash \forall x (b \notrlzin x \rightarrow a \notrlzin x)$. But this follows immediately from the observation that for any $c \in \rlzstr$, if $t \Vdash b \notrlzin c$ then $t \Vdash a \notrlzin c$ since $a = b$.

    For the second claim, suppose that $t \Vdash \forall x (b \notrlzin x \rightarrow a \notrlzin x)$, $s \Vdash a \neq b$ and $\pi \in \Pi$. Then, for any $c \in \rlzstr$ we must have that $t \Vdash b \notrlzin c \rightarrow a \notrlzin c$. In particular, this holds for $c \coloneqq \{a\} \times \Pi$ and, for this $c$, $\pi \in \falsity{a \notrlzin c}$. If $a = b$ then $\falsity{a \neq b} = \falsity{\perp} = \Pi$ and therefore $s \Vdash \perp$. From this, it follows that $s \Vdash b \notrlzin c$. On the other hand, if $a \neq b$ then $\falsity{a \notrlzin c} = \emptyset$ and therefore we again have that $s \Vdash b \notrlzin c$. Hence, we must have that $s \Vdash b \notrlzin c$ and $\app{t}{s} \Vdash a \notrlzin c$, from which it follows that $t \star s \stackapp \pi \in \Perp$.
\end{proof}

\noindent We can also show that the usual axioms of equality are realized for $=$, noting in \Cref{theorem:NonExtensionalEquality5} that we only have one direction of the implication.

\begin{proposition} \,
    \begin{thmlist}
        \item \label{theorem:NonExtensionalEquality1} $\identity \Vdash \forall x (x = x)$;
        \item \label{theorem:NonExtensionalEquality2} $\identity \Vdash \forall x \forall y \, (x = y \inclusion y = x)$;
        \item \label{theorem:NonExtensionalEquality3} $\identity \Vdash \forall x \forall y \forall z \, (x = y \inclusion (y = z \inclusion x = z))$;
        \item \label{theorem:NonExtensionalEquality4} For any formula $\varphi(u) \in Fml_{\rlzin}$, $\identity \Vdash \forall x \forall y \, ( x = y \inclusion (\varphi(x) \rightarrow \varphi(y)))$;
        \item \label{theorem:NonExtensionalEquality5} $\identity \Vdash \forall x \forall y \, (x = y \inclusion \forall z (z \rlzin x \leftrightarrow z \rlzin y))$.
    \end{thmlist}
\end{proposition}

\begin{proof} \,
    \begin{enumerate}[label = $\roman*.$]
        \item Fix $a \in \rlzstr$, $t \Vdash a \neq a$ and $\pi \in \Pi$. Then $\pi \in \Pi = \falsity{a \neq a}$ and therefore $\identity \star t \stackapp \pi \succ t \star \pi \in \Perp$.
        \item If $a \neq b$ then $\falsity{a = b \inclusion b = a} = \emptyset$ and therefore $\identity \Vdash a = b \inclusion b = a$. On the other hand, if $a = b$ then $\falsity{a = b \inclusion b = a} = \falsity{b = a} = \falsity{b \neq a \rightarrow \perp} = \{ t \stackapp \pi \divline t \Vdash b \neq a, \, \pi \in \Pi\}$. Therefore, since $\falsity{b \neq a} = \Pi$, if $t \Vdash b \neq a$ and $\pi \in \Pi$, then $t \star \pi \in \Perp$ and hence $\identity \star t \stackapp \pi \in \Perp$.
        \item Fix $a, b, c \in \rlzstr$. If $a \neq b$ then $\falsity{a = b \inclusion (b = c \inclusion a = c)} = \emptyset$ and therefore $\identity$ realizes the desired statement. So, suppose that $a = b$. Then we have $\falsity{a = b \inclusion (b = c \inclusion a = c)} \linebreak[4] = \falsity{b = c \inclusion a = c}$. Again, if $b \neq c$ then $\falsity{b = c \inclusion a = c} = \emptyset$ and therefore $\identity$ realizes the desired statement. So, suppose that $a = b = c$. Then $\falsity{a = b \inclusion (b = c \inclusion a = c)} = \linebreak[4] \falsity{a = c}$ and the same argument as in $(ii)$ gives us that $\identity$ realizes the desired statement.
        \item As per usual, if $a \neq b$ then the falsity value is empty and $\identity$ realizes the statement. So, suppose that $a = b$. Then $\falsity{a = b \inclusion (\varphi(a) \rightarrow \varphi(b))} = \falsity{\varphi(a) \rightarrow \varphi(b)}$. But then if $t \Vdash \varphi(a)$ and $\pi \in \falsity{\varphi(b)} = \falsity{\varphi(a)}$ then $t \star \pi \in \Perp$ and thus $\identity \star t \stackapp \pi \in \Perp$ as required.
        \item Follows from $(iv)$.
    \end{enumerate}
\end{proof}

\section{The axioms of \texorpdfstring{$\ZFepsilon$}{ZFepsilon} are realized in \texorpdfstring{$\rlzmodel$}{N}} \label{section:RealzingZFepsilon}

We can now show that if \tf{V} is a model of \tf{ZF} and $\mathcal{A} = (\Lambda, \Pi, \prec, \Perp)$ is a realizability algebra then the realizability model $\rlzmodel^{\mathcal{A}, \tf{V}}$ is a model of $\ZFepsilon$. We shall realize each axiom in turn, following the proof structure from \cite{Krivine2012}, often showing the equivalence by using the double negation translation of each of our axioms.

\bigskip

\noindent \textbf{Extensionality}

We begin by proving the first clause. That is, we shall show that 
\begin{eqnarray*}
\identity \Vdash \forall x \forall y \Big( x \not\in y \rightarrow \forall z \Big( x \subseteq z \rightarrow ( z \subseteq x \rightarrow z \notrlzin y ) \Big) \Big), \\
\identity \Vdash \forall x \forall y \Big( \forall z \Big( x \subseteq z \rightarrow (z \subseteq x \rightarrow z \notrlzin y) \Big) \rightarrow x \not\in y \Big).
\end{eqnarray*}
For this, by \Cref{theorem:falsitysubsets} it will suffice to show that the two respective $\falsity{\cdot}$ values are equal for any $a, b \in \rlzstr$. So fix $a, b \in \rlzstr$. Then, 
\[
\falsity{\forall z ( a \subseteq z \rightarrow (z \subseteq a \rightarrow z \notrlzin b) )} = \bigcup_{c \in \rlzstr} \{ t \stackapp t' \stackapp \pi \divline t \Vdash a \subseteq c, t' \Vdash c \subseteq a, \pi \in \falsity{c \notrlzin b} \}.
\] 
Next,
\[
\falsity{ a \not\in b} = \bigcup_{c \in \ff{dom}(b)} \{ t \stackapp t' \stackapp \pi \divline (c, \pi) \in b,  t \Vdash a \subseteq c, t' \Vdash c \subseteq a \}.
\]
But, by definition, $(c, \pi) \in b$ if and only if $\pi \in \falsity{c \notrlzin b}$. This means that the only sets $c$ which can appear in the first union are those which are in $\ff{dom}(b)$ and therefore these two sets are indeed equal. Which means that $\identity$ realizes the desired implications. \\

Moving on to the second clause, this will be proved in the same way by showing that 
\begin{eqnarray*}
\identity \Vdash \forall x \forall y \Big( x \subseteq y \rightarrow \forall z (z \not\in y \rightarrow z \notrlzin x) \Big) \\
\identity \Vdash \forall x \forall y \Big( \forall z (z \not\in y \rightarrow z \notrlzin x) \rightarrow x \subseteq y \Big)
\end{eqnarray*}
For this, it will again suffice to show that the two respective $\falsity{\cdot}$ values are equal for any $a, b \in \rlzstr$. So fix $a, b \in \rlzstr$. Then, 
\[
\falsity{\forall z (z \not\in b \rightarrow z \notrlzin a)} = \bigcup_{c \in \rlzstr} \{ t \stackapp \pi \divline t \Vdash c \notin b, \pi \in \falsity{c \notrlzin a} \}
\]
and
\[
\falsity{a \subseteq b} = \bigcup_{c \in \ff{dom}(a)} \{t \stackapp \pi \divline (c, \pi) \in a, t \Vdash c \not\in b\}.
\]
From this is it again easy to see that there two sets are indeed equal which means that $\identity$ realizes the desired implications. 

\bigskip

\noindent \textbf{Induction}

Let $\psi(z, z_1, \dots, z_n)$ be a formula, it what follows we shall drop the free variables for readability. We want to show that
\[
\rlzmodel \Vdash \forall x \big( \forall y (y \rlzin x \rightarrow \psi(y)) \rightarrow \psi(x) \big) \rightarrow \forall z \psi(z).
\]
Taking $\varphi(z) \equiv \psi(z) \rightarrow \perp$, it will suffice to show that
\[
\rlzmodel \Vdash \forall x \big( \forall y (\varphi(y) \rightarrow y \notrlzin x) \rightarrow (\varphi(x) \rightarrow \perp) \big) \rightarrow \forall z (\varphi(z) \rightarrow \perp).
\]
We shall prove this by making use of a ``Turing fixed point combinator''.

\begin{proposition}
    Let $\rlzfont{A} \coloneqq \lambda u \lambdaapp \lambda v \lambdaapp \inapp{v}{\fapp{\app{u}{u}}{v}}$ and set $\rlzfont{Y} \coloneqq \rlzfont{A} \rlzfont{A}$. Then for any $s, t \in \Lambda$ and $\pi \in \Pi$, $\rlzfont{Y} \star t \stackapp s \stackapp \pi \succ t \star \app{\rlzfont{Y}}{t} \stackapp s \stackapp \pi$.
\end{proposition}

\begin{proof} \, 
    \[
    \rlzfont{Y}[v \coloneqq t] = \rlzfont{A}[u \coloneqq \rlzfont{A}, v \coloneqq t] = \inapp{t}{\fapp{\app{A}{A}}{t}} = \inapp{t}{\app{\rlzfont{Y}}{t}}
    \]
    Therefore, $\rlzfont{Y} \star t \stackapp s \stackapp \pi \succ \rlzfont{Y}[v \coloneqq t] \star s \stackapp \pi = \inapp{t}{\app{\rlzfont{Y}}{t}} \star s \stackapp \pi \succ t \star \app{\rlzfont{Y}}{t} \stackapp s \stackapp \pi$.
\end{proof}

\begin{lemma}
    For every formula $\varphi(u, u_1, \dots, u_n)$, 
    \[
    \rlzfont{Y} \Vdash \forall x \Big( \forall y (\varphi(y) \rightarrow y \notrlzin x) \rightarrow (\varphi(x) \rightarrow \perp ) \Big) \rightarrow \forall z (\varphi(z) \rightarrow \perp).
    \]
\end{lemma}

\begin{proof} 
    Fix a term $t$ such that $t \Vdash \forall x \Big( \forall y (\varphi(y) \rightarrow y \notrlzin x) \rightarrow (\varphi(x) \rightarrow \perp ) \Big)$. Then for any set $a \in \rlzstr$ we have 
    \[
    t \Vdash \Big( \forall y (\varphi(y) \rightarrow y \notrlzin a) \rightarrow (\varphi(a) \rightarrow \perp ) \Big).
    \]
    Since $\falsity{\forall z (\varphi(z) \rightarrow \perp)} = \bigcup_{b \in \rlzstr} \{ s \stackapp \pi \divline s \Vdash \varphi(b), \pi \in \Pi \}$, the proof will follow from the following lemma.

    \begin{lemma}
        For any set $b \in \rlzstr$, $\pi \in \Pi$ and $s \in \verity{\varphi(b)}$, $\rlzfont{Y} \star t \stackapp s \stackapp \pi \in \Perp$.
    \end{lemma}
    
    \noindent We shall prove the claim by induction on the rank (in \tf{V}) of $b$. So suppose that the claim holds for every $c \in \rlzstr$ with $\ff{rk}(c) < \ff{rk}(b)$. Since $\rlzfont{Y} \star t \stackapp s \stackapp \pi \succ t \star \app{\rlzfont{Y}}{t} \stackapp s \stackapp \pi$, it suffices to prove that $t \star \app{\rlzfont{Y}}{t} \stackapp s \stackapp \pi \in \Perp$. 
    
    Next, we observe that, by the assumptions on $t$ and $s$, the claim will follow from proving that $\app{\rlzfont{Y}}{t} \Vdash \forall y \, (\varphi(y) \rightarrow y \notrlzin b)$. So fix $c \in \rlzstr$, and assume that $r \Vdash \varphi(c)$ and $\sigma \in \falsity{c \notrlzin b}$. Then $(c, \sigma) \in b$ which gives us that $\ff{rk}(c) < \ff{rk}(b)$. Therefore, by the inductive hypothesis, $\rlzfont{Y} \star t \stackapp r \stackapp \sigma \in \Perp$, from which it follows that $\app{\rlzfont{Y}}{t} \star r \stackapp \sigma \in \Perp$ as required. 
\end{proof}

\bigskip

\noindent \textbf{Separation}

Let $a \in \rlzstr$ and $\varphi$ be a formula (with parameters). We shall prove that a set which realizes this related instance of Separation is $b \coloneqq \{ (x, t \stackapp \pi) \divline (x, \pi) \in a, t \Vdash \varphi(x) \}$. Now, 
\[
\falsity{x \notrlzin b} = \{ \sigma \in \Pi \divline (x, \sigma) \in b \} = \{ t \stackapp \pi \divline t \Vdash \varphi(x), \pi \in \Pi, (x, \pi) \in a \} = \falsity{\varphi(x) \rightarrow x \notrlzin a}.
\]
Therefore, $\identity \Vdash \forall x (x \notrlzin b \rightarrow ( \varphi(x) \rightarrow x \notrlzin a))$ and $\identity \Vdash \forall x ((\varphi(x) \rightarrow x \notrlzin a ) \rightarrow x \notrlzin b).$

\bigskip

\noindent \textbf{Pairing}

Fix $a, b \in \rlzstr$. We shall prove that a set which realizes this instance of pairing is $c \coloneqq \{a, b\} \times \Pi$. This is because $\falsity{a \notrlzin c} = \falsity{b \notrlzin c} = \Pi = \falsity{\perp}$. Therefore, $\identity \Vdash a \rlzin c$ and $\identity \Vdash b \rlzin c$ (note that $a \rlzin c$ is the formula $a \notrlzin c \rightarrow \perp$).

\bigskip

\noindent \textbf{Union}

Fix $a \in \rlzstr$. We shall prove that a set which realizes this instance of unions is $b \coloneqq \{ (c, \sigma) \divline \exists (x, \pi) \in a \, (c, \sigma) \in x \}$. By \Cref{theorem:falsitysubsets}, it will suffice to prove that for any $x,c \in \rlzstr$,
\[
\falsity{c \notrlzin b \rightarrow x \notrlzin a} \subseteq \falsity{ c \notrlzin x \rightarrow x \notrlzin a}
\]
because from this we will obtain $\identity \Vdash \forall x \forall c \Big( ( c \notrlzin x \rightarrow x \notrlzin a) \rightarrow (c \notrlzin b \rightarrow x \notrlzin a) \Big)$ which is equivalent to the axiom of unions.

So suppose that $s \Vdash c \notrlzin b$ and $(x, \pi) \in a$. We first prove that $s \Vdash c \notrlzin x$. To see this, take $\sigma \in \falsity{c \notrlzin x}$. Then, by definition, $(c, \sigma) \in x$ from which it follows that $(c, \sigma) \in b$ and therefore $s \star \sigma \in \Perp$, which implies that $s \Vdash c \notrlzin x$. Therefore, $s \stackapp \pi \in \falsity{c \notrlzin x \rightarrow x \notrlzin a}$ as required.

\bigskip

\noindent \textbf{Weak Power Set}

Fix $a \in \rlzstr$. We shall prove that a set which realizes this instance of weak power set is $b \coloneqq \mathcal{P}(\ff{dom}(a) \times \Pi) \times \Pi$. To do this, given $x \in \rlzstr$, let 
\[
y_x \coloneqq \{ (z, t \stackapp \pi) \divline t \Vdash z \rlzin x, (z, \pi) \in a \} \in \mathcal{P}(\ff{dom}(a) \times \Pi).
\]
Then $\falsity{z \notrlzin y_x} = \falsity{z \rlzin x \rightarrow z \notrlzin a}$ and $\falsity{y_x \notrlzin b} = \Pi$. So, $\identity \Vdash \forall z (z \notrlzin y_x \rightarrow (z \rlzin x \rightarrow z \notrlzin a))$ and $\identity \Vdash \forall z ((z \rlzin x \rightarrow z \notrlzin a) \rightarrow z \notrlzin y_x)$. Therefore, for any $\pi \in \Pi$, 
\[
\identity \stackapp \identity \stackapp \pi \in \falsity{\forall y \big( \forall z (z \notrlzin y \rightarrow (z \rlzin x \rightarrow z \notrlzin a))
\rightarrow \big( \forall z((z \rlzin x \rightarrow z \notrlzin a) \rightarrow z \notrlzin y) \rightarrow y \notrlzin b \big) \big)}.
\]
From which it follows that
\begin{multline*}
\lambda u \lambdaapp \fapp{\app{u}{\identity}}{\identity} \Vdash \forall x \bigg( \forall y \bigg( \forall z (z \notrlzin y \rightarrow (z \rlzin x \rightarrow z \notrlzin a)) \\
\rightarrow \Big( \forall z((z \rlzin x \rightarrow z \notrlzin a) \rightarrow z \notrlzin y) \rightarrow y \notrlzin b \Big) \bigg) \rightarrow \perp \bigg).
\end{multline*}

\bigskip

\noindent \textbf{Collection}

\noindent Let $a \in \rlzstr$ and $\varphi(u, v)$ be a formula (with parameters). Using Collection in \tf{V} we can find a set $Y$ such that 
\[
\forall (x, \pi) \in a \, \forall t \in \Lambda \, \exists y \, t \Vdash \varphi(x, y) \rightarrow \forall (x, \pi) \in a \, \forall t \in \Lambda \, \exists y \in Y \, t \Vdash \varphi(x, y).
\]
We shall prove that a set which realizes the instance of collection for $a$ and $\varphi$ in $\rlzmodel$ is $b \coloneqq \{ (y, \pi) \divline \exists t \in \Lambda, \exists x \, (x, \pi) \in a, t \Vdash \varphi(x, y), y \in Y \}.$ For this, it will suffice to prove that for any $x \in \rlzstr$
\[
\falsity{\forall y (\varphi(x, y) \rightarrow x \notrlzin a)} \subseteq \falsity{\forall y (\varphi(x, y) \rightarrow y \notrlzin b)}
\]
as this will then imply that $\identity \Vdash \forall x \Big( \forall y ( \varphi(x, y) \rightarrow y \notrlzin b) \rightarrow \forall y (\varphi(x, y) \rightarrow x \notrlzin a) \Big)$ which one can see is equivalent to Collection. So, fix $x \in \rlzstr$ and suppose that $t \stackapp \pi \in \falsity{\forall y(\varphi(x, y) \rightarrow x \notrlzin a)} = \bigcup_{c \in \rlzstr} \falsity{\varphi(x, c) \rightarrow x \notrlzin a}$. Then, for some $c \in \rlzstr$, $t \Vdash \varphi(x, c)$ and $(x, \pi) \in a$. By definition of $Y$, we can fix some $c' \in Y$ such that $t \Vdash \varphi(x, c')$. Moreover, by definition, $(c', \pi) \in b$ and therefore $\pi \in \falsity{c' \notrlzin b}$. From which we can conclude that $t \stackapp \pi \in \falsity{\varphi(x,c') \rightarrow y \notrlzin c'} \subseteq \falsity{\forall y (\varphi(x, y) \rightarrow y \notrlzin b)}$ as required.

\bigskip

\noindent \textbf{Infinity}

Note that the axiom of infinity is equivalent to 
\[
\forall a \, \exists b \, \Big( a \rlzin b \land \forall x (\forall y ( x \rlzin y \rightarrow y \notrlzin b) \rightarrow x \notrlzin b)  \Big). 
\]
Fix $a \in \rlzstr$. We shall define a set $b$ witnessing the axiom of infinity as follows: first set $a^0 \coloneqq a$ and, for $n \in \omega$, let $a^{n+1} \coloneqq \{a^n\} \times \Pi$.\footnote{this will in fact be a name for a set which is extensionally equal to $\{a\}$ by \Cref{theorem:SingularFunction}} Finally, let
\[
b \coloneqq \{ (a^n, \pi) \divline n \in \omega, \pi \in \Pi \}.
\]
Since $\{a\} \times \Pi \subseteq b$, $\falsity{a \notrlzin b} = \Pi$ and therefore $\identity \Vdash a \rlzin b$. Therefore, it will suffice to prove that 
\[
\lambda u \lambdaapp \app{u}{\identity} \Vdash \forall x (\forall y ( x \rlzin y \rightarrow y \notrlzin b) \rightarrow x \notrlzin b).
\]
So, fix $c$ and suppose that we can find some $t$ and $\pi$ such that $t \Vdash \forall y ( c \rlzin y \rightarrow y \notrlzin b)$ and $\pi \in \falsity{c \notrlzin b}$. Note that, by definition, $(c, \pi) \in b$. For our witness $y$ we shall take $d = \{c\} \times \Pi$, noting that for any $\pi' \in \Pi$, we shall have that $(d, \pi') \in b$.

We shall first prove that $\identity \Vdash c \notrlzin d \rightarrow \perp$. To do this suppose $s \Vdash c \notrlzin d$ and $\sigma \in \Pi$. By definition, $\forall \tau \in \falsity{c \notrlzin d} = \{ \pi' \divline (c, \pi') \in d \} = \Pi$, $s \star \tau \in \Perp$. Therefore, $\identity \star s \stackapp \sigma \succ s \star \sigma \in \Perp$, which completes this observation. Then, since $\identity \Vdash c \rlzin d$, $\app{t}{\identity} \Vdash d \notrlzin b$ and thus, $\lambda u \lambdaapp \app{u}{\identity} \star t \stackapp \pi \succ \app{t}{\identity} \stackapp \pi \in \Perp$, which completes the proof.

\medskip

We can now conclude that our realizability model satisfies $\ZFepsilon$.

\begin{theorem} \label{theorem:rlzmodelModelsZFepsilon}
Suppose that $\tf{V} \models \tf{ZF}$ and $\mathcal{A} = (\Lambda, \Pi, \prec, \Perp)$ is a realizability algebra. Then for every axiom $\varphi$ of $\ZFepsilon$, $\rlzmodel^{\mathcal{A}, \tf{V}} \Vdash \varphi$.
\end{theorem}

\section{Names and Ordinals} \label{section:ReishandOrdinals}

\subsection{Reish Names}

In this short subsection we introduce a way to interpret every ground model set in the realizability model, by defining the $\emph{reish of } a$, where reish stands for \emph{recursive}. In the traditional presentation where the sets of the realizability model are all sets in the ground model, one considers the gimel of $a$, $\cjgimel(a)$, defined by $\cjgimel(a) \coloneqq a \times \Pi$. The issue with this operation is that if $a$ is not a name in \rlzstr, then in the extensional model $\cjgimel(a)$ may be very impoverished. Notably, if $a$ does not contain any pair of the form $(b, \pi)$ where $\pi \in \Pi$ then one can easily see that $\falsity{\forall x (x \notrlzin a)} = \emptyset$ from which it follows that in the extensional model $\Vdash a \simeq \emptyset$. In particular, from this we can conclude that for any ordinal $\alpha$, $\Vdash \forall x^{\cjgimel(\alpha)} (x \simeq 0)$.

Instead, we shall consider $\fullname{a}$ which is constructed by applying the gimel process recursively to all elements of $a$. This is a very blunt tool which in general will be difficult to work with because for any $b \in a$, $\falsity{\fullname{b} \notrlzin \fullname{a}} = \Pi$. This will cause problems in practice because if $\pi \in \falsity{\forall x^{\fullname{a}} \varphi(x)}$ then we cannot specify the $b \in a$ for which $\pi \in \falsity{\varphi(\fullname{b})}$. Therefore, in order to have $t \Vdash \forall x^{\fullname{a}} \varphi(x)$ one needs $t$ to uniformly realize $t \Vdash \varphi(\fullname{b})$ for all $b \in a$ without reference to which witness we are working with. For example, it is not known if it is possible to prove
\[
\forall n \rlzin \fullname{\omega} \, \forall x ((\forall y (y + 1 \notrlzin x \rightarrow y \notrlzin x) \rightarrow n \notrlzin x) \rightarrow \fullname{0} \notrlzin x),
\]
which would imply that every element of $\fullname{\omega}$ was an integer and therefore that $\fullname{\omega}$ was extensionaly equal to the first limit ordinal. However, it is a good starting point to work from and we will consider an alternative approach in \ref{section:OrdinalRepresentations}.

\begin{definition}
    We define an interpretation $\fullname{x}$ recursively for $x \in \tf{V}$ as
    \[
    \fullname{x} \coloneqq \{ (\fullname{y}, \pi) \divline y \in x, \, \pi \in \Pi \}.
    \]
\end{definition}

\begin{proposition} \label{theorem:SubsetsandFullNames}
    If $a \subseteq b$ then $\rlzmodel \Vdash \fullname{a} \subseteq \fullname{b}$.
\end{proposition}

\begin{proof}
    Fix $t$ to be a realizer such that $t \Vdash \forall x (x \subseteq x)$ (for example, see \Cref{theorem:RealizerofSubseteq}) and fix $s \stackapp \pi \in \falsity{\fullname{a} \subseteq \fullname{b}} = \bigcup_c \{ s \stackapp \pi \divline 	(\fullname{c}, \pi) \in \fullname{a}, s \Vdash \fullname{c} \not\in \fullname{b} \}$. Now, if $(\fullname{c}, \pi) \in \fullname{a}$, then we must have that $c \in a$ and therefore $c \in b$. Thus, $t \stackapp t \stackapp \pi \in \falsity{\fullname{c} \not\in \fullname{b} }$ which means that $s \star t \stackapp t \stackapp \pi \in \Perp$. Thus, $\lambda u \lambdaapp \fapp{\app{u}{t}}{t} \Vdash \fullname{a} \subseteq \fullname{b}$. 
\end{proof}

\begin{observation} \label{observation:reishPreservesElementhood}
    If $a \in b$ then $\identity \Vdash \fullname{a} \rlzin \fullname{b}$. Thus, if $a \in b$ then $\rlzmodel \Vdash \fullname{a} \not\simeq \fullname{b}$.
\end{observation}

\subsection{Ordinals}

We now wish to discuss ordinals in $\ZFepsilon$. In order to do this, we will first specify precisely what we mean by this concept. For the purpose of this section, we will be very careful to differentiate between the $\rlzin$ version and the $\in$ version of each definition.

\begin{definition}
    ($\ZFepsilon$) We say that a set $a$ is a $\rlzin$\emph{-ordinal} if it is a $\rlzin$-transitive set of $\rlzin$-transitive sets. That is,
    \[
    \forall x \rlzin a \, \forall y \rlzin x \, (y \rlzin a) \quad \land \quad \forall z \rlzin a \, \forall x \rlzin z \, \forall y \rlzin x \, (y \rlzin z).
    \]
\end{definition}

\noindent Note that over $\ZFepsilon$, $\rlzin$-ordinals are not always the same as the standard definition of ordinals as transitive sets well-ordered by the $\rlzin$-relation. For example, if $\fullname{2}$ has size greater than 2 then there are two ordinals, $a$ and $b$, such that $a \notrlzin b$, $b \notrlzin a$ but $(a \cup \{a\}) \cup (b \cup \{b\})$ is an $\rlzin$-ordinal on which the $\rlzin$-relation does not linearly order the set. On the other hand, it is easy to see in \tf{ZF} that $\alpha$ is an ordinal if and only if it is a transitive set of transitive sets.

\begin{proposition} $(\ZFepsilon)$ \,
    \begin{thmlist}
        \item If $a$ is a $\rlzin$-transitive set, then it is a $\in$-transitive set.
        \item If $a$ is a $\rlzin$-ordinal, then it is a $\in$-ordinal.
    \end{thmlist}
\end{proposition}

\begin{proof} Suppose that $a$ is a $\rlzin$-transitive set and take $c \in b \in a$. Then there exists some $x \rlzin a$ such that $x \simeq b$ and there exists some $y \rlzin x$ such that $y \simeq c$. Since $a$ is assumed to be $\rlzin$-transitive, $y \rlzin a$. Therefore $x \in a$ by definition of $\in$.

Next, suppose that $a$ is a $\rlzin$-ordinal. We have already shown that $a$ is $\in$-transitive, so it suffices to prove that every $b \in a$ is $\in$-transitive. So let $d \in c \in b \in a$. Then we can find $z \rlzin y \rlzin x \rlzin a$ such that $x \simeq b$, $y \simeq c$ and $z \simeq d$. Since $a$ is a $\rlzin$-ordinal, $z \rlzin x$ and therefore $d \in x$. Finally, $d \in x$ and $x \simeq b$ gives us $d \in b$, as required.
\end{proof}

\begin{proposition}
    There exists a realizer $\theta \in \mathcal{R}$ such that whenever $\delta$ is an ordinal in \tf{V} then $\theta \Vdash ``\fullname{\delta} \text{ is a } \rlzin\text{-ordinal''}$.
\end{proposition}

\begin{proof}
    Let $\delta$ be an ordinal in \tf{V}. We shall show that $\fullname{\delta}$ is a $\rlzin$-transitive set. The fact that it consists of $\rlzin$-transitive sets will follow by a similar argument. To do this, we shall show that 
    \[
    \identity \Vdash \forall x^{\fullname{\delta}} \forall y (y \notrlzin \fullname{\delta} \rightarrow y \notrlzin x).
    \]
    To do this, fix $\beta \in \delta$, $c \in \rlzstr$, $t \Vdash c \notrlzin \fullname{\delta}$ and $\pi \in \falsity{c \notrlzin \fullname{\beta}}$. Now $\falsity{c \notrlzin \fullname{\beta}} = \{ \sigma \divline (c, \sigma) \in \fullname{\beta} \}$. So, since this set is non-empty, it must be the case that $\falsity{c \notrlzin \fullname{\beta}} = \Pi$ and $c = \fullname{\gamma}$ for some $\gamma \in \beta$. Therefore, $\falsity{c \notrlzin \fullname{\delta}} = \falsity{\fullname{\gamma} \notrlzin \fullname{\delta}} = \Pi$ and therefore $t \star \pi \in \Perp$, from which the result follows.
\end{proof}

\subsection{Ordinal Representations} \label{section:OrdinalRepresentations}

We now present an alternative method to represent certain ordinals in the realizability model which will give a much more usable interpretation. This interpretation will give a representation for every ordinal less than the size of the realizability model. This method is explored in much more detail in \cite{FontanellaGeoffroy2020} where it was shown that the objects in the realizability model named by these ordinals behave like their ground model counterparts.

Let $\mathcal{A} = (\Lambda_{(A, B)}, \Pi_{(A, B)}, \prec, \Perp)$ be a realizability algebra where $|A| = \kappa$ is infinite and $|B| \leq \kappa$ and let $(\nu_\alpha \divline \alpha \in \kappa)$ be an enumeration of $\Lambda$. Without loss of generality, we can assume that for every $n \in \omega$, $\nu_n = \underline{n}$ with just the small caveat that if $\kappa = \omega$ then this will in fact be an enumeration of order type $\omega + \omega$ rather than $|\Lambda|$.

\begin{definition}
    For $\alpha \in \kappa$, let $\hat{\alpha} \coloneqq \{ ( \hat{\beta}, \nu_\beta \stackapp \pi) \divline \pi \in \Pi, \beta \in \alpha\}$.
\end{definition}

\begin{remark} \label{remark:ChurchNumerals}
    The precise ordering for the term structure does not matter, what is important is that we have a way to ``decode'' which element of $\hat{\alpha}$ we have picked. Namely, if $\falsity{x \notrlzin \hat{\alpha}} \neq \emptyset$ then it must be of the form $\{ \nu_\beta \stackapp \pi \divline \pi \in \Pi \}$ for some $\beta \in \alpha$ and therefore our falsity value tells us which element of $\hat{\alpha}$ our $x$ is meant to correspond to.
\end{remark}

\begin{proposition} \label{theorem:HatNameElements}
    If $\beta < \alpha \leq \kappa$ then $\lambda u \lambdaapp \app{u}{\nu_\beta} \Vdash \hat{\beta} \rlzin \hat{\alpha}$.
\end{proposition}

\begin{proof}
    Suppose that $t \Vdash \hat{\beta} \notrlzin \hat{\alpha}$ and $\pi \in \Pi$. Then, by definition, we have that $\nu_\beta \stackapp \pi \in \falsity{\hat{\beta} \notrlzin \hat{\alpha}}$ from which it follows that $\lambda u \lambdaapp \app{u}{\nu_\beta} \star t \stackapp \pi \succ \app{t}{\nu_\beta} \star \pi \succ t \star \nu_\beta \stackapp \pi \in \Perp$.
\end{proof}

\noindent Recall \Cref{definition:RestrictedQuantifier} where we defined the restricted quantifier $\forall x^a$ to have the falsity value
\[
\falsity{\forall x^a \varphi(x)} = \{ \pi \divline \exists b \in \ff{dom}(a) \, \pi \in \falsity{\varphi(b)} \}
\]
and then proved that this behaves as a restricted quantifier in certain cases. Here we want to modify the definition of the restricted quantifier to take into account the $\nu_\alpha$. In order to ease notation we shall use the same notation for both these notions and note that it will always be clear from the context which of these interpretations we are using.

\begin{definition}
    For $\alpha \leq \kappa$, we define the restricted quantifier $\forall x^{\hat{\alpha}}$ to have the following meaning:
    \[
    \falsity{\forall x^{\hat{\alpha}} \varphi(x)} = \bigcup_{\beta \in \alpha} \{ \nu_\beta \stackapp \pi \divline \pi \in \falsity{\varphi(\hat{\beta})} \}.
    \]
\end{definition}

As before, we can see that this behaves as the bounded quantifier.

\begin{proposition} \label{theorem:RealizingBoundedUniversalsForHats} For every $\alpha \leq \kappa$,
    \begin{thmlist}
        \item \label{item:RealizingBoundedUniversalsForHats1} $\lambda u \lambdaapp \lambda v \lambdaapp \lambda w \lambdaapp \inapp{v}{\app{u}{w}} \Vdash \forall x^{\hat{\alpha}} \, \varphi(x) \rightarrow \forall x (\neg \varphi(x) \rightarrow x \notrlzin \hat{\alpha})$;
        \item \label{item:RealizingBoundedUniversalsForHats2} $\lambda u \lambdaapp \app{\cc \,}{u} \Vdash \forall x (\neg \varphi(x) \rightarrow x \notrlzin \hat{\alpha}) \rightarrow \forall x^{\hat{\alpha}} \, \varphi(x)$.
    \end{thmlist}
\end{proposition}

\begin{proof}
    $(i)$. Suppose that $t \Vdash \forall x^{\hat{\alpha}} \, \varphi(x)$, $s \Vdash \neg \varphi(b)$ for some $b \in \rlzstr$ and $\pi \in \falsity{b \notrlzin \hat{\alpha}}$. Since $\falsity{b \notrlzin \hat{\alpha}} = \{ \sigma \divline (b, \sigma) \in \hat{\alpha}\}$ and this is non-empty by assumption, we must have that $b = \hat{\beta}$ for some $\beta \in \alpha$ and $\pi = \nu_\beta \stackapp \pi'$ for some $\pi' \in \Pi$. We need to prove that $\lambda u \lambdaapp \lambda v \lambdaapp \lambda w \lambdaapp \inapp{v}{\app{u}{w}} \star t \stackapp s \stackapp \nu_\beta \stackapp \pi' \in \Perp$ or, in other words, that $s \star (\app{t}{\nu_\beta}) \stackapp \pi' \in \Perp$. By the hypothesis on $t$, since $\beta \in \alpha$, $t \star \nu_\beta \stackapp \sigma \in \Perp$ for any $\sigma \in \falsity{\varphi(\hat{\beta})}$. Therefore, this gives us that $\app{t}{\nu_\beta} \Vdash \varphi(\hat{\beta})$. Finally, since $s \Vdash \varphi(\hat{\beta}) \rightarrow \perp$, $s \star \app{t}{\nu_\beta} \stackapp \pi \in \Perp$ as desired.

    $(ii)$. Suppose that $t \Vdash \forall x (\neg \varphi(x) \rightarrow x \notrlzin \hat{\alpha})$ and $\sigma \in \falsity{\forall x^{\hat{\alpha}} \, \varphi(x)}$. First, fix $\beta \in \alpha$ and $\pi \in \Pi$ such that $\sigma = \nu_\beta \stackapp \pi \in \falsity{\varphi(\hat{\beta})}$. Now 
    \[
    \falsity{\forall x (\neg \varphi(x) \rightarrow x \notrlzin \hat{\alpha})} = \bigcup_{c \in \rlzstr} \falsity{\neg \varphi(c) \rightarrow c \notrlzin \hat{\alpha}} = \bigcup_{\gamma \in \alpha} \{ s \stackapp \nu_\gamma \stackapp \sigma \divline s \Vdash \neg \varphi(\hat{\gamma}), \nu_\gamma \stackapp \sigma \in \falsity{\hat{\gamma} \notrlzin \hat{\alpha}} \}.
    \]
    Therefore, for any $s \stackapp \nu_\gamma \stackapp \sigma$ in the above set, $t \star s \stackapp \nu_\gamma \stackapp \sigma \in \Perp$. Next, since $\nu_\beta \stackapp \pi \in \falsity{\varphi(\hat{\beta})}$, $\saverlz{\nu_\beta \stackapp \pi}\Vdash \neg \varphi(\hat{\beta})$ by \Cref{theorem:SaveCommandandNegation}. Thus, $\saverlz{\nu_\beta \stackapp \pi} \stackapp \nu_\beta \stackapp \pi \in \falsity{\forall x (\neg \varphi(x) \rightarrow x \notrlzin \hat{\alpha})}$ and $t \star \saverlz{\nu_\beta \stackapp \pi} \stackapp \nu_\beta \stackapp \pi \in \Perp$. Finally, the result follows by evaluating the realizer in the proposition:
    \[
    \lambda u \lambdaapp \app{\cc \,}{u} \star t \stackapp \nu_\beta \stackapp \pi \succ \app{\cc \,}{t} \star \nu_\beta \stackapp \pi \succ \cc \star t \stackapp \nu_\beta \stackapp \pi \succ t \star \saverlz{\nu_\beta \stackapp \pi} \stackapp \nu_\beta \stackapp \pi \in \Perp.
    \]
\end{proof}

\begin{proposition}
    There exists a realizer $\theta \in \mathcal{R}$ such that for every $\alpha \leq \kappa$, 
    \[
    \theta \Vdash ``\hat{\alpha} \text{ is a } \rlzin\text{-ordinal''}.
    \]
\end{proposition}

\begin{proof}
    As in the case for $\fullname{\alpha}$ we shall show that $\hat{\alpha}$ is a $\rlzin$-transitive set. The fact that it consists of $\rlzin$-transitive sets will follow by a similar argument. To do this, we shall show that
    \[
    \lambda v \lambdaapp \lambda w \lambdaapp \lambda k \lambdaapp \inapp{w}{\lambda u \lambdaapp \app{u}{k}} \Vdash \forall x^{\hat{\alpha}} \forall y (y \notrlzin \hat{\alpha} \rightarrow y \notrlzin x).
    \]
    First, observe that $\falsity{\forall x^{\hat{\alpha}} \forall y (y \notrlzin \hat{\alpha} \rightarrow y \notrlzin x)} = \bigcup_{\beta \in \alpha} \bigcup_{\gamma \in \beta} \{ \nu_\beta \stackapp t \stackapp \nu_\gamma \stackapp \pi \divline t \Vdash \hat{\gamma} \notrlzin \hat{\alpha}, \, \pi \in \Pi\}$. So, fix $\gamma < \beta < \alpha$, $\pi \in \Pi$ and $t \Vdash \hat{\gamma} \notrlzin \hat{\alpha}$. Since $\gamma \in \alpha$, by \Cref{theorem:HatNameElements}, $\lambda u \lambdaapp \app{u}{\nu_\gamma} \Vdash \hat{\gamma} \rlzin \hat{\alpha}$ and therefore 
    \[
    \lambda v \lambdaapp \lambda w \lambdaapp \lambda k \lambdaapp \inapp{w}{\lambda u \lambdaapp \app{u}{k}} \star \nu_\beta \stackapp t \stackapp \nu_\gamma \stackapp \pi \succ \inapp{t}{\lambda u \lambdaapp \app{u}{\nu_\gamma}} \star \pi \succ t \star (\lambda u \lambdaapp \app{u}{\nu_\gamma}) \stackapp \pi \in \Perp.
    \]
\end{proof}

\noindent A key difference between $\rlzin$-ordinals and $\in$-ordinals is the principle of \emph{trichotomy}; for any $\in$-ordinal $\alpha$, if $\beta, \gamma \in \alpha$ then either $\beta \in \gamma$, $\gamma \in \beta$ or $\beta \simeq \gamma$. We end this section with a discussion of $\rlzin$-ordinals which satisfy $\rlzin$\emph{-trichotomy}. Such ordinals are particularly useful because they have a much richer structure within the non-extensional language. For example, we shall see that such ordinals contain unique representatives of their $\in$-elements and satisfy the \emph{Least Ordinal Principle}; if $X$ is a non-empty subset of $\hat{\alpha}$ then $X$ has a $\rlzin$-least element. To ease notation it is useful to add in the symbol $\subseteq_{\rlzin}$ for non-extensional subsets.

\begin{definition}
    $(\ZFepsilon)$ $a \subseteq_{\rlzin} b$ is an abbreviation for the formula $\forall x (x \rlzin a \rightarrow x \rlzin b)$.
\end{definition}

\begin{definition}
    $(\ZFepsilon)$ We say that a $\rlzin$-ordinal $a$ is $\rlzin$\emph{-Trichotomous} if
    \[
    \forall x \rlzin a \forall y \rlzin a (x \rlzin y \lor x = y \lor y \rlzin x).
    \]
\end{definition}

\begin{proposition} \label{theorem:HatOrdinalsHaveUniqueElements}
    $\ZFepsilon \vdash \forall a (a \text{ is a } \rlzin\text{-Trichotomous ordinal } \rightarrow \forall x \rlzin a \forall y \rlzin a ( x \simeq y \rightarrow x = y)).$
\end{proposition}

\begin{proof}
    Suppose that $a$ is a $\rlzin$-Trichotomous ordinal and fix $x, y \rlzin a$ such that $x \neq y$. Then, by trichotomy, we must have that either $x \rlzin y$ or $y \rlzin x$ from which it follows that either $x \in y$ or $y \in x$. But, from either of these cases it follows that $x \not\simeq y$, otherwise we would obtain the contradictory statement $x \in x$.
\end{proof}

\begin{proposition}[Least Ordinal Principle] \label{LeastOrdinalPrinciple}
    \begin{multline*}
    \ZFepsilon \vdash \forall a \forall X (a \text{ is a } \rlzin\text{-Trichotomous ordinal } \land X \subseteq_{\rlzin} a \land \exists z (z \rlzin X) \\
    \longrightarrow \quad \exists z \rlzin X \forall y \rlzin X (z \rlzin y \lor z = y)).
    \end{multline*}
\end{proposition}

\begin{proof}
    Suppose that $a$ is a $\rlzin$-Trichotomous ordinal and fix $X \subseteq_{\rlzin} a$ such that $\exists z (z \rlzin X)$. Using $\in$-Foundation (which follows classically from $\in$-Induction) we can fix $z \rlzin X$ such that $z \cap X \simeq \emptyset$, that is $\forall y \neg (y \in z \land y \in X)$. To see that this $z$ witnesses the claim, by trichotomy it suffices to show that $\forall y \rlzin X (y \notrlzin z)$. But, if $y \rlzin X$ then $y \in X$ and therefore $y \not\in z$ by $\in$-minimality of $z$. Thus, $y \notrlzin z$. 
\end{proof}

\noindent One would like to show that for any $\alpha \in \kappa$, $\hat{\alpha}$ is a trichotomous $\rlzin$-ordinal. However, in order to do that we need to add a special instruction, $\chi$, which allows one to compare the indices of ordinals. This is done in section 3 of \cite{FontanellaGeoffroy2020}.

\begin{definition} \label{definition:PropertyChi}
    Let $\chi \in \Lambda$ be a special instruction. We extend the order $\prec$ to be the smallest pre-order on $\Lambda \star \Pi$ such that for any $\alpha, \beta \in \kappa$, $t, s, r \in \Lambda$ and $\pi \in \Pi$,
    \[
    \chi \star \nu_\alpha \stackapp \nu_\beta \stackapp t \stackapp s \stackapp r \stackapp \pi \succ \begin{cases}
        t \star \pi & \text{if }\ \alpha < \beta, \\
        s \star \pi & \text{if }\ \alpha = \beta, \\
        r \star \pi & \text{if }\ \beta < \alpha.
    \end{cases}
    \]
\end{definition}

\begin{proposition}
    Suppose that $\chi \in \Lambda$ and $\prec$ has been expanded to satisfy \Cref{definition:PropertyChi}. Then there exists a realizer $\theta \in \mathcal{R}$ such that for every $\alpha \leq \kappa$,
    \[
    \theta \Vdash ``\hat{\alpha} \text{ is a } \rlzin\text{-Trichotomous ordinal''}.
    \]
\end{proposition}

\begin{proof}
    It will suffice to show that
    \[
    \lambda b \lambdaapp \lambda c \lambdaapp \lambda u \lambdaapp \lambda v \lambdaapp \lambda w \lambdaapp \twoapp{\twoapp{\twoapp{\fapp{\app{\chi}{b}}{c}}{\app{u}{b}}}{v}}{\app{w}{c}} \Vdash \forall x^{\hat{\alpha}} \, \forall y^{\hat{\alpha}} ( x \notrlzin y \rightarrow (x \neq y \rightarrow (y \notrlzin x \rightarrow \perp))).
    \]
    Fix $\beta, \gamma \in \kappa$ and set $\eta = \lambda u \lambdaapp \lambda v \lambdaapp \lambda w \lambdaapp \twoapp{\twoapp{\twoapp{\fapp{\app{\chi}{\nu_\beta}}{\nu_\gamma}}{\app{u}{\nu_\beta}}}{v}}{\app{w}{\nu_\gamma}}$. Fix $\beta, \gamma \in \kappa$, $t \Vdash \hat{\beta} \rlzin \hat{\gamma}$, $s \Vdash \hat{\beta} \neq \hat{\gamma}$, $r \Vdash \hat{\gamma} \notrlzin \hat{\beta}$ and $\pi \in \Pi$. If $\beta < \gamma$ then $\eta \star t \stackapp s \stackapp r \stackapp \pi \succ \app{t}{\nu_\beta} \star \pi \succ t \star \nu_\beta \stackapp \pi \in \Perp$ since $(\hat{\beta}, \nu_\beta \stackapp \pi) \in \hat{\gamma}$. Similarly, if $\gamma < \beta$ then $\eta \star t \stackapp s \stackapp r \stackapp \pi \succ r \star \nu_\gamma \stackapp \pi \in \Perp$. Finally, if $\beta = \gamma$ then $\falsity{\hat{\beta} \neq \hat{\gamma}} = \Pi$ and therefore, $\eta \star t \stackapp s \stackapp r \stackapp \pi \succ s \star \pi \in \Perp$.
\end{proof}

\begin{remark} \label{remark:OmegaTrichotomous}
    If one restricts $\chi$ to only comparing indices of Church numbers then $\chi$ can in fact be constructed as a $\lambda$-term without the need for special instructions (see \Cref{section:ComparingNumerals} for the details). From this it follows that $\hat{\omega}$ is always a $\rlzin$-Trichotomous ordinal. 

    It is further shown in \cite{FontanellaGeoffroy2020} that any realizability model $\rlzmodel$ satisfies ``$\hat{n}$ is the $n$\textsuperscript{th} ordinal'' and ``$\hat{\omega}$ is the first limit ordinal''.
\end{remark}

\pagebreak[4]
\section{Implementing Pairing} \label{section:Pairing}

The next problem we want to address is how to discuss functions in the realizability model. There are two main issues associated with this problem; what do we mean by a function, and given a function in the ground model is it possible to ``lift'' it to a function in the realizability model. A very typical example of the question of lifting is when given a function $1_E \colon a \rightarrow 2$ defined by $1_E(x) = 1$ if $x \in E$ and $1_E(x) = 0$ otherwise. The lifting of such functions to $\rlzmodel$, in particular the trivial Boolean algebra operations on $\{0, 1\}$, are used to prove statements about the Boolean algebra $\fullname{2}$. In order to address these problems, we first need a method to implement unordered and ordered pairs.

Traditionally, the ordered pair of two sets $a$ and $b$, $( a, b )$, is defined to be $\{ \{a, b \}, \{ a \} \}$. However, this turns out to be a difficult definition to encode because of the difficulty in using disjunctions in the realizability model. Instead we shall first give an encoding for unordered pairs which uses the two distinct $\lambda_c$-terms $\underline{0}$ and $\underline{1}$ to distinguish the elements of the pairs. Then we shall implement Wiener's definition of ordered pairs, which is $\{ \{ \{ a \}, \emptyset \}, \{ \{ b \} \} \}$. This will turn out to work because it will assign different structures to $a$ from $b$, allowing us to identify which one is which. Note that the specific coding mechanism used is not important because internally within our realizability models we could easily move between different implementations of ordered pairs.

\subsection{Unordered Pairs}

\noindent We begin with Krivine's implementation of singletons in the realizability model from \cite{Krivine2012}. After this we will give an implementation of unordered pairs. Recall that in the proof of Pairing in $\rlzmodel$ the name for the pair of the sets $a$ and $b$ was $\{a, b\} \times \Pi$, however it will turn out to be beneficial to take a different implementation in order to be able to distinguish $a \notrlzin \{a, b\}$ from $b \notrlzin \{a, b\}$.

\begin{definition}
    For $a \in \rlzstr$, $\sing(a) \coloneqq \{ a \} \times \Pi$.
\end{definition}

\begin{theorem} \label{theorem:SingularFunction}
    The following are realizable in $\rlzmodel$:
    \begin{thmlist}
        \item \label{theorem:SingularFunctionItem1} $\forall x \forall y (\sing(x) = \sing(y) \inclusion x = y)$;
        \item \label{theorem:SingularFunctionItem2} $\forall x (\sing(x) \not\simeq \fullname{0})$;
        \item \label{theorem:SingularFunctionItem3} $\forall x \forall y (x \in \sing(y) \rightarrow x \simeq y)$;
        \item \label{theorem:SingularFunctionItem4} $\forall x \forall y (x \simeq y \rightarrow \sing(x) \simeq \sing(y))$;
        \item \label{theorem:SingularFunctionItem5} $\forall x \forall y (\sing(x) \simeq \sing(y) \rightarrow x \simeq y)$.
    \end{thmlist}
\end{theorem}

\begin{proof} \,
    \begin{enumerate}[label=\textit{\roman*}.]
    \item Fix $a, b \in \rlzstr$, we shall prove that $\identity \Vdash \sing(a) = \sing(b) \inclusion a = b$. If $\sing(a) \neq \sing(b)$ then $\falsity{\sing(a) = \sing(b) \inclusion a = b} = \emptyset$ and the result trivially follows. So suppose that $\sing(a) = \sing(b)$. Then we obviously have that $a = b$ and thus $\falsity{\sing(a) = \sing(b) \inclusion a = b} = \falsity{a = b} = \{ t \stackapp \pi \divline \pi \in \Pi, \, \forall \sigma \in \Pi (t \star \sigma \in \Perp)\}$. Therefore, $\identity$ realizes the desired statement. 

    \item Fix $a \in \rlzstr$, it will suffice to prove that $\lambda u \lambdaapp \app{u}{\identity} \Vdash \sing(a) \subseteq \fullname{0} \rightarrow \perp$. So suppose that $t \Vdash \sing(a) \subseteq \fullname{0}$ and $\pi \in \Pi$. Then, by the proof of Extensionality from \Cref{section:RealzingZFepsilon}, we have that $t \Vdash \forall z (z \not\in \fullname{0} \rightarrow z \notrlzin \sing(a))$. By \Cref{theorem:realizinguniversals}, this in particular gives $t \Vdash a \not\in \fullname{0} \rightarrow a \notrlzin \sing(a)$. Now $\falsity{a \notrlzin \sing(a)} = \Pi$ while $\falsity{a \not\in \fullname{0}} = \emptyset$, because $\fullname{0} = \emptyset$. Thus $t \star \identity \stackapp \pi \in \Perp$, and the result follows. 
    
    \item Fix $a, b \in \rlzstr$. It will suffice to prove that
    \[
    \identity \Vdash a \not\simeq b \rightarrow a \not\in \sing(b).
    \]
    So suppose that $t \Vdash a \not\simeq b$ and $\pi \in \falsity{a \not\in \sing(b)}$. Then $\pi = s \stackapp s' \stackapp \sigma$ where $(b, \sigma) \in \sing(b)$, $s \Vdash a \subseteq b$ and $s' \Vdash b \subseteq a$. But this gives us that $t \star s \stackapp s' \stackapp \sigma \in \Perp$ and thus $\identity \star t \stackapp \pi$ is also in $\Perp$. 

    \item Fix $a, b \in \rlzstr$, it will suffice to prove that
    \[
    \lambda u \lambdaapp \lambda v \lambdaapp \lambda w \lambdaapp \fapp{\app{w}{u}}{v} \Vdash a \subseteq b \rightarrow ( b \subseteq a \rightarrow (\sing(a) \subseteq \sing(b))).
    \]
    To do this, suppose that $t \Vdash a \subseteq b$, $s \Vdash b \subseteq a$, $(c, \pi) \in \sing(a)$ and $r \Vdash c \not\in \sing(b)$. Since $\sing(a) = \{a\} \times \Pi$, we must have that $c = a$ and therefore $r \Vdash a \not\in \sing(b)$. However, since $(b, \pi) \in \sing(b)$, $t \stackapp s \stackapp \pi \in \falsity{a \not\in \sing(b)}$. Thus, $r \star t \stackapp s \stackapp \pi \in \Perp$, completing the proof.

     \item Fix $a, b \in \rlzstr$, it will suffice to show that
    \[
    \lambda u \lambdaapp \app{\cc \,}{u} \Vdash \sing(a) \subseteq \sing(b) \rightarrow a \subseteq b.
    \]
    We also note that a very similar argument could be used to show that $\sing(a) \subseteq \sing(b) \rightarrow b \subseteq a$ is also realizable.

    So, fix $t \Vdash \sing(a) \subseteq \sing(b)$ and $\pi \in \falsity{a \subseteq b}$. Then $t \Vdash \forall z (z \not\in \sing(b) \rightarrow z \notrlzin \sing(a))$ so, in particular, $t \Vdash a \not\in \sing(b) \rightarrow a \notrlzin \sing(a)$. Now, by definition, $\falsity{a \not\in \sing(b)} =$ $\falsity{a \subseteq b \rightarrow (b \subseteq a \rightarrow b \notrlzin \sing(b))}$. But, by \Cref{theorem:SaveCommandandNegation}, $\saverlz{\pi} \Vdash a \subseteq b \rightarrow (b \subseteq a \rightarrow b \notrlzin \sing(b))$ and therefore $\app{\cc \,}{t} \star \pi \succ \cc \star t \stackapp \pi \succ t \star \saverlz{\pi} \stackapp \pi \in \Perp$.
  \end{enumerate} 
\end{proof}

\begin{definition}
    Let $a, b \in \rlzstr$. The \emph{unordered pair} of $a$ and $b$, $\up(a, b)$, is
    \[
    \up(a, b) \coloneqq \{ (a, \underline{0} \stackapp \pi) \divline \pi \in \Pi \} \cup \{ (b, \underline{1} \stackapp \pi) \divline \pi \in \Pi \}.
    \]
\end{definition}

\noindent We next show that $\up$ satisfies many of the basic properties one would expect an unordered pair to have in the realizability model. Namely; it is well-defined with respect to non-extensional equality, extensionally the ordered pair of $a$ and $b$ is the same as that of $b$ and $a$, and if $z$ is in the unordered pair of $a$ and $b$ then $z$ is either extensionally equal to $a$ or to $b$.

\begin{proposition} \label{theorem:UnorderedPairs}
    The following are realizable in $\rlzmodel$:
    \begin{thmlist}
        \item \label{theorem:UnorderedPairs1} $\forall x_1 \forall x_2 \forall y_1 \forall y_2 (\up(x_1, y_1) = \up(x_2, y_2) \inclusion (x_1 = x_2 \land y_1 = y_2))$;
        \item \label{theorem:UnorderedPairs2} $\forall x \forall y \forall z (z \rlzin \up(x,y) \rightarrow z \rlzin \up(y,x))$;
        \item \label{theorem:UnorderedPairs3} $\forall x \forall y (\up(x, y) \simeq \up(y, x))$;
        \item \label{theorem:UnorderedPairs4} $\forall x \forall y \forall z (z \in \up(x,y) \rightarrow (z \not\simeq y \rightarrow z \simeq x))$.
    \end{thmlist}
\end{proposition}

\begin{proof} \,
    \begin{enumerate}[label=\textit{\roman*}.]
    \item For the first claim, fix $a_1, a_2, b_1$ and $b_2$ in $\rlzstr$. If $\up(a_1, b_1) \neq \up(a_2, b_2)$ then 
    \[
    \falsity{\up(a_1, b_1) = \up(a_2, b_2) \inclusion (a_1 = a_2 \land b_1 = b_2)} = \emptyset
    \]
    and the result trivially follows. So suppose that $\up(a_1, b_1) = \up(a_2, b_2)$. Then it is obviously the case that $a_1 = a_2$ and $b_1 = b_2$ and thus $\identity \Vdash a_1 = b_1$ and $\identity \Vdash a_2 = b_2$, from which the result follows.

    \item Fix $a, b$ and $c$ in $\rlzstr$. We shall show that a realizer for the second statement is $\lambda u \lambdaapp \lambda i \lambdaapp \bigfapp{\inapp{i}{\app{\lambda v \lambdaapp u}{\underline{0}}}}{\app{u}{\underline{1}}}$. So, suppose that $t \Vdash c \rlzin \up(a,b)$ and $(c, \underline{i} \stackapp \pi) \in \up(b,a)$. Then either $i = 0$ and $c = b$ or $i = 1$ and $c = a$. In the first case we have $t \star \underline{1} \stackapp \pi \in \Perp$ while in the second case $t \star \underline{0} \stackapp \pi \in \Perp$. The conclusion then follows from the verification that 
    
\noindent $\bigfapp{\inapp{\underline{0}}{\app{\lambda v \lambdaapp t}{\underline{0}}}}{\app{t}{\underline{1}}} \star \pi \succ t \star \underline{1} \stackapp \pi$ and $\bigfapp{\inapp{\underline{1}}{\app{\lambda v \lambdaapp t}{\underline{0}}}}{\app{t}{\underline{1}}} \star \pi \succ t \star \underline{0} \stackapp \pi$.

    \item The third claim follows from the second claim alongside the observation that,  in $\ZFepsilon$, $\forall a \forall x (x \rlzin a \rightarrow x \in a)$.

    \item For the final claim, fix $a, b$ and $c$ in $\rlzstr$, $t \Vdash c \not\in \up(a,b) \rightarrow \perp$, $s \Vdash c \simeq b \rightarrow \perp$ and $\pi \in \falsity{c \simeq a}$. We shall show that 
    \[
    \theta_1 \coloneqq \lambda e_t \lambdaapp \lambda e_s \lambdaapp \inapp{cc \,}{\app{\lambda e_k \lambdaapp e_t}{\theta_0}}
    \]
    is a realizer for the desired statement, where
    \[
    \theta_0 \coloneqq \lambda e_q \lambdaapp \lambda e_r \lambdaapp \lambda i \lambdaapp \fapp{\fapp{\fapp{\app{i}{\lambda w \lambdaapp e_s}}{e_k}}{e_q}}{e_r}
    \]
    will give a realizer for $c \not\in \up(a,b)$. To do this, first observe that $\saverlz{\pi} \Vdash c \simeq a \rightarrow \perp$. Now, any element of $\falsity{c \not\in \up(a,b)}$ is either of the form $q \stackapp r \stackapp \underline{0} \stackapp \sigma$, where $q \Vdash c \subseteq a$ and $r \Vdash a \subseteq c$ or of the form $q \stackapp r \stackapp \underline{1} \stackapp \sigma$, where $q \Vdash c \subseteq b$ and $r \Vdash b \subseteq c$. In the first case, we have that $\saverlz{\pi} \star q \stackapp r \stackapp \sigma \in \Perp$ while in the second $s \star q \stackapp r \stackapp \sigma \in \Perp$. Next,
    \begin{align*}
        \theta_0[e_s \coloneqq s, e_k \coloneqq \saverlz{\pi}] \star q \stackapp r \stackapp \underline{0} \stackapp \sigma & \succ \fapp{\fapp{\fapp{\app{\underline{0}}{\lambda w \lambdaapp s}}{\saverlz{\pi}}}{q}}{r} \star \sigma \\ & \succ \lambda u \lambdaapp \lambda v \lambdaapp v \star (\lambda w \lambdaapp s) \stackapp \saverlz{\pi} \stackapp q \stackapp r \stackapp \sigma \succ \saverlz{\pi} \star q \stackapp r \stackapp \sigma.
    \end{align*}
    Similarly, since $\underline{1} = \lambda u \lambdaapp \lambda v \lambdaapp u(v)$, $\theta_0[e_s \coloneqq s, e_k \coloneqq \saverlz{\pi}] \star q \stackapp r \stackapp \underline{1} \stackapp \sigma \succ s \star q \stackapp r \stackapp \sigma$. Thus, we have that $\theta_0[e_s \coloneqq s, e_k \coloneqq \saverlz{\pi}] \Vdash c \not\in \up(a,b)$ which means that $t \star \theta_0[e_s \coloneqq s, e_k \coloneqq \saverlz{\pi}] \stackapp \pi \in \Perp$. Finally, 
    \[
    \theta_1 \star t \stackapp s \stackapp \pi \succ \cc \star (\lambda e_k \lambdaapp \app{t}{\theta_0[e_s \coloneqq s]}) \stackapp \pi \succ \lambda e_k \lambdaapp \app{t}{\theta_0[e_s \coloneqq s]} \star \saverlz{\pi} \stackapp \pi \succ t \star \theta_0[e_s \coloneqq s, e_k \coloneqq k] \stackapp \pi \in \Perp.
    \]
    \end{enumerate}
\end{proof}

\noindent Finally, with a more involved proof, we can show that $\up$ is compatible with extensional equivalence.

\begin{proposition} \label{theorem:UnorderedPairRespectsSimeq}
    It is realizable in $\rlzmodel$ that
    \[
    \forall x_1 \forall x_2 \forall y_1 \forall y_2 \, ( x_1 \simeq x_2 \rightarrow ( y_1 \simeq y_2 \rightarrow \up(x_1, y_1) \simeq \up(x_2, y_2) )).
    \]
\end{proposition}

\begin{proof}
    Fix $a_1, a_2, b_1$ and $b_2$ in $\rlzstr$ and let 
    \[
    \theta_2 \coloneqq \lambda e_t \lambdaapp \lambda e_s \lambdaapp \lambda e_r \lambdaapp \lambda i \lambdaapp \superfapp{\Bigfapp{\Biginapp{\underline{i}}{\lambda w \lambdaapp \twoapp{\app{\identity}{e_s}}{\app{\lambda u \lambdaapp u}{\underline{1}}}}} {\twoapp{\app{\identity}{e_t}}{\app{\lambda u \lambdaapp u}{\underline{0}}}}}{e_r}.
    \]
    It will suffice to prove that 
    \[
    \theta_2 \Vdash a_1 \simeq a_2 \rightarrow (b_1 \simeq b_2 \rightarrow \up(a_1, b_1) \subseteq \up(a_2, b_2)).
    \]
    First note that, by using the proof that Extensionality is realizable in $\rlzmodel$ from \Cref{section:RealzingZFepsilon}, we have that 
    \[
    \identity \Vdash \forall x \forall y \forall z \, (x \simeq y \rightarrow (y \rlzin z \rightarrow x \in z)).
    \]
    Also, $\falsity{\up(a_1, b_1) \subseteq \up(a_2, b_2)} = \bigcup_c \{ r \stackapp \underline{i} \stackapp \pi \divline (c, \underline{i} \stackapp \pi) \in \up(a_1, b_1), \, r \Vdash c \not\in \up(a_2, b_2) \}.$ \\
    
    \noindent Now, fix $t \Vdash a_1 \simeq a_2$, $s \Vdash b_1 \simeq b_2$, $(c, \underline{i} \stackapp \pi) \in \up(a_1, b_1)$ and $r \Vdash c \not\in \up(a_2, b_2)$. If $i = 0$ then we must have that $c = a_1$ and thus $r \Vdash a_1 \not\in \up(a_2, b_2)$. Now, $\falsity{a_2 \notrlzin \up(a_2, b_2)} = \{ \sigma \divline (a_2, \sigma) \in \up(a_2, b_2) \} = \{ \underline{0} \stackapp \sigma \divline \sigma \in \Pi \}$ so $\lambda u \lambdaapp u \underline{0} \Vdash a_2 \rlzin \up(a_2, b_2)$. Therefore, $\twoapp{\app{\identity}{t}}{\app{\lambda u \lambdaapp u}{\underline{0}}} \Vdash a_1 \in \up(a_2, b_2)$ and thus
    \[
    \twoapp{\app{\identity}{t}}{\app{\lambda u \lambdaapp u}{\underline{0}}} \star r \stackapp \pi \in \Perp.
    \]
    In the same way, if $i = 1$ then $r \Vdash b_1 \not\in \up(a_2, b_2)$ and $ \twoapp{\app{\identity}{s}}{\app{\lambda u \lambdaapp u}{\underline{1}}} \star r \stackapp \pi \in \Perp.$ Finally, the result follows from the computations
    \[
    \superfapp{\Bigfapp{\Biginapp{\underline{0}}{\lambda w \lambdaapp \twoapp{\app{\identity}{s}}{\app{\lambda u \lambdaapp u}{\underline{1}}}}} {\twoapp{\app{\identity}{t}}{\app{\lambda u \lambdaapp u}{\underline{0}}}}}{r} \star \pi \succ \twoapp{\app{\identity}{t}}{\app{\lambda u \lambdaapp u}{\underline{0}}} \star r \stackapp \pi \in \Perp
    \]
    and 
    \[
     \superfapp{\Bigfapp{\Biginapp{\underline{1}}{\lambda w \lambdaapp \twoapp{\app{\identity}{s}}{\app{\lambda u \lambdaapp u}{\underline{1}}}}} {\twoapp{\app{\identity}{t}}{\app{\lambda u \lambdaapp u}{\underline{0}}}}}{r} \star \pi \succ \twoapp{\app{\identity}{s}}{\app{\lambda u \lambdaapp u}{\underline{1}}} \star r \stackapp \pi \in \Perp
    \]
\end{proof}

\subsection{Ordered Pairs}

The previous analysis shows that $\up$ gives a reasonable interpretation of the unordered pair. The next objective is to define an ordered pair and repeat the analysis to show that it behaves as expected.

\begin{definition}
    Let $a, b \in \rlzstr$. The \emph{ordered pair} of $a$ and $b$, $\op(a, b)$, is
    \[
    \op(a, b) \coloneqq \up( \up(\sing(a), \fullname{0}), \sing(\sing(b))).
    \]
\end{definition}

The first thing we wish to deduce is that the ordered pair can distinguish between the first and the second element. Namely, if $\op(x_1, y_1) \simeq \op(x_2, y_2)$ then we should be able to realize that $x_1 \simeq x_2$ and $y_1 \simeq y_2$. In order to do this, we first need to prove a technical lemma. This lemma is the reason why we have chosen to using Wiener's definition of ordered pairs.

\begin{lemma} \label{theorem:UnorderedPairTechnical}
    For any $a, b, x \in \rlzstr$, the following are realizable in $\rlzmodel$:
    \begin{thmlist}
        \item \label{theorem:UnorderedPairTechnical1} $\up(\sing(x), \fullname{0}) \not\simeq \up(\sing(a), \fullname{0}) \rightarrow \up(\sing(x), \fullname{0}) \not\in \op(a, b)$;
        \item \label{theorem:UnorderedPairTechnical2} $\sing(\sing(x)) \not\simeq \sing(\sing(b)) \rightarrow \sing(\sing(x)) \not\in \op(a, b)$;
    \end{thmlist}
\end{lemma}

\begin{proof}
    First, note that 
    \begin{multline*}
        \falsity{\up(\sing(x), \fullname{0}) \not\in \op(a, b)} = \bigcup_c \{ s \stackapp s' \stackapp \underline{i} \stackapp \pi \divline (c, \underline{i} \stackapp \pi) \in \op(a, b), \\
        s \Vdash \up(\sing(x), \fullname{0}) \subseteq c, \, s' \Vdash c \subseteq \up(\sing(x), \fullname{0}) \}.
    \end{multline*}
    Suppose that $t \Vdash \up(\sing(x), \fullname{0}) \not\simeq \up(\sing(a), \fullname{0})$, $(c, \underline{i} \stackapp \pi) \in \op(a, b)$ and $s, s'$ realize the relevant statements. Then $c$ is either $\up(\sing(a), \fullname{0})$ or $\sing(\sing(b))$. If $c = \up(\sing(a), \fullname{0})$ then $i = 0$ and $t \star s \stackapp s' \stackapp \pi \in \Perp$.

    So, suppose that $c = \sing(\sing(b))$ and thus $i = 1$. Then 
    \[
    s \Vdash \forall z ( z \not\in \sing(\sing(b)) \rightarrow z \notrlzin \up(\sing(x), \fullname{0})).
    \]
    In particular, $s \Vdash \fullname{0} \not\in \sing(\sing(b)) \rightarrow \fullname{0} \notrlzin \up(\sing(x), \fullname{0})$. Then, since $\{ \underline{1} \stackapp \sigma \divline \sigma \in \Pi \} = \falsity{ \fullname{0} \notrlzin \up(\sing(x), \fullname{0})}$, $\app{s}{\underline{1}} \Vdash \fullname{0} \not\in \sing(\sing(b)) \rightarrow \perp$. Next, by \Cref{theorem:SingularFunctionItem3}, we have that there exists some term $r_s$ (dependent only on $s$) such that $r_s \Vdash \sing(b) \simeq \fullname{0}$. However, by \Cref{theorem:SingularFunctionItem2} there is a realizer $r_0$ such that $r_0 \Vdash \sing(b) \not\simeq \fullname{0}$ because $\fullname{0} = \emptyset$. Thus, $r_s \star r_0 \stackapp \pi \in \Perp$.

    The conclusion for the first claim follows from creating a realizer $r_1$ such that $r_1 \star t \stackapp s \stackapp s' \stackapp \underline{0} \stackapp \pi \succ t \star s \stackapp s' \stackapp \pi$ while $r_1 \star t \stackapp s \stackapp s' \stackapp \underline{1} \stackapp \pi \succ r_s \star r_0 \stackapp \pi$, which is relatively straightforward to do. \\

    \noindent The second claim follows be a very similar argument to the one above. Namely, we take \linebreak[4] $t \Vdash \sing(\sing(x)) \not\simeq \sing(\sing(b))$, $(c, \underline{i} \stackapp \pi) \in \op(a, b)$, $s \Vdash \sing(\sing(b)) \subseteq c$ and $s' \Vdash c \subseteq \sing(\sing(b))$. If $c = \sing(\sing(b))$, then $i = 1$ and $t \star s \stackapp s' \stackapp \pi \in \Perp$.

    So suppose that $c = \up(\sing(a), \fullname{0})$. Then $s' \Vdash \up(\sing(a), \fullname{0}) \subseteq \sing(\sing(x))$ and the same argument as above gives us that there is a realizer $r_{s'}$ such that $r_{s'} \star r_0 \stackapp \pi \in \Perp$. The conclusion then follows from creating a realizer $r_2$ such that $r_2 \star t \stackapp s \stackapp s' \stackapp \underline{1} \stackapp \pi \succ t \star s \stackapp s' \stackapp \pi$ while $r_2 \star t \stackapp s \stackapp s' \stackapp \underline{0} \stackapp \pi \succ r_{s'} \star r_0 \stackapp \pi$.
\end{proof}

\noindent The following two propositions show that the function $\op$ is also compatible with extensional equality and thus gives a reasonable interpretation of ordered pairs.

\begin{proposition} \label{theorem:OrderedPairRespectsSimeq}
    It is realizable in $\rlzmodel$ that
    \[
    \forall x_1 \forall x_2 \forall y_1 \forall y_2 \, ( \op(x_1, y_1) \simeq \op(x_2, y_2) \rightarrow (x_1 \simeq x_2 \land y_1 \simeq y_2)).
    \]
\end{proposition}

\begin{proof}
    Fix $a_1, b_1, a_2$ and $b_2$ in $\rlzstr$. It will suffice to show that if $\rlzmodel \Vdash \op(a_1, b_1) \subseteq \op(a_2, b_2)$ (that it to say, some $t$ realizes this statement) then
    \[
    \rlzmodel \Vdash a_1 \simeq a_2 \qquad \text{and} \qquad \rlzmodel \Vdash b_1 \simeq b_2. 
    \]
    For the first of these statements, suppose that $t \Vdash \forall z (z \not\in \op(a_2, b_2) \rightarrow z \notrlzin \op(a_1, b_1))$. In particular, this means that 
    \[
    t \Vdash \up(\sing(a_1), \fullname{0}) \not\in \op(a_2, b_2) \rightarrow \up(\sing(a_1), \fullname{0}) \notrlzin \op(a_1, b_1).
    \]
    Now, $\falsity{\up(\sing(a_1), \fullname{0}) \notrlzin \op(a_1, b_1)} = \{\underline{0} \stackapp \pi \divline \pi \in \Pi\}$ which means that $t \star s \stackapp \underline{0} \stackapp \sigma \in \Perp$ whenever $\sigma \in \Pi$ and $s \Vdash \up(\sing(a_1), \fullname{0}) \not\in \op(a_2, b_2)$. From this, we can conclude that
    \[
    \lambda u \lambdaapp \fapp{\app{t}{u}}{\underline{0}} \Vdash \up(\sing(a_1), \fullname{0}) \in \op(a_2, b_2).
    \]
    We now work internally in any model, $\rlzmodel$, of the realizable theory $\tf{T}$. By the contrapositive of \Cref{theorem:UnorderedPairTechnical1},
    \[
    \rlzmodel \Vdash \up(\sing(a_1), \fullname{0}) \simeq \up(\sing(a_2), \fullname{0})
    \]
    Next, we obviously have that $\rlzmodel \Vdash \sing(a_1) \rlzin \up(\sing(a_1), \fullname{0})$ so, by \Cref{theorem:ZFepsilonStatements}, $\rlzmodel \Vdash \sing(a_1) \in \up(\sing(a_2), \fullname{0})$. However, by \Cref{theorem:SingularFunctionItem2}, $\rlzmodel \Vdash \sing(a_1) \not\simeq \fullname{0}$ so, by \Cref{theorem:UnorderedPairs4}, $\rlzmodel \Vdash \sing(a_1) \simeq \sing(a_2)$. Finally, by \Cref{theorem:SingularFunctionItem5}, $\rlzmodel \Vdash a_1 \simeq a_2$. \\

    \noindent The second of these statements follows from a very similar argument: With the same assumption on $t$ we have that $t \star s \stackapp \underline{1} \stackapp \sigma \in \Perp$ whenever $\sigma \in \Pi$ and $s \Vdash \sing(\sing(b_1)) \not\in \op(a_2, b_2)$ and therefore
    \[
    \lambda u \lambdaapp \fapp{\app{t}{u}}{\underline{1}} \Vdash \sing(\sing(b_1)) \in \op(a_2, b_2).
    \]
    Therefore, by the contrapositive of \Cref{theorem:UnorderedPairTechnical2}, $\rlzmodel \Vdash \sing(\sing(b_1)) \simeq \sing(\sing(b_2))$. Finally, two applications of \Cref{theorem:SingularFunctionItem5} gives us that $\rlzmodel \Vdash b_1 \simeq b_2$.
\end{proof}

\begin{proposition} \label{theorem:SimeqRespectsOrderedPair}
    It is realizable in $\rlzmodel$ that
    \[
    \forall x_1 \forall x_2 \forall y_1 \forall y_2 \, ( x_1 \simeq x_2 \rightarrow (y_1 \simeq y_2 \rightarrow \op(x_1, y_1) \simeq \op(x_2, y_2))).
    \]
\end{proposition}

\begin{proof}
    The proof of this proposition will be very similar to proof of \Cref{theorem:UnorderedPairRespectsSimeq}. Fix $a_1, a_2, b_1$ and $b_2$ in $\rlzstr$. If will suffice to find a realizer $\theta$ such that
    \[
    \theta \Vdash a_1 \simeq a_2 \rightarrow (b_1 \simeq b_2 \rightarrow \op(a_1, b_1) \subseteq \op(a_2, b_2)).
    \]
    To do this, let $t \Vdash a_1 \simeq a_2$, $s \Vdash b_1 \simeq b_2$, $(c, \underline{i} \stackapp \pi) \in \op(a_1, b_1)$ and $r \Vdash c \not\in \op(a_2, b_2)$. If $i = 0$ then we must have that $c = \up(\sing(a_1), \fullname{0})$ and thus $r \Vdash \up(\sing(a_1), \fullname{0}) \not\in \op(a_2, b_2)$. Now, since $\falsity{\up(\sing(a_2), \fullname{0}) \notrlzin \op(a_2, b_2)} = \{ \underline{0} \stackapp \sigma \divline \sigma \in \Pi\}$, $\lambda u \lambdaapp \app{u}{\underline{0}} \Vdash \up(\sing(a_2), \fullname{0}) \rlzin \op(a_2, b_2)$. Next, using \Cref{theorem:SingularFunctionItem4} and \Cref{theorem:UnorderedPairRespectsSimeq} we have that 
    \[
    \rlzmodel \Vdash ( a_1 \simeq a_2 ) \rightarrow ( \sing(a_1) \simeq \sing(a_2) ) \rightarrow ( \up(\sing(a_1), \fullname{0}) \simeq \up(\sing(a_2), \fullname{0}) ),
    \]
    so we can find a realizer $v_t$ (dependent on $t$) such that
    $v_t \Vdash \up(\sing(a_1), \fullname{0})) \simeq \up(\sing(a_2, \fullname{0}))$.
    Hence, using the proof that Extensionality is realizable in $\rlzmodel$ from \Cref{section:RealzingZFepsilon}, we have that $\twoapp{\app{\identity}{v_t}}{\app{\lambda u \lambdaapp u}{\underline{0}}} \Vdash \up(\sing(a_1), \fullname{0}) \in \op(a_2, b_2)$ and thus
    \[
    \twoapp{\app{\identity}{v_t}}{\app{\lambda u \lambdaapp u}{\underline{0}}} \star r \stackapp \pi \in \Perp.
    \]
    In a similar way, if $i=1$ then $r \Vdash \sing(\sing(b_1)) \not\in \op(a_2, b_2)$ and we can find a realizer $v_s$ (dependent on $s$) such that $\twoapp{\app{\identity}{v_s}}{\app{\lambda u \lambdaapp u}{\underline{1}}} \Vdash \sing(\sing(b_1)) \in \op(a_2,b_2)$, giving $\twoapp{\app{\identity}{v_s}}{\app{\lambda u \lambdaapp u}{\underline{1}}} \star r \stackapp \pi \in \Perp$.

    One can see that the desired realizer is then
    
    \noindent $\lambda t \lambdaapp \lambda s \lambdaapp \lambda r \lambdaapp \lambda i \superfapp{\Bigfapp{\Biginapp{i}{\lambda w \lambdaapp \twoapp{\app{\identity}{v_s}}{\app{\lambda u \lambdaapp u}{\underline{1}}}}}{\twoapp{\app{\identity}{v_t}}{\app{\lambda u \lambdaapp u}{\underline{0}}}}}{r}$, which concludes the proof.
\end{proof}

\subsection{Cartesian Products}

\noindent To conclude this section, we see that this definition of ordered pairs allows us to define Cartesian products of recursive names.

\begin{definition}
    $c$ is said to be a \emph{Cartesian product} of $a$ and $b$ if:
    \begin{itemize}
        \item $\forall x \rlzin a \, \forall y \rlzin b \, (\op(x, y) \rlzin c)$;
        \item $\forall z (( \forall x \rlzin a \, \forall y \rlzin b \, \op(x, y) \not\simeq z) \rightarrow z \notrlzin c)$.
    \end{itemize}
\end{definition}

\noindent Observe that the second condition is equivalent to the claim that every element of $c$ is an ordered pair consisting of an element of $a$ and an element of $b$, that is $\forall z \rlzin c \, \exists x \rlzin a \, \exists y \rlzin b (z \simeq \op(x, y))$.

\begin{definition}
    Given sets $a$ and $b$ in the ground model define 
    
    \noindent $\ff{Cart}(a, b) \coloneqq \{ ( \op(\fullname{x}, \fullname{y}), \pi) \divline  x \in a, \, y \in b, \, \pi \in \Pi \}$.
\end{definition}

\begin{proposition}
    $\rlzmodel \Vdash \ff{Cart}(a, b) \text{ is the Cartesian product of } \fullname{a} \text{ and } \fullname{b}$.
\end{proposition}

\begin{proof}
    We shall first prove that
    \[
    \identity \Vdash \forall w_0^{\fullname{a}} \, \forall w_1^{\fullname{b}} \, (\op(w_0, w_1) \notrlzin \ff{Cart}(a, b) \rightarrow \perp).
    \]
    To this end, fix $x \in a$, $y \in b$, $t \Vdash \op(\fullname{x}, \fullname{y}) \notrlzin \ff{Cart}(a, b)$ and $\pi \in \Pi$. Then, by definition of \ff{Cart}, $\falsity{\op(\fullname{x}, \fullname{y}) \notrlzin \ff{Cart}(a, b)} = \Pi$ and thus $\identity \star t \stackapp \pi \succ t \star \pi \in \Perp$.

    For the second claim, we shall prove that
    \[
    \lambda u \lambdaapp \app{u}{\theta} \Vdash \forall z (( \forall w_0^{\fullname{a}} \, \forall w_1^{\fullname{b}} \, \op(w_0, w_1) \not\simeq z) \rightarrow z \notrlzin \ff{Cart}(a, b)),
    \]
    where $\theta$ is a fixed realizer such that $\theta \Vdash \forall x (x \simeq x)$ (for an example of such a realizer see \Cref{theorem:RealizerofSimeq}). 
    
    For this, fix $(\op(\fullname{x}, \fullname{y}), \pi) \in \ff{Cart}(a, b)$ and $t \Vdash \forall w_0^{\fullname{a}} \, \forall w_1^{\fullname{b}} \, \op(w_0, w_1) \not\simeq \op(\fullname{x}, \fullname{y})$. Then, in particular, we have that $t \Vdash \op(\fullname{x}, \fullname{y}) \not\simeq \op(\fullname{x}, \fullname{y})$ and thus $t \star \theta \stackapp \pi \in \Perp$.
\end{proof}

\section{Functions} \label{section:Functions}

In its most naive form, a function is a collection of ordered pairs such that whenever the first co-ordinates are equal then so are the second co-ordinates. Due to there being two versions of ``equality'' in $\ZFepsilon$ this gives rise to multiple different interpretations of what a function should be.

\begin{definition}[$\ZFepsilon$]
    A $\rlzin$-function with domain $a$ is a set $f$ satisfying:
    \begin{itemize}
        \item ($\rlzin$-totality) $\forall x \rlzin a \, \exists y \, \op(x, y) \rlzin f$;
        \item $\forall x \rlzin a \, \forall y, y' (\op(x, y) \rlzin f \land \op(x, y') \rlzin f \rightarrow y = y')$.
    \end{itemize}
    We will write $\rlzin\ff{-Func}(f)$ to indicate that $f$ is a $\rlzin$-function.
\end{definition}

\begin{definition}[$\ZFepsilon$]
    A $\in$-function with domain $a$ is a set $f$ satisfying:
    \begin{itemize}
        \item ($\in$-totality) $\forall x \in a \, \exists y \, \op(x, y) \in f$;
        \item $\forall x \in a \, \forall y, y' (\op(x, y) \in f \land \op(x, y') \in f \rightarrow y \simeq y')$.
    \end{itemize}
    We will write $\in\ff{-Func}(f)$ to indicate that $f$ is a $\in$-function.
\end{definition}

\begin{remark}
    $b$ is said to be a co-domain of a $\rlzin$-function $f$ if whenever $\op(x, y) \rlzin f$ then $y \rlzin b$ (and similarly for $\in$-functions). When we wish to identify such a $b$ we will alternatively write $f$ as $f \colon a \rightarrow b$, and it will be clear from the context what type of function $f$ is.
\end{remark}

\begin{remark}
    The definition of functions given here is perhaps not the most accurate definition because we have not required that every element of $f$ is an ordered pair and the first component of such an ordered pair is an element of $a$. This could easily be added to the definition, however it would make proving that a given set in $\rlzstr$ is a function more tedious without adding any useful information. Instead we note that for any reasonable definition of function we can restrict the set using $\rlzin$-Separation to only look at relevant ordered pairs.

    Also, on occasion it will be simpler to phrase a function as a binary relation satisfying the relevant translations of the conditions, for example $\forall x \rlzin a \, \forall y, y' (f(x, y) \land f(x, y') \rightarrow y = y')$. Clearly, given a function defined in this way we can find a set representing this function in the obvious way.
\end{remark}

\noindent In most of the conservativity results between \tf{ZF} and $\ZFepsilon$ one shows that the $\epsilon$ definition of a concept directly relates to the $\in$ definition. However, this is not the case here. For example, take two distinct sets $x$ and $y$ which are extensionally equal (such as $\{ (a, \pi) \divline \pi \in \Pi \}$ and $\{ (a, \underline{0} \stackapp \pi) \divline \pi \in \Pi \}$). Then there is a $\rlzin$-function $f \colon \{ x, y \} \times \Pi \rightarrow \fullname{2}$ given by sending $x$ to $\fullname{0}$ and $y$ to $\fullname{1}$. However, this does not translate to a $\in$-function because such a function would need to send $x$ to both $\fullname{0}$ and $\fullname{1}$ which are not extensionally equal sets. Instead, for conservativity results, we need an intermediate notion of function.

\begin{definition}[$\ZFepsilon$]
    An \emph{extensional} function with domain $a$ is a set $f$ satisfying:
    \begin{itemize}
        \item ($\rlzin$-totality) $\forall x \rlzin a \, \exists y \, \op(x, y) \rlzin f$;
        \item $\forall x, x' \rlzin a \, \forall y, y' (x \simeq x' \land \op(x, y) \rlzin f \land \op(x', y') \rlzin f \rightarrow y \simeq y')$.
    \end{itemize}
    We will write $\ff{Ext-Func}(f)$ to indicate that $f$ is an  extensional function.
\end{definition}

\begin{remark}
In the definitions of $\rlzin$-functions and $\in$-functions we only needed to assume that for every $x$ in $a$ and $y, y'$ if $f(x, y)$ and $f(x, y')$ both hold then $y$ and $y'$ are equal. On the other hand, in the definition of an extensional function we have needed to strengthen this condition to the claim that if $x \simeq x'$ and both $f(x, y)$ and $f(x', y')$ hold then $y\simeq y'$. It is relatively straight forward to see that the additional $x'$ is not needed in either of the first two definitions, however its inclusion for extensional functions is needed to ensure that they translate cleanly to $\in$-functions.
\end{remark}

\begin{observation}
    Suppose that $f$ is a $\in$-function with domain $a$. Let $g$ be the binary relation defined by 
    \[
    \op(x, y) \rlzin g \Longleftrightarrow x \rlzin a \land \op(x, y) \in f.
    \]
    Then $g$ is an extensional function with domain $a$. Similarly, if $g$ is an extensional function with domain $a$ and $f$ is the binary relation defined by
    \[
    \op(x, y) \in f \Longleftrightarrow \exists c \rlzin a ( x \simeq c \land \op(c, y) \rlzin g,
    \]
    then $f$ is a $\in$-function with domain $a$.
\end{observation}

\noindent As an example of conservativity, we briefly discuss a definition of \emph{cardinals} in $\ZFepsilon$.

\begin{definition}
    A relation $f$ with domain $a$ is said to be a $\rlzin$-injection if 
    \[
    \forall x, x' \rlzin a \, \forall y \, \big( \op(x, y) \rlzin f \land \op(x', y) \rlzin f \rightarrow x \simeq x' \big).
    \]
    A relation $f$ with domain $a$ is said to be a $\rlzin$-surjection on $b$ if
    \[
    \forall y \rlzin b \, \exists x \rlzin a \, \op(x, y) \rlzin f.
    \]
\end{definition}

\begin{definition}
    An ordinal $a$ is said to be a $\rlzin$\emph{-cardinal} if for every $b \rlzin a$ there is no extensional function which is a $\rlzin$-surjection of $b$ onto $a$.
\end{definition}

\begin{proposition}
    If $a$ is a $\rlzin$-cardinal, then it is a $\in$-cardinal.
\end{proposition}

\begin{proof}
    Let $a$ be an $\rlzin$-ordinal. We shall show that if $a$ is not a $\in$-cardinal then $a$ is not a $\rlzin$-cardinal. To this end, suppose that $b \in a$ and $f \colon b \rightarrow a$ was a $\in$-surjection. Fix $z \rlzin a$ such that $b \simeq z$ and define a binary relation $g$ by 
    \[
    \op(x, y) \rlzin g \Longleftrightarrow \exists c \in a \, \big( \op(x, c) \in f \land c \simeq y \big).
    \]
    We shall show that $g$ defines a $\rlzin$-surjection from $z$ onto $a$. There are three things to prove: that $g$ is $\rlzin$-total on $a$; $g$ is an extensional function; and $g$ is a $\rlzin$-surjection.

    For the first claim, take $x \rlzin z$. Then $x \in z$ and therefore $x \in b$, by \Cref{ZFepsilonProperty:ExtensionalityB}. Since $f$ is $\in$-total, we can fix $c \in a$ such that $\op(x, c) \in f$ and then fix $y \rlzin a$ such that $y \simeq c$. Hence, we have $\op(x,y) \rlzin g$.

    For the second claim, fix $x, x' \rlzin b$, $y, y' \rlzin a$ and suppose that $\op(x, y) \rlzin g$, $\op(x', y') \rlzin g$ and $x \simeq x'$. Fix $c, c' \in a$ such that $\op(x, c) \in f$, $\op(x',c') \in f$, $c \simeq y$ and $c' \simeq y'$. Since $f$ is a $\in$-function, $x, x' \in b$ and $x \simeq x'$, we must have that $c \simeq c'$. Therefore $y \simeq c \simeq c' \simeq y'$.

    For the final claim, fix $y \rlzin a$. Then $y \in a$ so, since $f$ is a $\in$-surjection, we can fix $t \in b$ such that $\op(t, y) \in f$ holds. Since $t \in b$ and $b \simeq z$, $t \in z$. Therefore, there exists some $d \rlzin z$ such that $t \simeq d$. Since $g$ is $\rlzin$-total, we can fix $y' \rlzin a$ such that $\op(d, y') \rlzin g$ and then, by definition, $c \in a$ such that $\op(d, c) \in f$ and $c \simeq y'$. Finally, since $f$ is a function and $t \simeq d$, $y \simeq c \simeq y'$, from which we can conclude that $\op(d, y) \rlzin g$ as required.
\end{proof}

\section{Lifting Functions} \label{section:LiftingFunctions}

Here we give a general framework to lift functions from the ground model to the realizability model, the idea being that they should extend the (natural interpretations of the) function from the ground model. For example, if $f \colon 2 \rightarrow 2$ is a ground model function then there should be some $\rlzin$-function $\lift{f} \colon \fullname{2} \rightarrow \fullname{2}$ such that $\rlzmodel \Vdash \lift{f}(\fullname{i}) = \fullname{f(i)}$ for $i \in \{0, 1\}$. 

\begin{definition}
    Suppose that $a, b \in \tf{V}$ and $f \colon a \rightarrow b$ is a function. We define the \emph{lift} of $f$, $\lift{f}$ as 
    \[
        \lift{f} \coloneqq \{ (\op(\fullname{c}, \fullname{f(c)}), \pi) \divline c \in a, \pi \in \Pi \}.
    \]    
\end{definition}

\noindent We begin by showing that $\lift{f}$ is a lift of the function $f$ in some precise sense:

\begin{proposition} For any $c$ in $a$,
    \[
    \identity \Vdash \forall z \big( \op(\fullname{c}, z) \rlzin \lift{f} \rightarrow \fullname{f(c)} = z \big) \text{ and } \identity \Vdash \forall z \big( \fullname{f(c)} = z \inclusion \op(\fullname{c}, z) \rlzin \lift{f} \big).
    \]
\end{proposition}

\begin{proof}
    For the first claim, suppose that $t \Vdash \op(\fullname{c}, z) \rlzin \lift{f}$, $s \Vdash \fullname{f(c)} \neq z$ and $\pi \in \Pi$. We shall show that $s \Vdash \op(\fullname{c}, z) \notrlzin \lift{f}$ from which it shall follow that $\identity \star t \stackapp s \stackapp \pi \succ t \star s \stackapp \pi \in \Perp$. But this follows from the observation that $\falsity{\op(\fullname{c}, z) \notrlzin \lift{f}} = \Pi$ if $z = \fullname{f(c)}$ and equals $\emptyset$ otherwise, which means that $\falsity{\op(\fullname{c}, z) \notrlzin \lift{f}} = \falsity{\fullname{f(c)} \neq z}$.

    For the second claim, if $\fullname{f(c)} \neq z$ then $\falsity{\fullname{f(c)} = z \inclusion \op(\fullname{c}, z) \rlzin \lift{f}} = \emptyset$ and thus $\identity$ trivially realizes this. Otherwise, $\fullname{f(c)} = z$ and $\falsity{\fullname{f(c)} = \fullname{f(c)} \inclusion \op(\fullname{c}, \fullname{f(c)}) \rlzin \lift{f}} = \falsity{\op(\fullname{c}, \fullname{f(c)}) \rlzin \lift{f}} = \{t \stackapp \pi \divline \pi \in \Pi, \forall \sigma \in \Pi (t \star \sigma \in \Perp) \}$. Therefore, for any such $t$ and $\pi$, $\identity \star t \stackapp \pi \in \Perp$.
\end{proof}

\begin{proposition}
    $\lambda u \lambdaapp \lambda v \lambdaapp \app{v}{u} \Vdash \forall x^{\fullname{a}} \forall y \forall y' (y \neq y' \rightarrow \op(x, y) \rlzin \lift{f} \rightarrow \op(x, y') \notrlzin \lift{f}).$ Therefore, $\lift{f}$ is realized to be a $\rlzin$-function with domain $\fullname{a}$.
\end{proposition}

\begin{proof}
    Fix $c \in a$, $d, d' \in \rlzstr$ and suppose that $t \Vdash d \neq d'$, $s \Vdash \op(\fullname{c}, d) \rlzin \lift{f}$ while $(\op(\fullname{c}, d'), \pi) \in \lift{f}$. By definition of $\lift{f}$, we must have that $d' = \fullname{f(c)}$. There are now two cases; $d \neq \fullname{f(c)}$ and $d = \fullname{f(c)}$.

    If $d \neq \fullname{f(c)}$ then $\falsity{\op(\fullname{c}, d) \notrlzin \lift{f}} = \emptyset$ and therefore $\falsity{\op(\fullname{c}, d) \rlzin \lift{f}} = \Lambda \star \Pi$. From this it follows that $\lambda u \lambdaapp \lambda v \lambdaapp \app{v}{u} \star t \stackapp s \stackapp \pi \succ s \star t \stackapp \pi \in \Perp$.

    On the other hand, if $d = \fullname{f(c)}$ then $\falsity{\op(\fullname{c}, d) \notrlzin \lift{f}} = \Pi = \falsity{d \neq d'}$. From this it follows that $t \stackapp \pi \in \falsity{\op(\fullname{c}, d) \rlzin \lift{f}}$ and therefore $\lambda u \lambdaapp \lambda v \lambdaapp \app{v}{u} \star t \stackapp s \stackapp \pi \succ s \star t \stackapp \pi \in \Perp$.
\end{proof}

\noindent Finally, we can see that $\fullname{b}$ is a co-domain of $\lift{f}$:

\begin{proposition} \label{theorem:LiftCodomain}
    $\identity \Vdash \forall x^{\fullname{a}} \forall y (y \notrlzin \fullname{b} \rightarrow \op(x, y) \notrlzin \lift{f})$.
\end{proposition}

\begin{proof}
    Fix $c \in a$, $d \in \rlzstr$ and suppose that $t \Vdash d \notrlzin \fullname{b}$ while $(\op(\fullname{c}, d), \pi) \in \lift{f}$. Then, by definition, we must have that $d = f(\fullname{c})$ which means that $\identity \Vdash d \rlzin \fullname{b}$. Therefore, we have $\identity \star t \stackapp \pi \in \Perp$.
\end{proof}

\subsection{Lifting Functions with Multiple Arguments}

\noindent If one were to try and lift a function $f \colon a_0 \times a_1 \rightarrow b$ using the previous methodology this would give us a function $\lift{f} \colon \fullname{a_0 \times a_1} \rightarrow b$ in $\rlzmodel$ rather than a more useful function with domain $\fullname{a_0} \times \fullname{a_1}$. However, this can be achieved by a slight alteration to the definition of the lift. Since the proof follows by the same argument as in the earlier propositions, we will omit it.

\begin{proposition} \label{theorem:LiftingFunctionOnProduct}
    Suppose that $f \colon a_0 \times a_1 \rightarrow b$ is a function. Let
    \[
    \lift{f} \colon \{ (\op(\op(\fullname{x_0}, \fullname{x_1}), \fullname{f(x_0, x_1)}), \pi) \divline x_0 \in a_0, x_1 \in a_1, \pi \in \Pi \}.
    \]
    Then $\rlzmodel \Vdash \lift{f} \colon \ff{Cart}(a_0, a_1) \rightarrow \fullname{b} \text{ is a } \rlzin\text{-function}$. Moreover, for any $x_0 \in a_0$ and $x_1 \in a_1$, 
    \begin{align*}
    \identity & \Vdash \forall z \big( \op(\op(\fullname{x_0}, \fullname{x_1}), z) \rlzin \lift{f} \rightarrow \fullname{f(x_0, x_1)} = z \big) \text{ and } \\ \identity & \Vdash \forall z \big( \fullname{f(x_0, x_1)} = z \inclusion \op(\op(\fullname{x_0}, \fullname{x_1}), z) \rlzin \lift{f} \big).
    \end{align*}
\end{proposition}

\noindent We recall here that, in $\rlzmodel$, $\ff{Cart}(a_0, a_1)$ is a name for the Cartesian product of $\fullname{a_0}$ and $\fullname{a_1}$. Therefore, unless there is room for confusion, we shall also refer to this set as the lift of the function $f$ and use the same notation because this is what we want to mean by lifting a function with domain $a_0 \times a_1$.

\begin{remark}
    The above result could be easily further generalised to allow for the lift of a function $f \colon a_0 \times \dots a_n \rightarrow b$ for any finite $n$ in the obvious way. 
\end{remark}

\begin{theorem} \label{theorem:fullname2BooleanAlgebra}
    $\fullname{2}$ is a Boolean algebra with binary operations $\Booleanor$ and $\Booleanand$ and unary operation $-$ given by the lift of the corresponding (trivial) operations on $\{0, 1\}$.   
\end{theorem}

\begin{proof}
    Formally, this is proven by showing that the above operations satisfy commutativity, associativity, distributivity, absorption and complementations. All of these are straightforward, and we briefly discuss one of the absorption cases as an example.
    \[
    \identity \Vdash \forall x^{\fullname{2}} \forall y^{\fullname{2}} \; (x \Booleanand (x \Booleanor y) = x).
    \]
    It suffices to prove that if $i, j \in \{0, 1\}$ then $\identity \Vdash \fullname{i} \Booleanand (\fullname{i} \Booleanor \fullname{j}) = \fullname{i}$. But, this is immediate from the fact that $\{0, 1\}$ is a Boolean algebra and the operations lift the operations on $\{0, 1\}$.
\end{proof}

\begin{remark}
    We would have liked to lift the function to an extensional function in the realizability model, however in general this does not seem possible to do. The issue is that there may not be a realizer $t$ such that, for any $a \neq b$, $t \Vdash \fullname{a} \not\simeq \fullname{b}$. 

    This becomes a problem when one wants to prove $\rlzmodel \Vdash \forall z (\op(\fullname{c}, z) \in \lift{f} \rightarrow \fullname{f(c)} \simeq z)$. If $\op(\fullname{c}, z) \in \lift{f}$ then we can find some $d \rlzin a$ for which $\op(\fullname{c}, z) \simeq \op(\fullname{d}, \fullname{f(d)})$. It will then be the case that, in $\rlzmodel$, $z \simeq \fullname{f(d)}$. The issue is that we cannot prove $\rlzmodel \Vdash \fullname{c} \simeq \fullname{d} \rightarrow \fullname{f(c)} \simeq \fullname{f(d)}$ and therefore it is unclear how one could conclude that $\rlzmodel \Vdash \fullname{f(c)} \simeq z$. 
\end{remark}

\subsection{Lifting Class Functions} \label{section:LiftingClassFunctions}

Formally, the implementations for singletons, unordered pairs and ordered pairs have only been defined externally for elements of $\rlzstr$. In this short section we outline how to internalise these functions so as to naturally view them as class functions over the structure. As mentioned in \Cref{section:AddingDefinedFunctions} this will then allow us to enrich the theory $\ZFepsilon$ by any $\mathcal{A}$-definable function in a way that allows us to internalise such functions with the theory. In order to do this, we need a general method to universally lift a function $f \colon \rlzstr \rightarrow \rlzstr$.

\begin{definition} \label{definition:ClassLift}
    Suppose that $f \colon \rlzstr \rightarrow \rlzstr$ is a function (as defined in the ground model). We define the \emph{universal lift} of $f$, $\lift{f}$, as 
    \[
    \lift{f} \coloneqq \{ (\op(a, f(a)), \pi) \divline a \in \rlzstr, \pi \in \Pi \}.
    \]
\end{definition}

\begin{theorem} \label{theorem:LiftingClassFunction}
    Suppose that $f \colon \rlzstr \rightarrow \rlzstr$. Then, in $\rlzmodel$, 
    \begin{thmlist}
        \item \label{theorem:ClassLiftIsLift} For every $a \in \rlzstr$, $\forall z \, (\op(a, z) \rlzin \lift{f}  \rightarrow z = f(a))$ and $\forall z \, (f(a) = z \imp \op(a, z) \rlzin \lift{f})$. That is, $\lift{f}$ agrees with $f$ on $\tf{N}$. 
        \item \label{theorem:LiftIsClassFunction} $\forall x \, \exists y \, \op(x, y) \rlzin \lift{f}$ and $\forall x \, \forall y \, \forall y' \, (\op(x, y) \rlzin \lift{f} \land \op(x, y') \rlzin \lift{f} \rightarrow y = y')$. That is, $\lift{f}$ is a class $\rlzin$-function.
    \end{thmlist}
\end{theorem}

Since the proof to this theorem is the same as for the lift of the set function $f \colon a \rightarrow b$, we will omit it.

\begin{remark}
    Using the analogous construction to \Cref{theorem:LiftingFunctionOnProduct}, it is also possible to lift any function $f \colon \rlzstr^n \rightarrow \rlzstr$ for any $n \in \omega$.
\end{remark}

\begin{theorem}
    The following functions universally lift to the realizability model and therefore can be viewed as internal functions in $\rlzmodel$.
    \begin{itemize} \setlength \itemsep{2pt}
        \item $\sing \colon \rlzstr \rightarrow \rlzstr$, $a \mapsto \{a\} \times \Pi$;
        \item $\up \colon \rlzstr \times \rlzstr \rightarrow \rlzstr$, $( a, b ) \mapsto \{ (a, \underline{0} \stackapp \pi ) \divline \pi \in \Pi \} \cup \{(b, \underline{1} \stackapp \pi) \divline \pi \in \Pi \}$;
        \item $\op \colon \rlzstr \times \rlzstr \rightarrow \rlzstr$, $( a, b ) \mapsto \up(\up(\sing(a), \fullname{0}), \sing(\sing(b)))$;
        \item $a \intcup b \colon \rlzstr \times \rlzstr \rightarrow$, $( a, b ) \mapsto \{ (x, \underline{0} \stackapp \pi) \divline (x, \pi) \in a \} \cup \{ (x, \underline{1} \stackapp \pi) \divline (x, \pi) \in b \}$. This function gives a canonical way to define the non-extensional union of two sets;
        \item $w \mathcal{P} \colon \rlzstr \rightarrow \rlzstr$, $a \mapsto \mathcal{P}(\ff{dom}(a) \times \Pi) \times \Pi$. This function gives a canonical weak power set for every set $a$.
    \end{itemize}
\end{theorem}

\begin{remark}[A note on unions] \phantom{a}

    It is easy to see that, for any $a, b \in \rlzstr$, $\lambda u \lambdaapp \app{u}{\underline{0}} \Vdash x \notrlzin a \intcup b \rightarrow x \notrlzin a$, $\lambda u \lambdaapp \app{u}{\underline{1}} \Vdash x \notrlzin a \intcup b \rightarrow x \notrlzin b$ and $\lambda u \lambdaapp \lambda v \lambdaapp \lambda i \lambdaapp \fapp{\app{i}{\lambda w \lambdaapp v}}{u} \Vdash x \notrlzin a \rightarrow ( x \notrlzin b \rightarrow x \notrlzin a \intcup b)$. Thus $\rlzmodel \Vdash \forall x (x \rlzin a \intcup b \leftrightarrow (x \rlzin a \lor x \rlzin b))$ so $\intcup$ does give a name for the non-extensional union of two sets.

    It would have been preferable to define the name for the union as simply $a \cup b$, using the true union in the ground model. One can see that $\rlzmodel \Vdash x \notrlzin a \cup b \rightarrow (x \notrlzin a \lor x \notrlzin b)$ and therefore $a \cup b$ contains the union of the two sets, the difficulty is in the reverse direction. This is because we run into the standard issue that from $(x, \pi) \in a \cup b$ we cannot decide from $\pi$ whether $(x, \pi) \in a$ or $(x, \pi) \in b$. 
\end{remark}

\subsection{Images of Functions} \label{section:FunctionImages}

Due to our coding of ordered pairs, it is cumbersome to deal with the image of a lifted function. To see the issue consider trying to solve
\[
\forall x \lift{f}(x) \rlzin a.
\]
Formally, to realize this we would need to realize something of the form
\[
\forall x ( \forall y (\op(x, y) \rlzin \lift{f} \rightarrow y \notrlzin a) \rightarrow \perp). 
\]
Therefore, we introduce an easier way to phrase such statements through a new falsity value.

\begin{definition}
    Suppose that $f \colon a \rightarrow b$ is a function in the ground model and let $\lift{f}$ be the standard name for the lift of $f$. Then, for any formula $\varphi(u, v, \overrightarrow{w})$ and $\overrightarrow{w} \in \rlzstr$,
    \[
    \falsity{ \varphi(x, \lift{f}(x), \overrightarrow{w})} = 
    \begin{cases}
        \falsity{\varphi(x, \fullname{d}, \overrightarrow{w})} & \text{ if } x = \fullname{c} \text{ and } f(c) = d \text{ for some } d \in b, \\
        \emptyset & \text{ otherwise.}
    \end{cases}
    \]
\end{definition}

We note that, since $f$ is a function, the above $d \in b$ is uniquely defined. We now prove that this set provides a way to interpret the image of $\lift{f}$. For simplicity, we will suppress the parameter $\overrightarrow{w}$.

\begin{lemma}
    For every formula $\varphi$,
    \[
    \rlzmodel \Vdash \forall x^{\fullname a} \forall y \big( \op(x, y) \rlzin \lift{f} \rightarrow (\varphi(x, y) \rightarrow \varphi(x, \lift{f}(x))) \big).
    \]
\end{lemma}

\begin{proof}
    It will suffice to prove that
    \[
    \identity \Vdash \forall x^{\fullname{a}} \forall y ( \neg \varphi(x, \lift{f}(x)) \rightarrow (\varphi(x, y) \rightarrow \op(x, y) \notrlzin \lift{f})).
    \]
    To do this, fix $c \in a$, $z \in \rlzstr$, and $d \in B$ such that $f(c) = d$. Then, $\falsity{\neg \varphi(\fullname{c}, \lift{f}(\fullname{c})} = \falsity{\neg \varphi(\fullname{c}, \fullname{d})}$. So, if $t \Vdash \neg \varphi(\fullname{c}, \fullname{d})$, $s \Vdash \varphi(\fullname{c}, z)$ and $(\op(\fullname{c}, z), \pi) \in \lift{f}$ then we must have that $z = \fullname{d}$. Hence, $\identity \star t \stackapp s \stackapp \pi \succ t \star s \stackapp \pi \in \Perp$.
\end{proof}

\begin{lemma}
    For every formula $\varphi$,
    \[
    \rlzmodel \Vdash \forall x^{\fullname{a}} \big( \varphi(x, \lift{f}(x)) \rightarrow \forall y (\op(x, y) \rlzin \lift{f} \rightarrow \varphi(x, y)) \big).
    \]
\end{lemma}

\begin{proof}
    It will suffice to prove that
    \[
    \lambda g \lambdaapp \lambda h \lambdaapp \app{h}{g} \Vdash \forall x^{\fullname{a}} (\varphi(x, \lift{f}(x)) \rightarrow \forall y (\neg \varphi(x, y) \rightarrow \op(x, y) \notrlzin \lift{f})).
    \]
    Again, fix $c \in a$ and $d \in b$ such that $f(c) = d$. Next, suppose that $t \Vdash \varphi(\fullname{c}, \fullname{d})$ and $s \stackapp \pi \in \falsity{ \forall y (\neg \varphi(\fullname{c}, y) \rightarrow \op(\fullname{c}, y) \notrlzin \lift{f})}$. Then, for some $z \in \rlzstr$, $s \Vdash \neg \varphi(\fullname{c}, z)$ and $(\op(\fullname{c}, z), \pi) \in \lift{f}$, from which it follows that $z = \fullname{d}$. Thus $\lambda g \lambdaapp \lambda h \lambdaapp \app{h}{g} \star t \stackapp s \stackapp \pi \succ s \star t \stackapp \pi \in \Perp$.
\end{proof}

As an example of this alternative interpretation we reprove \Cref{theorem:LiftCodomain} to show that $\fullname{b}$ is a codomain of $\lift{f}$.

\begin{proposition}
$\identity \Vdash \forall x^{\fullname{a}} (\lift{f}(x) \rlzin \fullname{b})$.
\end{proposition}

\begin{proof}
Fix $c \in a$, $t \Vdash \lift{f}(c) \notrlzin \fullname{b}$ and $\pi \in \Pi$. Then $\falsity{\lift{f}(c) \notrlzin \fullname{b}} = \falsity{\fullname{d} \notrlzin \fullname{b}}$ where $d = f(c)$. But, since $d \in b$, this set is equal to $\Pi$, from which it follows that $\identity \star t \stackapp \pi \succ t \star \pi \in \Perp$.
\end{proof}

\section{Non-extensional Axiom of Choice and Dependant Choice} \label{section:NEACandDC}

In this section we discuss the non-extensional Axiom of Choice, denoted \tf{NEAC}, which is a version of the Axiom of Choice that holds in many realizability models. Assuming full AC in \tf{V}, Theorem 4.18 of \cite{Krivine2012} shows that \tf{NEAC} holds in any realizability model produced from a countable realizability algebra which contains the additional instruction ``\emph{quote}''. A more generalised argument is also given in Section 4.2 of \cite{FontanellaGeoffroy2020}, which adds a generalisation to quote as well as an additional instruction\footnote{which essentially allows one to realize that $\widehat{\alpha + 1}$ is the successor of $\hat{\alpha}$} to obtain \tf{NEAC} in realizability models of uncountable size.

\begin{definition}
    Suppose that $\mathcal{A} = (\Lambda, \Pi, \prec, \Perp)$ is a countable realizability algebra and let $t \mapsto \eta_t$ be an enumeration of $\Lambda$ in order type $\omega$. Then $\mathcal{A}$ contains the instruction \emph{quote} if there exists a special instruction $\rlzfont{q} \in \Lambda$ such that for any $t, s \in \Lambda$ and $\pi \in \Pi$,
    \[
    \rlzfont{q} \star t \stackapp s \stackapp \pi \succ t \star \underline{\eta_s} \stackapp \pi,
    \]
    where $\underline{\eta_s}$ is Church numeral associated to $\eta_s \in \omega$.
\end{definition}

So let us fix a countable realizability algebra $\mathcal{A}$ and suppose that $\Lambda$ contains the instruction \emph{quote}. \tf{NEAC} is the statement that every binary relation can be refined into a (non-extensional) $\rlzin$-function. More formally,

\begin{definition}
    \tf{NEAC} is the statement:
    \[
    \forall z \, \exists f \, ( f \subseteq_{\rlzin} z \land \rlzin\ff{-Func}(f) \land \forall x \forall y \exists y' (\op(x, y) \rlzin z \rightarrow \op(x, y') \rlzin f)).
    \]
\end{definition}

\begin{observation}
    The extensional version of \tf{NEAC} (where $f$ is an extensional function) implies the Axiom of Choice. To see this, suppose $A$ were a family of non-empty sets and let $z$ be the binary relation defined by $(i, x) \in z$ if and only if $i \in A \land x \in i$. Now, if we can refine $z$ to an extensional function, $f$, this will be a choice function for $A$. This is because if $i, j \in A$ and $i \simeq j$ then $f(i) \simeq f(j)$.

    On the other hand, if we can only refine $z$ to a non-extensional $\rlzin$-function this will not suffice because we may have $i \simeq j$ but $f(i) \not\simeq f(j)$. Therefore, we have not chosen a unique extensional element of each set in $A$.
\end{observation}

\begin{lemma} \label{theorem:FunctionApproximating}
    Suppose that $\mathcal{A}$ is a countable realizability algebra and $\Lambda$ contains a special instruction $\rlzfont{q}$ satisfying the instruction quote. Then for any formula $\varphi(u, x_1, \dots, x_n)$ in $Fml_{\rlzin}$ there exists a class function $g_\varphi \colon \omega \times \rlzstr^n \rightarrow \rlzstr$ such that 
    \[
    \rlzmodel \Vdash \forall x_1, \dots, x_n (\exists u \, \varphi(u, x_1, \dots, x_n) \rightarrow \exists n^{\hat{\omega}} \varphi(\lift{g}_\varphi(n, x_1, \dots, x_n), x_1, \dots, x_n)),
    \]
    where $\lift{g}_\varphi$ is the lift of $g_\varphi$ to a $\rlzin$-class function as defined in \Cref{definition:ClassLift}.
\end{lemma}

\begin{proof}
    For simplicity, we will assume that $\varphi(u, x)$ has precisely two free variables. Fix an enumeration $(\nu_n \divline n \in \omega)$ of $\Lambda$ and let $t \mapsto \eta_t$ be the inverse of the enumeration (so $\nu_{\eta_t} = t$). 

    Given $n \in \omega$, set $P_n \coloneqq \{ \pi \in \Pi \divline \nu_n \star \underline{n} \stackapp \pi \not\in \Perp\}$. Now, using the Axiom of Choice, we can define a function $g_\varphi \colon \omega \times \rlzstr \rightarrow \rlzstr$ such that for any $a \in \rlzstr$ and $n \in \omega$, if $P_n \cap \falsity{\forall u \neg \varphi(u, a)} \neq \emptyset$ then $P_n \cap \falsity{\neg \varphi(g_\varphi(n, a), a)} \neq \emptyset$. Finally, let $\lift{g}_\varphi \coloneqq \{ (\op(\op(u,a), g_\varphi(u,a)), \pi) \divline u,a \in \rlzstr, \pi \in \Pi\}$ be the lift of $g_\varphi$ to a class $\rlzin$-function.  
    
    It will suffice to show that
    \[
    \lambda u \lambdaapp \fapp{\app{\rlzfont{q}}{u}}{u} \Vdash \forall x \, ( \forall n^{\hat{\omega}} \neg \varphi(\lift{g}_\varphi(n, x), x) \rightarrow \forall u \neg \varphi(u, x)). 
    \]
    Suppose, for sake of a contradiction, that this was not the case. Then we can fix $a \in \rlzstr$, \hbox{$t \Vdash \forall n^{\hat{\omega}} \neg \varphi(\lift{g}_\varphi(n, a), a)$} and $\pi \in \falsity{\forall u \neg \varphi(u, a)}$ such that $\lambda u \lambdaapp \fapp{\app{\rlzfont{q}}{u}}{u} \star t \stackapp \pi \in \Perp$. Next, observe that from this we obtain
    \[
    \fapp{\app{\rlzfont{q}}{u}}{u} \star t \stackapp \pi \succ \rlzfont{q} \star t \stackapp t \stackapp \pi \succ t \star \underline{\eta_t} \stackapp \pi \not\in \Perp.
    \]
    Therefore $\pi \in P_{\eta_t} \cap \falsity{\forall u \neg \varphi(u, a)}$. So, by the construction of $g_\varphi$, we can fix some $\sigma \in P_{\eta_t} \cap \falsity{\neg \varphi(g_\varphi(\eta_t, a), a)}$. But then we have the desired contradiction because $\sigma \in P_{\eta_t}$ implies that $t \star \underline{\eta_t} \stackapp \sigma \not\in \Perp$ while $\underline{\eta_t} \stackapp \sigma \in \falsity{\forall n^{\hat{\omega}} \neg \varphi(\lift{g}_\varphi(n, a), a)}$ gives us $t \star \underline{\eta_t} \stackapp \sigma \in \Perp$.
\end{proof}

\begin{lemma}
    Suppose that $\mathcal{A}$ is a countable realizability algebra and $\Lambda$ contains a special instruction $\rlzfont{q}$ satisfying the instruction quote. Then $\rlzmodel \Vdash \tf{NEAC}$.
\end{lemma}

\begin{proof}
    First, by \Cref{theorem:FunctionApproximating}, let $g \colon \omega \times \rlzstr^3 \rightarrow \rlzstr$ be a $\rlzin$-function such that
    \[
    \rlzmodel \Vdash \forall z \forall x (\exists u  (\op(x, u) \rlzin z) \rightarrow \exists n^{\hat{\omega}} (\op(x, \lift{g}(n, z, x)) \rlzin z)).
    \]
    Now, working within any model of the realized theory, since $\hat{\omega}$ is a Trichotomous $\rlzin$-ordinal (\Cref{remark:OmegaTrichotomous}) it satisfies the Least Ordinal Principle (\Cref{LeastOrdinalPrinciple}); if $\exists u (\op(x, u) \rlzin z)$ then $\exists n \in \hat{\omega}$ such that $\op(x, \lift{g}(n, z, x)) \rlzin z$ and for all $m < n$, $\op(x, \lift{g}(m, z, x)) \rlzin z$. 
    
    Therefore, we can define a function $f$ in $\rlzmodel$ such that $\op(x, y) \rlzin f$ if and only if $\op(x, y) \rlzin z$ and $y = \lift{g}(n, z, x)$ where $n$ is the minimal such element of $\hat{\omega}$ (whenever it exists). Clearly, we have $f \subseteq_{\rlzin} z$ and for any $x, y$ there exists some $y'$ such that if $\op(x, y) \rlzin z$ then $\op(x, y') \rlzin z$.

    To show that $f$ is a $\rlzin$-function, suppose that $\op(x, y) \rlzin f$ and $\op(x, y') \rlzin f$. Then $y = \lift{g}(n, z, x)$ and $y' = \lift{g}(n', z, x)$ for some $n, n' \rlzin \hat{\omega}$ both of which are minimal. Since they are both minimal and $\hat{\omega}$ is extensionally an ordinal, in the extensional structure it must be the case that $n \simeq n'$ and therefore $n = n'$ by \Cref{theorem:HatOrdinalsHaveUniqueElements}.
\end{proof}

\begin{theorem}
    $\ZFepsilon \vdash \tf{NEAC} \rightarrow \tf{DC}_\in$. Hence, if $\mathcal{A}$ is a countable realizability algebra and $\Lambda$ contains a special instruction $\rlzfont{q}$ satisfying the instruction quote, then the extensional version of \tf{DC} is realized in $\rlzmodel$.
\end{theorem}

\begin{proof}
    We take $\tf{DC}_\in$ in the following form:
    \begin{multline*}
        \forall a \forall R (\forall x \in a \exists y \in a \, \op(x, y) \in R \\
        \rightarrow \exists g (\ff{Ext-Func}(g) \land g \colon \hat{\omega} \rightarrow a \land \forall n \in \hat{\omega} \, \op(g(n), g(n+1)) \in R)).
    \end{multline*}
    Fixing $a$ and $R$, suppose that $\forall x \in a \exists y \in a \, \op(x, y) \in R$ and set 
    \[
    z \coloneqq \bigcup_{x \in a} \{ \op(x, y) \divline y \in a \land \op(x, y) \in R \}.
    \]
    By \tf{NEAC} we can fix a non-extensional function $f \subseteq_{\rlzin} z$ which refines the relation $z$. Now, fix $x \rlzin a$. Then we can define a function $g \colon \hat{\omega} \rightarrow a$ recursively by setting $g(0) = x$ and $g(n+1)$ to be the unique $y$ such that $\op(g(n), y) \rlzin f$, where we observe that uniqueness follows from $f$ being a non-extensional function. To see that $g$ is the required extensional function, we first have that $g$ is clearly a total function on $\hat{\omega}$ since $f \subseteq_{\rlzin} z$. Now, fix $n, n' \rlzin \hat{\omega}$ and $y, y' \rlzin a$ and suppose that $n \simeq n'$, $\op(n, y) \rlzin g$ and $\op(n', y') \rlzin g$. Then, by \Cref{theorem:HatOrdinalsHaveUniqueElements}, we must have that $n = n'$ and therefore $y = y'$. In particular, this means that $y \simeq y'$, so $g$ is indeed extensional.
\end{proof}

\section{Non-trivial Realizability Models} \label{section:Reish2Size4}

Here we repeat the analysis from Section 4 of \cite{Krivine2018} to show that it is possible to produce a non-trivial realizability model such that the Boolean algebra $\fullname{2}$ has exactly $4$ elements. To produce this model we will use the first three Church numerals; $\underline{0} = \lambda u \lambdaapp \lambda v \lambdaapp v$, $\underline{1} = \lambda u \lambdaapp \lambda v \lambdaapp \app{u}{v}$ and $\underline{2} = \lambda u \lambdaapp \lambda v \lambdaapp \twoapp{\app{\underline{1}}{u}}{\app{u}{v}} = \lambda u \lambdaapp \lambda v \lambdaapp \inapp{u}{\app{u}{v}}$. We begin with the observation by Krivine that while $\fullname{2}$ may contain $\rlzin$-elements which are neither $\fullname{0}$ or $\fullname{1}$, all such sets are empty.

\begin{proposition} \label{theorem:Subsetsof2}
    $\identity \Vdash \forall x ^{\fullname{2}} \forall y (x \neq \fullname{1} \rightarrow y \notrlzin x)$.
\end{proposition}

\begin{proof}
    By \Cref{theorem:RealizingBoundedUniversals} it suffices to prove the claim for $x = \fullname{0}$ and $x = \fullname{1}$. Firstly, if $x = \fullname{0}$ then $\falsity{a \notrlzin \fullname{0}} = \emptyset$ for any $a \in \rlzstr$ and thus $\falsity{\forall y(\fullname{0} \neq \fullname{1} \rightarrow y \notrlzin \fullname{0})} = \emptyset$. Therefore, in this case any term realizes the statement. On the other hand, if $x = \fullname{1}$ then $\falsity{\fullname{1} \neq \fullname{1}} = \Pi$. Therefore, for any $a \in \rlzstr$, if $t \Vdash \fullname{1} \neq \fullname{1}$ and $(a, \pi) \in \fullname{1}$, $\identity \star t \stackapp \pi \in \Pi$. Thus $\identity$ realizes the statement.
\end{proof}

\begin{definition}
    Say that $(\bar{\Lambda}, \bar{\Pi})$ is \emph{generated} by $(T, X)$ if $\bar{\Lambda}$ and $\bar{\Pi}$ are the smallest classes satisfying:
    \begin{itemize} \setlength \itemsep{0pt}
        \item $\forall t \in T \; t \in \bar{\Lambda}$,
        \item $\forall \omega \in X \; \omega \in \bar{\Pi}$,
        \item Every variable is in $\bar{\Lambda}$,
        \item If $t, s \in \bar{\Lambda}$ then $\app{t}{s} \in \bar{\Lambda}$,
        \item If $u$ is a variable and $t \in \bar{\Lambda}$ then $\lambda u \lambdaapp t \in \bar{\Lambda}$,
        \item $\cc \in \bar{\Lambda}$,
        \item $\forall \pi \in \bar{\Pi} \; \saverlz{\pi} \in \bar{\Lambda}$,
        \item If $t \in \bar{\Lambda}$ and $\pi \in \bar{\Pi}$ then $t \stackapp \pi \in \bar{\Pi}$.
    \end{itemize}
    Namely, $(\bar{\Lambda}, \bar{\Pi})$ is generated by $(T,X)$ if the collection of special instructions is contained in $T$ and the collection of stack constants is contained in $X$.
\end{definition}

\medskip

\noindent We consider the realizability algebra $\mathcal{A} = (\Lambda, \Pi, \prec, \Perp)$ given as follows:

\begin{definition} \label{defn:RealizabilityAlgebraGimel2is4} \,
\begin{itemize}
    \item $(\Lambda, \Pi)$ is generated by $(\{\rlzfont{d}\}, \{\pi^0, \pi^1 \})$. That is, there is exactly one special instruction, $\rlzfont{d}$ and exactly two stack constants, $\pi^0$ and $\pi^1$.
    \item For $i \in \{0,1\}$ we let $(\Lambda^i, \Pi^i)$ be the sets generated by $(\{\rlzfont{d}\}, \{\pi^i\})$.
    \item For $i, j \in \{0,1\}$ we define $\Perp^i_j$ to be the least set $P \subseteq \Lambda^i \star \Pi^i$ such that:
    \begin{enumerate}
        \item $\rlzfont{d} \star \underline{j} \stackapp \pi \in P$ for every $\pi \in \Pi^i$,
        \item If $t \star \pi \in \Lambda^i \star \Pi^i$ and $s \star \sigma \in P$ then $t \star \pi \succ s \star \sigma \Rightarrow t \star \pi \in P$,
        \item If at least two out of the three processes $t \star \pi$, $s \star \pi$, $r \star \pi$ are in $P$ then $\rlzfont{d} \star \underline{2} \stackapp t \stackapp s \stackapp r \stackapp \pi \in P$.
    \end{enumerate}
    \item $\Perp$ is defined by $(\Lambda \star \Pi) \setminus \Perp = (\Lambda^0 \star \Pi^0) \setminus \Perp^0_0 \, \cup \, (\Lambda^1 \star \Pi^1) \setminus \Perp^1_1$. That is, a process $t \star \pi$ is in $\Perp$ if and only if either:
    \begin{enumerate}
        \item $t \star \pi \in \Perp^0_0 \cup \Perp^1_1$,
        \item $t \star \pi \in (\Lambda \star \Pi) \setminus (\Lambda^0 \star \Pi^0 \cup \Lambda^1 \star \Pi^1)$, i.e. both stack constants $\pi^0$ and $\pi^1$ appear in $t \star \pi$.
    \end{enumerate}
\end{itemize}
\end{definition}

\begin{remark} \label{remark:Gimel2Algebra}
    Since there is no additional requirement on $\rlzfont{d}$ in terms of how it interacts with the pre-order, 
    \[
    \rlzfont{d} \star \pi \succ s \star \sigma \quad \Longleftrightarrow \quad \rlzfont{d} \star \pi = s \star \sigma.
    \]
    Also, by the minimality requirement on the ordering, for any processes $t \star \pi$ and $s \star \sigma$, if $t \star \pi \succ s \star \sigma$ then there must exist a finite sequence 
    \[
    t \star \pi = t_0 \star \pi_0 \succ_1 \dots \succ_1 t_k \star \pi_k = s \star \sigma
    \]
    of distinct processes such that each reduction $t_n \star \pi_n \succ_1 t_{n+1} \star \pi_{n+1}$ is a one-step evaluation given by one of the defining constraints on $\succ_1$. In this case the only restrains on $\succ_1$ are the four basic requirements of push, grab, save, and restore. Moreover, by inspection, it is clear that only one of these can give a possible reduction of a given term and thus this finite sequence is uniquely determined. From this we can conclude that (again, for this realizability algebra): 
    \begin{itemize}
        \item If $t \star \pi \succ s \star \sigma$ then there is no infinite sequence of distinct processes $(t_n \star \pi_n \divline n \in \omega)$ such that $t \star \pi = t_0 \star \pi_0 \succ t_1 \star \pi_1 \succ \dots \succ t_n \star \pi_n \succ \dots \succ s \star \sigma$.
        \item If $t \star \pi \succ s \star \sigma$ and $t \star \pi \succ r \star \tau$ then either $s \star \sigma \succ r \star \tau$ or $r \star \tau \succ s \star \sigma$.
    \end{itemize}
\end{remark}

\noindent Due to the nature of the structure, there is an obvious isomorphism between $\Lambda^0 \star \Pi^0$ and $\Lambda^1 \star \Pi^1$, which is the one generated by sending $\pi^0$ to $\pi^1$. Namely, define $\Xi \coloneqq \Lambda^0 \cup \Pi^0 \rightarrow \Lambda^1 \cup \Pi^1$ recursively by the following rules:

\pagebreak[3]
\begin{itemize}
    \item $\Xi(\pi^0) = \pi^1$;
    \item If $u$ is a variable, then $\Xi(u) = u$;
    \item If $u$ is a variable and $t \in \Lambda^0$, then $\Xi(\lambda u \lambdaapp t) = \lambda u \lambdaapp \Xi(t)$;
    \item If $u$ is a variable and $t, s \in \Lambda^0$, then $\Xi(t[u \coloneqq s]) = \Xi(t)[u \coloneqq \Xi(s)]$;
    \item If $s, t \in \Lambda^0$, then $\Xi(\app{t}{s}) = \inapp{\Xi(t)}{\Xi(s)}$;
    \item $\Xi(\cc) = \cc$ and $\Xi(\rlzfont{d}) = \rlzfont{d}$;
    \item If $\pi \in \Pi^0$, then $\Xi(\saverlz{\pi}) = \saverlz{\Xi(\pi)}$;
    \item If $t \in \Lambda^0$, and $\pi \in \Pi^0$ then $\Xi(t \stackapp \pi) = \Xi(t) \stackapp \Xi(\pi)$.
\end{itemize}
This then extends to $\Xi \colon \Lambda^0 \star \Pi^0 \rightarrow \Lambda^1 \star \Pi^1$ by $\Xi(t \star \pi) = \Xi(t) \star \Xi(\pi)$. We now see that this map provides an isomorphism between the structures.

\begin{lemma} \label{theorem:XiPreservesOrdering}
For all $t \star \pi, s \star \sigma \in \Lambda^0 \star \Pi^0$, $s \star \sigma \succ t \star \pi$ if and only if $\Xi(s \star \sigma) \succ \Xi(t \star \pi)$.
\end{lemma}

\begin{proof}
    It suffices to prove the equivalence for the four basic constraints, which follow from the following observations:

    $\Xi \big( \app{t}{s} \star \pi \big) = \Xi(\app{t}{s}) \star \Xi(\pi) =  \inapp{\Xi(t)}{\Xi(s)} \star \Xi(\pi) \succ \Xi(t) \star \Xi(s) \stackapp \Xi(\pi) = \Xi(t \star s \stackapp \pi)$.

    $\Xi(\lambda u \lambdaapp t \star s \stackapp \pi) = \lambda u \lambdaapp \Xi(t) \star \Xi(s) \stackapp \Xi(\pi) \succ \Xi(t)[u \coloneqq \Xi(s)] \star \Xi(\pi) = \Xi(t[u \coloneqq s] \star \pi)$.

    $\Xi(\cc \star t \stackapp \pi) = \cc \star \Xi(t) \stackapp \Xi(\pi) \succ \Xi(t) \star \saverlz{\Xi(\pi)} \stackapp \Xi(\pi) = \Xi(t \star \saverlz{\pi} \stackapp \pi)$.

    $\Xi(\saverlz{\sigma} \star t \stackapp \pi) = \saverlz{\Xi(\sigma)} \star \Xi(t) \stackapp \Xi(\pi) \succ \Xi(t) \star \Xi(\sigma) = \Xi(t \star \sigma)$.    
\end{proof}

\begin{lemma} \label{theorem:PreservationOfPoles}
For $j \in \{0, 1\}$, $\Xi\pointwise\Perp^0_j = \Perp^1_j$.
\end{lemma}

\begin{proof}
    Fix $j \in \{0, 1\}$. Let $B^i \coloneqq \{ \rlzfont{d} \star \underline{j} \stackapp \pi \divline \pi \in \Pi^i \}$, and for any collection of processes $X$ let 
    \[
    G_0(X) \coloneqq \{ t \star \pi \divline \exists s \star \sigma \in X \; (t \star \pi \succ s \star \sigma) \}
    \]
    and
    \[
    G_1(X) \coloneqq \{ \rlzfont{d} \star \underline{2} \stackapp t \stackapp s \stackapp r \stackapp \pi \divline \text{at least two of } t \star \pi, s \star \sigma \text{ and } r \star \pi \text{ are in } X \}.
    \]
    It is then clear by the construction of $\Perp^i_j$ that if we set $T^i_0 = B^i$ and $T^i_{n+1} = G_0(T^i_n) \cup G_1(T^i_n)$ then $\Perp^i_j =  \bigcup_{n \in \omega} T^i_n$. We shall inductively show that $\Xi\pointwise T^0_n = T^1_n$ and thus $\Xi\pointwise\Perp^0_j = \Perp^1_j$. For this it suffices to only prove that $\Xi\pointwise T^0_n \subseteq T^1_n$, with the reverse inclusion following by considering the inverse of $\Xi$.

    Take $\rlzfont{d} \star \underline{j} \stackapp \pi \in T^0_0$. Then, since $\underline{j}$ is a $\lambda_c$-term which contains no appearance of $\pi^0$ or $\pi^1$,  $\Xi(\rlzfont{d} \star \underline{j} \stackapp \pi) = \rlzfont{d} \star \underline{j} \stackapp \Xi(\pi) \in T^1_0$. 

    Next, suppose that $t \star \pi \in G_0(T^0_n)$ and fix $s \star \sigma \in T^0_n$ such that $t \star \pi \succ s \star \sigma$. Then, by the inductive hypothesis, $\Xi(s \star \sigma) \in T^1_n$ and, by \Cref{theorem:XiPreservesOrdering}, $\Xi(t \star \pi) \succ \Xi(s \star \sigma)$. Thus, $\Xi(t \star \pi) \in G_0(T^1_n).$

    Finally, suppose that $\rlzfont{d} \star \underline{2} \stackapp t \stackapp s \stackapp r \stackapp \pi \in G_1(T^0_n)$. Without loss of generality, we shall assume that $t \star \pi$ and $s \star \pi$ are in $T^0_n$. Then, by the inductive hypothesis, $\Xi(t \star \pi) = \Xi(t) \star \Xi(\pi)$ and $\Xi(s \star \sigma) = \Xi(s) \star \Xi(\sigma)$ are in $T^1_n$. Thus, $\Xi(\rlzfont{d} \star \underline{2} \stackapp t \stackapp s \stackapp r \stackapp \pi) = \rlzfont{d} \star \underline{2} \stackapp \Xi(t) \stackapp \Xi(s) \stackapp \Xi(r) \stackapp \Xi(\pi) \in G_1(T^1_n)$.
\end{proof}

\begin{lemma} \label{theorem:Gimel2AlgebraReductionClosure}
    If $t \star \pi \in \Perp^i_j$ and $t \star \pi \succ s \star \sigma$, then $s \star \sigma \in \Perp^i_j$. 
\end{lemma}

\begin{proof}
    Suppose for a contradiction that $t \star \pi \succ s \star \sigma$ and $t \star \pi \in \Perp^i_j$ but $s \star \sigma \not\in \Perp^i_j$.
    Using \Cref{remark:Gimel2Algebra}, without loss of generality we can suppose that $t \star \pi \succ s \star \sigma$ by exactly one of the basic reduction steps. We shall show that $Q = \Perp^i_j \setminus \{ t \star \pi\}$ satisfies the three conditions in the definition of $\Perp^i_j$, contradicting the minimality requirement.
    \begin{enumerate}
        \item Take $\tau \in \Pi^i$. If $t \star \pi = \rlzfont{d} \star \underline{j} \stackapp \tau$ then $\rlzfont{d} \star \underline{j} \stackapp \tau \succ s \star \sigma$. Therefore, again by the remark, $s \star \sigma = \rlzfont{d} \star \underline{j} \stackapp \pi \in \Perp^i_j$, contradicting the assumption on $s \star \sigma$. Thus, $\rlzfont{d} \star \underline{j} \stackapp \tau \in Q$.
        \item Assume that $r \star \tau \in \Lambda^i \star \Pi^i$, $r' \star \tau' \in Q$ and $r \star \tau \succ r' \star \tau'$. Then $r \star \tau \in \Perp^i_j$. Now, if $r \star \tau = t \star \pi$ then $t \star \pi \succ r' \star \tau'$ and $t \star \pi \succ s \star \sigma$. Therefore, since the second one of these was assumed to be a one step reduction, we must have that $s \star \sigma \succ r' \star \tau'$. Therefore, by the second condition in the definition of $\Perp^i_j$, $s \star \sigma \in \Perp^i_j$, which is again a contradiction. Thus, $r \star \tau \in Q$.
        \item Suppose at least two of $r \star \tau$, $u \star \tau$ and $v \star \tau$ were in $Q$. Then $\rlzfont{d} \star \underline{2} \stackapp r \stackapp u \stackapp v \stackapp \tau \in \Perp^i_j$. Now, suppose that $\rlzfont{d} \star \underline{2} \stackapp r \stackapp u \stackapp v \stackapp \tau = t \stackapp \pi$. Then $\rlzfont{d} \star \underline{2} \stackapp r \stackapp u \stackapp v \stackapp \tau \succ s \star \sigma$ and thus $s \star \sigma = \rlzfont{d} \star \underline{2} \stackapp r \stackapp u \stackapp v \stackapp \tau \in \Perp^i_j$, which is a contradiction. Thus, $\rlzfont{d} \star \underline{2} \stackapp r \stackapp u \stackapp v \stackapp \tau \in Q$.
    \end{enumerate}
\end{proof}

\begin{lemma} \label{theorem:Gimel2AlgebraIntersectingPerp}
    For $i \in \{0, 1\}$, $\Perp^i_0 \cap \Perp^i_1 = \emptyset$.
\end{lemma}

\begin{proof}
    We first show that for any $\pi \in \Pi^i$, $\Perp^i_1 \setminus \{\rlzfont{d} \star \underline{0} \stackapp \pi\}$ satisfies the three conditions for the definition of $\Perp^i_1$ and therefore, by the minimality requirement, $\rlzfont{d} \star \underline{0} \stackapp \pi \not\in \Perp^i_1$.
    \begin{itemize}
        \item For any $\sigma \in \Pi^i$, $\rlzfont{d} \star \underline{1} \stackapp \sigma \in \Perp^i_1$ and $\rlzfont{d} \star \underline{1} \stackapp \sigma \neq \rlzfont{d} \star \underline{0} \stackapp \pi$, so $\rlzfont{d} \star \underline{1} \stackapp \sigma$ is in the set.
        \item Suppose $s \star \sigma \in \Lambda^i \star \Pi^i$, $r \star \tau \in \Perp^i_1 \setminus \{\rlzfont{d} \star \underline{0} \stackapp \pi \}$ and $s \star \sigma \succ r \star \tau$. Then $s \star \sigma \in \Perp^i_1$ so, if $s \star \sigma = \rlzfont{d} \star \underline{0} \stackapp \pi$, then $\rlzfont{d} \star \underline{0} \stackapp \pi \succ r \star \tau$ and therefore $d \star \underline{0} \stackapp \pi = r \star \tau \in \Perp^i_1 \setminus \{ d \star \underline{0} \stackapp \pi\}$, which is a contradiction. Thus $s \star \sigma \in \Perp^i_1 \setminus \{ \rlzfont{d} \star \underline{0} \stackapp \pi \}$.
        \item Suppose at least two of $t \star \sigma$, $s \star \sigma$ and $r \star \sigma$ were in $\Perp^i_1 \setminus \{ \rlzfont{d} \star \underline{0} \stackapp \pi \}$. Then $\rlzfont{d} \star \underline{2} \stackapp t \stackapp s \stackapp r \stackapp \pi \in \Perp^i_1$ and $\rlzfont{d} \star \underline{2} \stackapp t \stackapp r \stackapp s \stackapp \pi \neq \rlzfont{d} \star \underline{0} \stackapp \pi$.
    \end{itemize}
    We shall now prove that $\Perp^i_0 \subseteq (\Lambda^i \star \Pi^i) \setminus \Perp^i_1$. To do this, it suffices to prove that $\Lambda^i \star \Pi^i \setminus \Perp^i_1$ satisfies the three defining properties for $\Perp^i_0$.
    \begin{itemize}
        \item By the first claim, we have that for any $\pi \in \Pi$, $\rlzfont{d} \star \underline{0} \stackapp \pi \in (\Lambda^i \star \Pi^i) \setminus \Perp^i_1$.
        \item Suppose that $t \star \pi \in \Lambda^i \star \Pi^i$, $s \star \sigma \in (\Lambda^i \star \Pi^i) \setminus \Perp^i_1$ and $t \star \pi \succ s \star \sigma$. If $t \star \pi$ were in $\Perp^i_1$ then, by \Cref{theorem:Gimel2AlgebraReductionClosure}, $s \star \sigma$ would also be in $\Perp^i_1$, which is a contradiction.
        \item Suppose that $t \star \pi$, $s \star \pi$, $r \star \pi$ were all in $\Lambda^i \star \Pi^i$ but $t \star \pi, s \star \pi \not\in \Perp^i_1$. We shall show that $\rlzfont{d} \star \underline{2} \stackapp t \stackapp s \stackapp r \stackapp \pi \not\in \Perp^i_1$. As per usual, this will be done by showing that $\Perp^i_1 \setminus \{ \rlzfont{d} \star \underline{2} \stackapp t \stackapp s \stackapp r \stackapp \pi \}$ satisfies the three defining properties for $\Perp^i_1$:
        \begin{itemize}
            \item For any $\sigma \in \Pi^i$, $\rlzfont{d} \star \underline{1} \stackapp \sigma \in \Perp^i_1 \setminus \{ \rlzfont{d} \star \underline{2} \stackapp t \stackapp s \stackapp r \stackapp \pi \}$.
            \item Suppose $u \star \sigma \in \Lambda^i \star \Pi^i$, $u \star \sigma \succ v \star \tau$ and $v \star \tau \in \Perp^i_1 \setminus \{ \rlzfont{d} \star \underline{2} \stackapp t \stackapp s \stackapp r \stackapp \pi \}$. Then $u \star \sigma \in \Perp^i_1$. Now, if $u \star \sigma = \rlzfont{d} \star \underline{2} \stackapp t \stackapp s \stackapp r \stackapp \pi$ then $v \star \tau = \rlzfont{d} \star \underline{2} \stackapp t \stackapp s \stackapp r \stackapp \pi \in \Perp^i_1 \setminus \{\rlzfont{d} \star \underline{2} \stackapp t \stackapp s \stackapp r \stackapp \pi \}$, which is a contradiction.
            \item Suppose at least two of $u_1 \star \sigma$, $u_2 \star \sigma$ and $u_3 \star \sigma$ were in $\Perp^i_1 \setminus \{ \rlzfont{d} \star \underline{2} \stackapp t \stackapp s \stackapp r \stackapp \pi \}$. Then $\rlzfont{d} \star \underline{2} \stackapp u_1 \stackapp u_2 \stackapp u_3 \stackapp \sigma \in \Perp^i_1$. Now, if $\rlzfont{d} \star \underline{2} \stackapp u_1 \stackapp u_2 \stackapp u_3 \stackapp \sigma = \rlzfont{d} \star \underline{2} \stackapp t \stackapp s \stackapp r \stackapp \pi$ then we must have that $u_1 = t$, $u_2 = s$, $u_3 = r$, and $\sigma = \pi$. Therefore, at least one of $t \star \pi$ and $s \star \pi$ is in $\Perp^i_1$, which contradicts our initial assumption.
        \end{itemize}
    \end{itemize}
\end{proof}

\begin{theorem}
    $\mathcal{A}$ is a coherent realizability algebra.
\end{theorem}

\begin{proof}
    Let $t \in \mathcal{R}$ be a realizer and suppose that $t \star \pi \in \Perp$ for every $\pi \in \Pi$. As $t$ is a realizer, note that $t$ cannot contain any instance of $\saverlz{\pi^0}$ or $\saverlz{\pi^1}$ and therefore $t \star \pi^i \in \Lambda^i \star \Pi^i$ for $i \in \{0, 1\}$. In particular, $t \in \Lambda^0 \cap \Lambda^1$ and therefore $t \star \pi^0 \in \Perp^0_0$ and $t \star \pi^1 \in \Perp^1_1$. Moreover, the isomorphism $\Xi$ must fix $t$ and therefore $\Xi(t \star \pi^0) = t \star \pi^1$. Thus, by \Cref{theorem:PreservationOfPoles}, $t \star \pi^1 \in \Perp^1_0$ which means that $t \star \pi^1 \in \Perp^1_0 \cap \Perp^1_1$, contradicting \Cref{theorem:Gimel2AlgebraIntersectingPerp}.
\end{proof}

\begin{lemma}
    $\rlzfont{d}(\underline{2}) \Vdash \text{``the Boolean algebra } \fullname{2} \text{ has at most four } \rlzin\text{{-}elements''}$.
\end{lemma}

\begin{proof}
    We shall show that 
    \[
    \rlzfont{d}(\underline{2}) \Vdash \forall x^{\fullname{2}} \, \forall y^{\fullname{2}} \, (x \neq \fullname{0} \land y \neq \fullname{1} \land x \neq y \rightarrow x \Booleanand y \neq x).
    \]
	To see why this suffices, first observe that if $\mathbb{B}$ is a Boolean algebra and there exists an $a$ such that $\mathbb{0} < a < \mathbb{1}$ and both $a$ and $-a$ are atoms\footnote{$a$ is said to be an atom is whenever $b < a$ then $b = \mathbb{0}$}, then $|\mathbb{B}| = 4$.    
So suppose that $|\mathbb{B}| > 4$ and fix $x$ such that $\mathbb{0} < x < \mathbb{1}$ and $x$ is not an atom. This means that we can fix $y \in \mathbb{B}$ such that $\mathbb{0} < y < x$. Then $x, y$ witness the failure of the above statement.      

To prove the claim, let $m, n \in \{0, 1\}$, $t \Vdash \fullname{m} \neq \fullname{0}$, $s \Vdash \fullname{n} \neq \fullname{1}$, $r \Vdash \fullname{m} \neq \fullname{n}$ and $\pi \in \falsity{\fullname{m} \Booleanand \fullname{n} \neq \fullname{m}}$. Since $\falsity{\fullname{m} \Booleanand \fullname{n} \neq \fullname{m}} \neq \emptyset$, by definition $\fullname{m} \Booleanand \fullname{n} = \fullname{m}$, from which it follows that $m \leq n$. There are then three possible cases for $(m, n)$, $(0, 0)$, $(0, 1)$ and $(1, 1)$, and in all of these cases at least two of the terms $t, s, r$ realize $\perp$ and thus the resulting term is in $\Perp^i_j$ for any $i, j \in \{0, 1\}$. Hence, by the construction of the realizability model, $\rlzfont{d} \star \underline{2} \stackapp t \stackapp s \stackapp r \stackapp \pi \in \Perp$. 
\end{proof}

\begin{theorem}
    The Boolean algebra $\fullname{2}$ has exactly four $\rlzin$-elements.
\end{theorem}

\begin{proof}
    We begin by defining two elements of $\rlzstr$:
    \[
    \gamma_0 = (\{\fullname{0}\} \times \Pi^0) \cup ( \{\fullname{1}\} \times \Pi^1) \qquad \gamma_1 = (\{\fullname{1}\} \times \Pi^0) \cup (\{\fullname{0}\} \times \Pi^1).
    \]
    By construction, we have that $\gamma_0, \gamma_1 \subseteq \fullname{2}$. Now:
    \begin{itemize}
        \item  $\identity \Vdash \neg \forall x^{\fullname{2}} (x \notrlzin \gamma_0)$. To see this, let $t \Vdash \forall x^{\fullname{2}} (x \notrlzin \gamma_0)$ and $\pi \in \Pi$. If $\pi \in \Pi \setminus (\Pi^0 \cup \Pi^1)$ then $t \star \pi \in \Perp$ by definition. So suppose that $\pi \in \Pi^0 \cup \Pi^1$. Then, since $\falsity{\forall x^{\fullname{2}} (x \notrlzin \gamma_0)} = \bigcup_{i \in 2} \falsity{\fullname{i} \notrlzin \gamma_0} = \Pi^0 \cup \Pi^1$, we must also have that $t \star \pi \in \Perp$ as required.
        \item $\rlzfont{d}(\underline{0}) \Vdash \fullname{0} \notrlzin \gamma_0$ and $\rlzfont{d}(\underline{1}) \Vdash \fullname{1} \notrlzin \gamma_0$,
    \end{itemize}
    Therefore $\gamma_0$ is not $\rlzin$-empty and every $\rlzin$-element of $\gamma_0$ is not equal to $\fullname{0}$ or $\fullname{1}$. Hence, we must have that $\rlzmodel \Vdash \exists x (x \rlzin \fullname{2} \land x \neq \fullname{0} \land x \neq \fullname{1})$ from which we can conclude that $\fullname{2}$ has size precisely $4$ in $\rlzmodel$.
\end{proof}

\begin{remark}
    In fact, as proven in \cite{Krivine2018}, one can realize that $\gamma_0$ and $\gamma_1$ are singletons and therefore their $\rlzin$-elements are precisely the two atoms of $\fullname{2}$.
\end{remark}

\section{Forcing as a Realizability Algebra} \label{section:ForcingAsRealizability}

It is claimed that realizability is a generalisation of the method of forcing because any forcing Boolean algebra can be naturally seen as a realizability algebra. In this section we explicitly describe this translation, the presentation of which comes from \cite{FontanellaGeoffroy2020}, and expand on the previous work in order to formally state the relationship between forcing and realizability models. We will also assume a lot of background material about Boolean algebras and Boolean valued models, and we refer the reader to Chapters 7 and 14 of \cite{Jech} for the necessary background. It is worthwhile to clearly state here that while such a translation exists, the resulting realizability model ``\emph{loses its computation content}''. By this we mean there is essentially a unique realizer for all realized statements in the model so from a realizer for a proof we can not extract any computational data.

Let $\mathbb{B} = (\mathbb{B}, \mathbb{1}, \mathbb{0}, \Booleanand, \Booleanor, \neg)$ be a complete Boolean algebra.  We define a realizability algebra \hbox{$\mathcal{A}_{\mathbb{B}} = (\Lambda_{(A, B)}, \Pi_{(A, B)}, \prec, \Perp)$} as follows: Firstly, set $A = \emptyset$, that it to say there are no special instructions. Next, let $(\omega_p \divline p \in \mathbb{B})$ be the set of stack bottoms under some fixed enumeration. To define the pre-order and the pole, we begin by inductively defining a function $\tau \colon \Lambda^{\ff{open}} \cup \Pi \rightarrow \mathbb{B}$;
\begin{itemize}
    \item for every stack bottom $\omega_p$, we let $\tau(\omega_p) \coloneqq p$;
    \item for every variable $x$, $\tau(x) \coloneqq \tau(\cc) \coloneqq \mathbb{1}$;
    \item for every term $t$ and stack $\pi$, we let $\tau(t \stackapp \pi) = \tau(t) \Booleanand \tau(\pi)$;
    \item for all $\lambda_c$-terms $t, s$, we let $\tau(\app{t}{s}) \coloneqq \tau(t) \Booleanand \tau(s)$;
    \item for every variable $u$ and every term $t$, we let $\tau(\lambda u \stackapp t) \coloneqq \tau(t)$;
    \item for every stack $\pi$, we let $\tau(\saverlz{\pi}) \coloneqq \tau(\pi)$.
\end{itemize}
This is then naturally extended to the set of all processes by the assignment
\[
\tau(t \star \pi) \coloneqq \tau(t) \Booleanand \tau(\pi).
\]
Then, we define the pre-order by $t \star \pi \succ s \star \sigma$ if and only if $\tau(t \star \pi) \leq \tau(s \star \sigma)$ and we let $\Perp$ be the set of all processes $t \star \pi$ such that $\tau(t \star \pi) = \mathbb{0}$.

\begin{notation}
    There is obviously a large amount of overlapping notation between Boolean-valued models and realizability models. In order to hopefully make the notation clearer for this section, while keeping the notation for realizability harmonious with this rest of this work, we will do the following:
    \begin{itemize}
        \item $\falsity{\varphi}$ and $\verity{\varphi}$ will refer to their usual meanings in the realizability model.
        \item The Boolean value of a formula $\varphi$ will be denoted by $\Booleanforce{\varphi}$.
        \item $\Vdash_{\mathcal{A}}$ refers to the realizability relation over the structure $\mathcal{A}_\mathbb{B}$.
        \item $\Vdash_{\mathbb{B}}$ refers to forcing with the Boolean algebra. Namely, $p \Vdash_\mathbb{B} \varphi$ if and only if $p \leq \Booleanforce{\varphi}$.
        \item $\rlzstr^\mathcal{A}$ will refer to the domain of the realizability model while $\tf{V}^\mathbb{B}$ will refer to the domain of the Boolean-valued model.        
    \end{itemize}
\end{notation}

We continue by observing that if two terms are mapped to the same element of the Boolean algebra then they realize the same formulas. From this it follows that the realizability model ``\emph{loses its computational content}''. This is because there is essentially only one realizer which gives us one single program for all proofs.

\begin{proposition} \label{theorem:SameBooleanValues}
    For any $t, s \in \Lambda$ and formula $\varphi \in Fml_{\rlzin}$, if $\tau(t) = \tau(s)$ then $t \in \verity{\varphi} \Longleftrightarrow s \in \verity{\psi}$. 
    
    In particular, if $t$ is a realizer then $t \Vdash \varphi \Longleftrightarrow \identity \Vdash \varphi$.
\end{proposition}

\begin{proof}
    Suppose that $t \Vdash \varphi$ and $\pi \in \falsity{\varphi}$. By the definition of $\Perp$ this means that $\mathbb{0} = \tau(t \star \pi) = \tau(t) \Booleanand \tau(\pi) = \tau(s) \Booleanand \tau(\pi) = \tau(s \star \pi)$. Thus $s \star \pi \in \Perp$ so, since $\pi$ was arbitrary, $s \Vdash \varphi$.

    For the second statement it suffices to observe that if $t$ is a realizer, then $\tau(t) = \mathbb{1}$ and thus all realizers realize the same formulas. Formally, this is proven by induction on the construction of all - possibly open - $\lambda_c$-terms by showing that if $t$ does not contain any continuation constants then $\tau(t) = \mathbb{1}$. But this is obvious from the definition of $\tau$ since $\tau(\cc) = \mathbb{1}$ as does $\tau(u)$ for any variable $u$ and if $t$ is a $\lambda_c$-term such that $\tau(t) = \mathbb{1}$ then $\tau(\lambda u \lambdaapp t) = \tau(t) = \mathbb{1}$.
\end{proof}

From this it follows that in this case the realizability model is trivial, namely the Boolean algebra $\fullname{2}$ has only 2 elements. We shall prove a converse to this result in \Cref{section:Fullname2Trivial}.

\begin{proposition}
    $\identity \Vdash \forall x^{\fullname{2}} (x \neq \fullname{0} \rightarrow ( x \neq \fullname{1} \rightarrow \perp))$.
\end{proposition}

\begin{proof}
    As per usual, it suffices to prove that 
    \[
    \identity \Vdash x \neq \fullname{0} \rightarrow (x \neq \fullname{1} \rightarrow \perp)
    \]
    for $x = \fullname{0}$ and $x = \fullname{1}$. Fix $t \Vdash x \neq \fullname{0}$, $s \Vdash x \neq \fullname{1}$ and $\pi \in \Pi$. If $x=  \fullname{0}$ then $\falsity{\fullname{0} \neq \fullname{0}} = \Pi$ and therefore $\identity \star t \stackapp s \stackapp \pi \succ t \star s \stackapp \pi \in \Perp$. On the other hand, if $x = \fullname{1}$ then $\lambda u \lambdaapp \lambda v \lambdaapp v \star t \stackapp s \stackapp \pi \succ s \star \pi \in \Perp$. Thus, $\underline{0} \Vdash \fullname{1} \neq \fullname{0} \rightarrow (\fullname{1} \neq \fullname{1} \rightarrow \perp)$. Since $\underline{0}$ is a realizer, by \Cref{theorem:SameBooleanValues} we also have that $\identity$ realizes the same sentence, completing the proof.   
\end{proof}

In remains to prove that the Boolean-valued model and the realizability model prove the same sentences. To do this, recall that the Boolean-valued model $\tf{V}^\mathbb{B}$ is defined recursively as 
\begin{itemize}
    \item $\tf{V}_0^\mathbb{B} = \emptyset$;
    \item $\tf{V}_{\alpha+1}^{\mathbb{B}}$ is the set of all functions $b$ whose domain is a subset of $\tf{V}_\alpha^{\mathbb{B}}$ and whose range is a subset of $\mathbb{B}$;
    \item If $\lambda$ is a limit ordinal then $\tf{V}_\lambda^\mathbb{B} = \bigcup_{\alpha \in \lambda} \tf{V}_\alpha^{\mathbb{B}}$;
    \item $\tf{V}^\mathbb{B} = \bigcup_{\alpha \in \tf{Ord}} \tf{V}_\alpha^{\mathbb{B}}$.
\end{itemize}
On the other hand, the defining property for the domain of the realizability model was that $\rlzstr_{\alpha+1}^{\mathcal{A}} = \mathcal{P}(\rlzstr_\alpha^\mathcal{A} \times \Pi)$. It is easy to extend $\tau$ to map any name in the realizability model to one in the Boolean-valued model and similarly we can define a new function $\sigma$ which does the reverse. We note here that this interpretation is very similar to how one goes between the Boolean-valued approach and the partial order approach for forcing and is where we require the Boolean algebra to be complete.

\begin{definition}
    We recursively define $\tau(a) \in \tf{V}^\mathbb{B}$, for $a \in \rlzstr^\mathcal{A}$, to be the function $\tau(a) \colon \tau \pointwise \ff{dom}(a) \rightarrow \mathbb{B}$ given by
    \[
    \tau(a)(\tau(x)) \coloneqq \Booleansum \{ \tau(\pi) \divline ( x, \pi ) \in a \}.
    \]
    Similarly, we recursively define $\sigma(b) \in \rlzstr^\mathcal{A}$, for $b \in \tf{V}^\mathbb{B}$, as
    \[
    \sigma(b) \coloneqq \{ ( \sigma(x), \omega_p ) \divline b(x) = p \}.
    \]
\end{definition}

\begin{proposition}
    For any $a \in \rlzstr^\mathcal{A}$ and $b \in \tf{V}^\mathbb{B}$, 
    \[
    \identity \Vdash_\mathcal{A} a \simeq \sigma(\tau(a)) \qquad \text{and} \qquad \mathbb{1} \Vdash_\mathbb{B} b = \tau(\sigma(b)).
    \]
\end{proposition}

\begin{proof}
    We prove the first statement by induction on the rank of $a$. So suppose that for all $c \in \ff{dom}(a)$, $\identity \Vdash_{\mathcal{A}} c \simeq \sigma(\tau(c))$. Using \Cref{theorem:SameBooleanValues}, it will suffice to show that $\identity \Vdash a \subseteq \sigma(\tau(a))$ and $\identity \Vdash \sigma(\tau(a)) \subseteq a$.

    For the first of these, fix $c \in \rlzstr$ and suppose that $t \Vdash c \not\in \sigma(\tau(a))$ while $(c, \pi) \in a$. By the induction hypothesis $\identity \Vdash_{\mathcal{A}} c \simeq \sigma(\tau(c))$. Now, by definition of the translations, we have $(\sigma(\tau(c)), \omega_p) \in \sigma(\tau(a))$ where $p = \tau(a)(\tau(c)) = \Booleansum \{(\tau(\pi') \divline (c, \pi') \in a\} \geq \tau(\pi)$. Therefore, $t \star \identity \stackapp \omega_p \in \Perp$. However, $\tau(t) \Booleanand \tau(\pi)  \leq \tau(t) \Booleanand p = \tau(t) \Booleanand \tau(\identity) \Booleanand \tau(\omega_p) = 0$. Thus $t \star \pi \in \Perp$ so $\identity \Vdash a \subseteq \sigma(\tau(a))$.

    For the reverse direction, fix $c \in a$ and suppose that $t \Vdash \sigma(\tau(c)) \not\in a$ while $(\sigma(\tau(c)), \omega_p) \in \sigma(\tau(a))$. Now, by the induction hypothesis, $\identity \Vdash \sigma(\tau(c)) \simeq c$. Therefore, for any $\pi$ such that $(c, \pi) \in a$, $t \star \identity \stackapp \pi \in \Perp$, from which it follows that $\tau(t) \Booleanand \tau(\pi) = 0$. So, since $p$ is the supremum of $\{ \tau(\pi') \divline (c, \pi') \in a \}$ it must be the case that $\tau(t) \Booleanand p$ is also equal to $0$. Thus $\identity \star t \stackapp \omega_p \in \Perp$, as required. \\

    \noindent We again prove that second statement by induction on the rank of $b$. So suppose that for all $c \in \ff{dom}(b)$, $\mathbb{1} \Vdash_{\mathbb{B}} c = \tau(\sigma(c))$. As before we consider the cases $b \subseteq \tau(\sigma(b))$ and $\tau(\sigma(b)) \subseteq b$ separately. To prove this, observe that if $c \in \ff{dom}(b)$ then $(\sigma(c), \omega_{b(c)}) \in \sigma(b)$ and therefore $\tau(\sigma(c)) \in \ff{dom}(\tau(\sigma(b)))$ and $\tau(\sigma(b))(\tau(\sigma(c))) = b(c)$.

    For the first case, we shall prove that $\Booleanforce{b \subseteq \tau(\sigma(b))} \coloneqq \Booleanproduct \{ b(c) \rightarrow \Booleanforce{c \in \tau(\sigma(b))} \divline c \in \ff{dom}(b) \} = \mathbb{1}$. To do this, it suffices to prove that for every $c \in \ff{dom}(b)$, $\Booleanforce{c \in \tau(\sigma(b))} \geq b(c)$. But this follows from the inductive hypothesis since 
    \begin{align*}
    \Booleanforce{c \in \tau(\sigma(b))} & \coloneqq \Booleansum \{ \Booleanforce{c =d} \Booleanand \tau(\sigma(b))(d) \divline d \in \ff{dom}(\tau(\sigma(b))) \} \\
    & \geq \Booleanforce{c = \tau(\sigma(c))} \Booleanand \tau(\sigma(b))(\tau(\sigma(c))) \\
    & = b(c).
    \end{align*}

    For the reverse direction, we have that 
    \[
    \Booleanforce{\tau(\sigma(b)) \subseteq b} = \Booleanproduct \{ \tau(\sigma(b))(\tau(\sigma(c))) \rightarrow \Booleanforce{\tau(\sigma(c)) \in b} \divline c \in \ff{dom}(b) \}.
    \]
    Then, as before, using the inductive hypothesis we have that $\Booleanforce{\tau(\sigma(c)) \in b} \geq \Booleanforce{\tau(\sigma(c)) = c} \Booleanand b(c) = b(c) = \tau(\sigma(b))(\tau(\sigma(c)))$. From this it follows that $\Booleanforce{\tau(\sigma(b)) \subseteq b} = \mathbb{1}$, as required.  
\end{proof}

\begin{theorem} \label{theorem:BooleanVsRealize}
For any formula $\varphi \in Fml_\in$, $t \in \Lambda$ and $a_0, \dots, a_n \in \rlzstr^{\mathcal{A}}$ 
    \[
    t \Vdash_{\mathcal{A}} \varphi(a_0, \dots, a_n) \Longleftrightarrow \tau(t) \Vdash_{\mathbb{B}} \varphi(\tau(a_0), \dots, \tau(a_n)).
    \]
\end{theorem}

In order to prove this theorem when the formula is of the form $\varphi \rightarrow \psi$ we will make use of the following lemma which identifies the stacks which falsify a given formula. For ease of notation we will omit variables.

\begin{lemma} \label{theorem:BooleanVsFalsity}
    Suppose that $p \Vdash_\mathbb{B} \varphi \Longrightarrow \forall t \in \Lambda ( \tau(t) = p \rightarrow t \Vdash_{\mathcal{A}} \varphi)$. Then 
    \[
    (\pi \in \falsity{\varphi} \Longrightarrow p \Booleanand \tau(\pi) = \mathbb{0}).
    \]
\end{lemma}

\begin{proof}
    Suppose that the assumption of the lemma holds and fix $\pi \in \falsity{\varphi}$. Then, since $\tau(\saverlz{\omega_p}) = \tau(\omega_p) = p$, we must have that $\saverlz{\omega_p} \Vdash_{\mathcal{A}} \varphi$. Thus, $\saverlz{\omega_p} \star \pi \in \Perp$, that is, $\mathbb{0} = \tau(\saverlz{\omega_p}) \Booleanand \tau(\pi) = p \Booleanand \tau(\pi)$.
\end{proof}

\begin{proof}[Proof of \ref{theorem:BooleanVsRealize}]
    We shall prove this by induction on the complexity of formulas. For this it will suffice to consider in turn the following cases: $\top$, $\perp$, atomic formulas, implications, bounded universal quantification.

    Firstly, it is clear that for every $p \in \mathbb{B}$ and $t \in \Lambda$, $p \Vdash_\mathbb{B} \top$ and $t \Vdash_\mathcal{A} \top$ so the conclusion of the theorem is immediate here. Similarly, $t \Vdash_\mathcal{A} \perp$ if and only if for every $\pi \in \Pi$, $\tau(t) \Booleanand \tau(\pi) = \mathbb{0}$. In particular, since $\tau(\omega_\mathbb{1}) = \mathbb{1}$, $\tau(t) = \tau(t) \Booleanand \mathbb{1} = 0$. Conversely, $p \Vdash_\mathbb{B} \perp$ if and only if $p = \mathbb{0}$. \\

    \noindent We prove the two atomic cases by induction on rank. Suppose that $t \Vdash_{\mathcal{A}} a \subseteq b$, in order to prove that $\tau(t) \Vdash_{\mathbb{B}} \tau(a) \subseteq \tau(b)$ it suffices to prove that for any $c \in \ff{dom}(a)$, $\tau(t) \leq \tau(a)(\tau(c)) \rightarrow \Booleanforce{\tau(c) \in \tau(b)}$. By the laws of Boolean operations, this is equivalent to proving that for any $c \in \ff{dom}(a)$, 
    \[
    \mathbb{0} = \tau(t) - (- \tau(a)(\tau(c)) \Booleanor \Booleanforce{\tau(c) \in \tau(b)}) = \tau(t) \Booleanand \tau(a)(\tau(c)) \Booleanand \Booleanforce{\tau(c) \not\in \tau(b)}.
    \]
    To do this, let $p = \Booleanforce{\tau(c) \not\in \tau(b)}$. Then by the inductive hypothesis, $\saverlz{\omega_p} \Vdash_{\mathcal{A}} c \not\in b$. Therefore, by the assumption on $t$, if $(c, \pi) \in a$ then $t \star \saverlz{\omega_p} \stackapp \pi \in \Perp$. Thus, 
    \[
    \tau(t) \Booleanand \tau(a)(\tau(c)) \Booleanand \Booleanforce{\tau(c) \not\in \tau(b)} \leq \tau(t) \Booleanand \Booleansum \{ \tau(\pi) \divline (c, \pi) \in a \} \Booleanand p = \mathbb{0}.
    \]
    For the reverse implication, assume that $\tau(t) \Vdash_{\mathbb{B}} \tau(a) \subseteq \tau(b)$, $(c, \pi) \in a$ and $s \Vdash_{\mathcal{A}} c \not\in b$. Then, by the inductive hypothesis, we have $\tau(s) \Vdash_{\mathbb{B}} \tau(c) \not\in \tau(b)$ and 
    \[
    \tau(t) \Booleanand \tau(s) \Booleanand \tau(\pi) \leq \tau(t) \Booleanand \Booleanforce{\tau(c) \not\in \tau(b)} \Booleanand \tau(a)(\tau(c)) = \mathbb{0}.
    \]
    Thus, $t \star s \stackapp \pi \in \Perp$ so $t \Vdash_{\mathcal{A}} a \subseteq b$. \\

    \noindent Moving onto the second atomic case, suppose that $t \Vdash_{\mathcal{A}} a \not\in b$. Again, in order to prove that $\tau(t) \Vdash_{\mathbb{B}} \tau(a) \not\in \tau(b)$, it suffices to prove that 
    \[
    \tau(t) \leq - \Booleansum \{ \Booleanforce{\tau(a) = \tau(c)} \Booleanand \tau(b)(\tau(c)) \divline c \in \ff{dom}(b) \}.
    \]
    Using the De Morgan laws it can be seen that this is equivalent to showing that for every $c \in \ff{dom}(b)$, $\tau(t) \leq - \Booleanforce{\tau(a) = \tau(c)} \Booleanor - \tau(b)(\tau(c))$. By definition of $\Booleanforce{\tau(a) = \tau(c)}$, this reduces to showing that
    \[
    \tau(t) \leq - (\Booleanforce{\tau(a) \subseteq \tau(c)} \Booleanand \Booleanforce{\tau(c) \subseteq \tau(a)} \Booleanand \tau(b)(\tau(c))),
    \]
    which gives $\tau(t) \Booleanand \Booleanforce{\tau(a) \subseteq \tau(c)} \Booleanand \Booleanforce{\tau(c) \subseteq \tau(a)} \Booleanand \tau(b)(\tau(c)) = \mathbb{0}$. To prove this, set \hbox{$p = \Booleanforce{\tau(a) \subseteq \tau(c)}$} and $p' = \Booleanforce{\tau(c) \subseteq \tau(a)}$. Using the inductive hypothesis we have $\saverlz{\omega_p} \Vdash_{\mathcal{A}} a \subseteq c$ and $\saverlz{\omega_{p'}} \Vdash_{\mathcal{A}} c \subseteq a$. Therefore, by the assumption on $t$, if $(c, \pi) \in b$ then $t \star \saverlz{\omega_p} \stackapp \saverlz{\omega_{p'}} \stackapp \pi \in \Perp$. Thus,
    \[
    \tau(t) \Booleanand \Booleanforce{\tau(a) \subseteq \tau(c)} \Booleanand \Booleanforce{\tau(c) \subseteq \tau(a)} \Booleanand \tau(b)(\tau(c)) \leq \tau(t) \Booleanand p \Booleanand p' \Booleanand \Booleansum \{ \tau(\pi) \divline (c, \pi) \in b \} = \mathbb{0}.
    \]
    For the reverse implication, assume that $\tau(t) \Vdash_{\mathbb{B}} \tau(a) \not\in \tau(b)$, $(c, \pi) \in b$, $s \Vdash_{\mathcal{A}} a \subseteq c$ and $s' \Vdash_{\mathcal{A}} c \subseteq a$. Then, by the inductive hypothesis, we have $\tau(s) \Vdash_{\mathbb{B}} \tau(a) \subseteq \tau(c)$, $\tau(s') \Vdash_{\mathbb{B}} \tau(c) \subseteq \tau(a)$ and so
    \[
    \tau(t) \Booleanand \tau(s) \Booleanand \tau(s') \Booleanand \tau(\pi) \leq \tau(t) \Booleanand \Booleanforce{\tau(a) \subseteq \tau(c)} \Booleanand \Booleanforce{\tau(c) \subseteq \tau(a)} \Booleanand \tau(a)(\tau(c)) = \mathbb{0}.
    \]
    Thus, $t \star s \stackapp s' \stackapp \pi \in \Perp$, so $t \Vdash_{\mathcal{A}} a \not\in b$. \\

    \noindent \noindent Next, suppose that the claims holds for $\varphi$ and $\psi$. Suppose that $t \Vdash_\mathcal{A} \varphi \rightarrow \psi$ and let $p = \Booleanforce{\varphi}$. Then, by the inductive hypothesis, $\saverlz{p} \Vdash_\mathcal{A} \varphi$ so, by \Cref{theorem:ImplicationandApplication}, $\app{t}{\saverlz{p}} \Vdash_\mathcal{A} \psi$. Thus, by the inductive hypothesis again, $\tau(\app{t}{\saverlz{p}}) \Vdash_\mathbb{B} \psi$. Therefore, $\tau(t) \Booleanand \Booleanforce{\varphi} = \tau(t) \Booleanand \tau(\saverlz{p}) \leq \Booleanforce{\psi}$, which can easily be seen to give $\tau(t) \leq \Booleanforce{\varphi \rightarrow \psi}$.

    For the reverse direction, suppose that $\tau(t) \Vdash_\mathbb{B} \varphi \rightarrow \psi$, $s \Vdash_\mathcal{A} \varphi$ and $\pi \in \falsity{\psi}$. By the inductive hypothesis, $\tau(s) \leq \Booleanforce{\varphi}$ and so $\tau(t) \Booleanand \tau(s) \leq \Booleanforce{\psi}$. Then, by \Cref{theorem:BooleanVsFalsity}, since $\pi \in \falsity{\psi}$, $\tau(t) \Booleanand \tau(s) \Booleanand \tau(\pi) = \mathbb{0}$ which gives us that $t \star s \stackapp \pi \in \Perp$. Hence $t \Vdash_\mathcal{A} \varphi \rightarrow \psi$, as required. \\

    \noindent For the final case, suppose that the claim holds for $\varphi(x)$. Suppose that $t \Vdash_\mathcal{A} \forall x \varphi(x)$ and fix $b \in \tf{V}^\mathbb{B}$. Then $\sigma(b) \in \rlzstr^\mathcal{A}$ from which it follows that $t \Vdash_\mathcal{A} \varphi(\sigma(b))$. Therefore, by the inductive hypothesis, $\tau(t) \Vdash_\mathbb{B} \varphi(\tau(\sigma(b))$. Since $\mathbb{1} \Vdash_\mathbb{B} b = \tau(\sigma(b))$, this gives us that $\tau(t) \Vdash_\mathbb{B} \varphi(b)$ and thus $\tau(t) \leq \Booleanforce{\varphi(b)}$. Since $b$ was an arbitrary element of $\tf{V}^\mathbb{B}$ it follows that $\tau(t) \leq \Booleanforce{\forall x \varphi(x)}$, as required.  

    For the reverse direction, suppose that $\tau(t) \Vdash_{\mathbb{B}} \forall x \varphi(x)$ and fix $a \in \rlzstr^{\mathcal{A}}$. Then $\tau(a) \in \tf{V}^{\mathbb{B}}$ and thus $\tau(t) \leq \Booleanforce{\varphi(\tau(a))}$. Therefore, by the inductive hypothesis, $t \Vdash_{\mathcal{A}} \varphi(a)$. Thus, since $a \in \rlzstr^{\mathcal{A}}$ was arbitrary, $t \Vdash_{\mathcal{A}} \forall x \varphi(x)$. 
\end{proof}

We end by concluding that the Boolean-valued model and the realizability model are ``the same'' in some precise way.

\begin{corollary}
    For any sentence $\varphi \in Fml_\in$, $\rlzmodel \Vdash_\mathcal{A} \varphi$ if and only if $\varphi$ is valid \footnote{that is $\Booleanforce{\varphi} = \mathbb{1}$} in $\tf{V}^\mathbb{B}$.
\end{corollary}

\begin{proof}
    $\rlzmodel \Vdash_\mathcal{A} \varphi$ if and only if there is some realizer $t$ such that $t \Vdash_\mathcal{A} \varphi$. Since $\tau(t) = \mathbb{1}$ for any realizer, $\rlzmodel \Vdash_\mathcal{A} \varphi$ is equivalent to asserting that $\Booleanforce{\varphi} = \mathbb{1}$.
\end{proof}

\section{When \texorpdfstring{$\fullname{2}$}{Reish 2} is trivial} \label{section:Fullname2Trivial}

Here we discuss a partial converse to the previous section, which is the claim by Krivine that when $\fullname{2}$\footnote{or, with the original formulation, $\cjgimel{2}$} is trivial then the realizability model is in fact a forcing model. This proof has not appeared in print before but is contained in Krivine's slides on ``Some properties of realizability models'' given in Chamb\'{e}ry in June 2012 \cite{KrivineChambery}. We note here that this proof requires that the realizability algebra is \emph{countable}. Therefore, it is unknown if the more general claim, that when $\fullname{2}$ is trivial then the realizability model is a forcing model, is true when working without uncountable realizability algebras. For simplicity, we will also assume that there are no special instructions.

Suppose that $\mathcal{A}$ is a countable realizability algebra such that in the realizability model $\fullname{2}$ is trivial. This means that there exists a realizer $\rlzfont{r} \in \mathcal{R}$ such that
\[
\rlzfont{r} \Vdash \forall x^{\fullname{2}} \, (x \neq \fullname{0} \rightarrow (x \neq \fullname{1} \rightarrow \perp)).
\]
In particular, we have that $\rlzfont{r} \Vdash (\fullname{0} \neq \fullname{0} \rightarrow (\fullname{0} \neq \fullname{1} \rightarrow \perp)$ and \newline $\rlzfont{r} \Vdash (\fullname{1} \neq \fullname{0} \rightarrow (\fullname{1} \neq \fullname{1} \rightarrow \perp)$. Which is to say,
\[
\rlzfont{r} \in \verity{\top \rightarrow (\perp \rightarrow \perp)} \cap \verity{\perp \rightarrow (\top \rightarrow \perp).}
\]
Let $\rlzfont{r}' = \lambda u \lambdaapp \lambda v \lambdaapp \cc \Big( \lambda k \lambdaapp \big( \rlzfont{r} (k(u)) \big) (k(v)) \Big)$. We shall show that $r'$ satisfies the instruction \emph{fork}.

\begin{lemma} \label{claim:UniformRealizer}
    For any terms $t, s \in \Lambda$ and stacks $\pi \in \Pi$, if either $t \star \pi \in \Perp$ or $s \star \pi \in \Perp$ then $\rlzfont{r}' \star t \stackapp s \stackapp \pi \in \Perp$.
\end{lemma}

\begin{proof}
    First, observe that if $t \star \pi \in \Perp$ then for any $\sigma \in \Pi$, 
    \[
    \app{\saverlz{\pi}}{t} \star \sigma \succ \saverlz{\pi} \star t \stackapp \sigma \succ t \star \pi \in \Perp,
    \]
    and therefore $\app{\saverlz{\pi}}{t} \Vdash \perp$. Now,
    \begin{equation*}
    \begin{split}
    \rlzfont{r}' \star t \stackapp s \stackapp \pi & \succ \cc \Big( \lambda k \lambdaapp \twoapp{\inapp{\rlzfont{r}}{\app{k}{t}}}{\app{k}{s}} \Big) \star \pi \succ \cc \star \Big( \lambda k \lambdaapp \twoapp{\inapp{\rlzfont{r}}{\app{k}{t}}}{\app{k}{s}} \Big) \stackapp \pi \\
    & \succ \lambda k \lambdaapp \twoapp{\inapp{\rlzfont{r}}{\app{k}{t}}}{\app{k}{s}} \star \saverlz{\pi} \stackapp \pi \succ \twoapp{\inapp{\rlzfont{r}}{\app{\saverlz{\pi}}{t}}}{\app{\saverlz{\pi}}{s}} \star \pi.
    \end{split}
    \end{equation*}
    Since either $\app{\saverlz{\pi}}{t}$ or $\app{\saverlz{\pi}}{s}$ realizes $\perp$ by our initial assumption, by the construction of $r$ we must have that $\twoapp{\inapp{\rlzfont{r}}{\app{\saverlz{\pi}}{t}}}{\app{\saverlz{\pi}}{s}} \Vdash \perp$. Thus $\twoapp{\inapp{\rlzfont{r}}{\app{\saverlz{\pi}}{t}}}{\app{\saverlz{\pi}}{s}} \star \pi \in \Perp$, yielding the result.
\end{proof}

\noindent What we shall see next is that there is a single realizer which uniformly realizes every realizable statement in $\rlzmodel$. In order to do this, we shall focus on subsets of $\Pi$ rather than formulas, and we briefly explain here why this generality is sufficient. Given $X \subseteq \Pi$, if we let $1_X \coloneqq \{ (0, \pi) \divline \pi \in X\}$ and $\theta_X \equiv 0 \notrlzin 1_X$, then $\falsity{\theta_X} = \{ \pi \divline (0, \pi) \in 1_X\} = X$. Moreover, for any formula $\varphi$, $\falsity{\theta_{\falsity{\varphi}}} = \falsity{\varphi}$ and therefore $\identity \Vdash \varphi \rightarrow \theta_{\falsity{\varphi}}$ and $\identity \Vdash \theta_{\falsity{\varphi}} \rightarrow \varphi$. From this we can see that there is a one-to-one correspondence between falsity values and subsets of $\Pi$.

Expanding our notation, given $X \subseteq \Pi$, we shall say that $t \Vdash X$ iff $t \Vdash \theta_X$, namely for every $\pi \in X$, $t \star \pi \in \Perp$. Analogous notation will also be used for the statement $A \rightarrow B$. 

\begin{theorem} \label{theorem:Gimel2TrivialOneRealizer}
    $\exists \Phi \in \mathcal{R} \; \forall t \in \mathcal{R} \; \forall X \subseteq \Pi \; (t \Vdash X \Rightarrow \Phi \Vdash X)$.
\end{theorem}

\begin{proof}

We will use the result from $\lambda$-calculus that there is a single $\lambda$-term which enumerates all other closed $\lambda$-terms. An involved proof of this, for the standard $\lambda$-calculus, can be found in Chapter 8 of \cite{Barendregt1985}, and this can be adapted to the $\lambda_c$-calculus. 

Therefore, we fix a realizer $\rlzfont{e}$ such that for any realizer $t$ there is some $n \in \mathbb{N}$ such that for any $\pi \in \Pi$, $\inapp{\rlzfont{e}}{\underline{n}} \star \pi \succ t \star \pi$. Next, let $\rlzfont{s} = \lambda n \lambdaapp \lambda u \lambdaapp \lambda v \lambdaapp \twoapp{\app{n}{u}}{\app{u}{v}}$ and recall that for every $n$ we have $\app{\rlzfont{s}}{\underline{n}} = \underline{n+1}$.
    

    Now, define $q \coloneqq \lambda x \lambdaapp \lambda y \lambdaapp \twoapp{\inapp{\rlzfont{r}'}{\app{\rlzfont{e}}{y}}}{\inapp{\app{x}{x}}{\app{\rlzfont{s}}{y}}}$ and let $p \coloneqq qq$. Then, for any $n \in \mathbb{N}$,
    \[
    p(\underline{n}) = q[x \coloneqq q, y \coloneqq \underline{n}] = \twoapp{\inapp{\rlzfont{r}'}{\app{\rlzfont{e}}{\underline{n}}}}{\inapp{p}{\app{\rlzfont{s}}{\underline{n}}}}.
    \]
    and thus, for any $\pi \in \Pi$,
    \[
    p \star \underline{n} \stackapp \pi \succ \rlzfont{r}' \stackapp (\app{\rlzfont{e}}{\underline{n}}) \stackapp (\inapp{p}{\app{\rlzfont{s}}{\underline{n}}}) \stackapp \pi.
    \]
    We shall show that $\app{p}{\underline{0}}$ is our desired realizer. \\

    \noindent Suppose that $t \in \mathcal{R}$, $X \subseteq \Pi$ and $t \Vdash X$. Fix $\pi \in X$ and $n \in \mathbb{N}$ such that $\app{\rlzfont{e}}{\underline{n}} \star \pi \succ t \star \pi$. Then, since $t \star \pi \in \Perp$, we must have that $\app{\rlzfont{e}}{\underline{n}} \star \pi \in \Perp$. So, by \Cref{claim:UniformRealizer}, $p \star \underline{n} \stackapp \pi \succ \rlzfont{r}' \stackapp (\app{\rlzfont{e}}{\underline{n}}) \stackapp (\inapp{p}{\app{\rlzfont{s}}{\underline{n}}}) \stackapp \pi \in \Perp$. 

    The theorem will follow once we show that if $p \star \underline{m+1} \stackapp \pi \in \Perp$ then so is $p \star \underline{m} \stackapp \pi$. But, if $p \star \underline{m+1} \stackapp \pi \in \Perp$ then, since $\inapp{p}{\app{\rlzfont{s}}{\underline{m}}} \star \pi \succ p \star \underline{m+1} \stackapp \pi$ the former is also in $\Perp$. Therefore, by \Cref{claim:UniformRealizer}, $p \star \underline{m} \stackapp \pi \succ \rlzfont{r}' \stackapp (\app{\rlzfont{e}}{\underline{n}}) \stackapp (\inapp{p}{\app{\rlzfont{s}}{\underline{m}}}) \stackapp \pi \in \Perp$.
\end{proof}

\noindent Now, let $\mathbb{B} = \mathcal{P}(\Pi)$ be the Boolean algebra of truth values and define a partial order on $\mathbb{B}$ by
\[
A \leq B \quad \Longleftrightarrow \quad \exists t \in \mathcal{R} (t \Vdash A \rightarrow B) \quad \Longleftrightarrow \quad \Phi \Vdash A \rightarrow B.
\]
Observe that $\theta_A \land \theta_B \equiv (0 \notrlzin 1_A) \land (0 \notrlzin 1_B) \equiv 0 \notrlzin 1_{A \cup B} = \theta_{A \cup B}$ and similarly $\theta_A \lor \theta_B\equiv \theta_{A \cap B}$. Therefore, $\Pi \leq A \cup B \leq A \leq A \cap B \leq \emptyset$ for every $A, B \in \mathcal{P}(\Pi)$. Thus the meet operator of $\mathbb{B}$ is given by intersections while the join operator is given by unions, and the partial order extends reverse inclusion.

\begin{theorem}
    $\mathbb{B}$ is a complete Boolean algebra.
\end{theorem}

\begin{proof}
    It will suffice to prove that if $\{ X_i \divline i \in I \} \subseteq \mathbb{B}$ then $\ff{sup}_{i \in I} X_i = \bigcup_{i \in I} X_i$. Let $A \leq X_i$ for every $i \in I$. Suppose that $t \Vdash A$ and $\pi \in \bigcup_{i \in I} X_i$, then $\pi \in X_j$ for some $j$. Since $A \leq X_j$, $\Phi \Vdash A \rightarrow X_j$ and thus $\Phi \star t \stackapp \pi \in \Perp$. Hence, $\Phi \Vdash A \rightarrow \bigcup_{i \in I} X_i$ and so $A \leq \bigcup_{i \in I} X_i$.

    Conversely, if $t \Vdash \bigcup_{i \in I} X_i$ then for any $j \in I$ and $\pi \in X_j$, $t \star \pi \in \Perp$. Thus $\identity \Vdash \bigcup_{i \in I} X_i \rightarrow X_j$, from which it follows that $\Phi \Vdash \bigcup_{i \in I} X_i \rightarrow X_j$ and $\bigcup_{i \in I} X_i \leq X_j$.
\end{proof}

\noindent Since $\mathbb{B}$ is a complete Boolean Algebra, we can produce the Boolean-valued  forcing model $\tf{V}^\mathbb{B}$. Recall that a formula $\psi$ is said to be \emph{valid} in this model if $\Booleanforce{\psi}_\mathbb{B} = \emptyset$ (the maximal element of the Boolean algebra).

\begin{lemma}
    For any formula $\psi \in Fml_{\in}$, $\falsity{\psi} = \Booleanforce{\psi}_\mathbb{B}$.
\end{lemma}

\begin{proof}
    Formally, this will be done by induction on the complexity of formulas, but we shall only sketch the details below. For atomic formulas it suffices to show that: $\falsity{x \simeq x} = \emptyset$; $\falsity{x \simeq y} = \falsity{y \simeq x}$; $\falsity{x \simeq y} \cdot \falsity{y \simeq z} \leq \falsity{x \simeq z}$ and $\falsity{x \in y} \cdot \falsity{v \simeq x} \cdot \falsity{w \simeq y} \leq \falsity{v \in w}$.

    Observe that the propositional logical counterparts of these statements are: $x \simeq x \rightarrow \top$; $x \simeq y \leftrightarrow y \simeq x$; $x \simeq y \land y \simeq z \rightarrow x \simeq z$ and $x \in y \land v \simeq x \land w \simeq y \rightarrow v \in w$. Since all \linebreak[2] of these are true statements, and the realizability model preserves logic deductions, all of these statements are realized. Thus, they are all realized by $\Phi$ from which it follows that $\falsity{\cdot}$ satisfies these four properties.

    The logical connectives $\land$ and $\lor$ follow from the fact that meets correspond to intersections while joins correspond to unions. Meanwhile, the proof for implication comes directly from the definition of the partial order. Finally the proofs for the quantifiers will follow from $\mathbb{B}$ being a complete Boolean algebra with $\ff{sup}_{i \in I}X_i = \bigcup_{i \in I} X_i$ and $\ff{inf}_{i \in I}X_i = \bigcap_{i \in I} X_i$.
\end{proof}

\begin{corollary}
    For any formula $\psi \in Fml_\in$, $\rlzmodel \Vdash \psi$ if and only if $\Booleanforce{\psi}_\mathbb{B} = \emptyset$. Thus, $\rlzmodel$ is a forcing model.
\end{corollary}

\begin{proof}
    Fix $\psi$. First recall that $\rlzmodel \Vdash \psi$ means that there exists a realizer $t \in \mathcal{R}$ such that $t \Vdash \psi$ which, by \Cref{theorem:Gimel2TrivialOneRealizer}, is equivalent to saying that $\Phi \Vdash \psi$. Then $\Phi \Vdash \psi \Longleftrightarrow \Phi \Vdash \top \rightarrow \psi \Longleftrightarrow \emptyset = \falsity{\top} \leq \falsity{\psi} \Longleftrightarrow \text{the meet of } \emptyset \text{ and } \falsity{\psi} \text{ equals } \emptyset \Longleftrightarrow \falsity{\psi} = \emptyset \Longleftrightarrow \Booleanforce{\psi}_\mathbb{B} = \emptyset$.
\end{proof}

\section{Comparing numerals} \label{section:ComparingNumerals}

In this section we explicitly give the realizer that was mentioned in \Cref{section:OrdinalRepresentations}. This is a term $\chi$ such that for any $n, m \in \omega$, $t, s, r \in \Lambda$ and $\pi \in \Pi$,
\[
\chi \star \underline{n} \stackapp \underline{m} \stackapp t \stackapp s \stackapp r \stackapp \pi \succ \begin{cases}
    t \star \pi & \text{if }\ n < m, \\
    s \star \pi & \text{if }\ n = m, \\
    r \star \pi & \text{if }\ m < n.
\end{cases}
\]

\begin{lemma}
    Let $X \coloneqq \lambda u \lambdaapp \underline{1}$, $Y \coloneqq \lambda u \lambdaapp \underline{0}$, $A \coloneqq \lambda u \lambdaapp \lambda v \lambdaapp \app{v}{u}$ and $B \coloneqq \lambda u \lambdaapp \lambda v \lambdaapp \inapp{u}{\app{A}{v}}$. Finally, let
    \[
    \theta \coloneqq \lambda i \lambdaapp \lambda j \lambdaapp \superfapp{\bigfapp{\twoapp{\app{j}{A}}{\app{A}{X}}}{\app{i}{B}}}{Y}.
    \]
    Then, for any $n, m \in \omega$ and $\pi \in \Pi$, if $n < m$ then $\theta \star \underline{n} \stackapp \underline{m} \stackapp \pi \succ \underline{0} \star \pi$ and if $n \geq m$ then $\theta \star \underline{n} \stackapp \underline{m} \stackapp \pi \succ \underline{1} \star \pi$.
\end{lemma}

\begin{proof}
    Fix $n, m \in \omega$ and $\pi \in \Pi$. Then $\theta \star \underline{n} \stackapp \underline{m} \stackapp \pi \succ \underline{m} \star A \stackapp \app{A}{X} \stackapp \app{\underline{n}}{B} \stackapp Y \stackapp \pi$. We shall first prove that 
    \[
    \underline{m} \star A \stackapp \app{A}{X} \stackapp \app{\underline{n}}{B} \stackapp Y \stackapp \pi \succ \underline{n} \star B \stackapp \app{A^m}{X} \stackapp Y \stackapp \pi
    \]
    where $x^0 y \coloneqq y$ and $x^{k+1} y \coloneqq \inapp{x}{x^k y}$. To do this, we need to consider the case for $m = 0$ separately;
    \[
    \underline{0} \star A \stackapp \app{A}{X} \stackapp \app{\underline{n}}{B} \stackapp Y \stackapp \pi \succ \app{A}{X} \star \app{\underline{n}}{B} \stackapp Y \stackapp \pi \succ A \star X \stackapp \app{\underline{n}}{B} \stackapp Y \stackapp \pi \succ \fapp{\app{\underline{n}}{B}}{X} \star Y \stackapp \pi \succ \underline{n} \star B \stackapp X \stackapp Y \stackapp \pi.
    \]
    Now, for $m > 0$,
    \begin{align*}
        \underline{m} \star A \stackapp \app{A}{X} \stackapp \app{\underline{n}}{B} \stackapp Y \stackapp \pi & \succ A^m(\app{A}{X}) \star \app{\underline{n}}{B} \stackapp Y \stackapp \pi \succ A \star \app{A^m}{X} \stackapp \app{\underline{n}}{B} \stackapp Y \stackapp \pi \\ 
        & \succ \twoapp{\app{\underline{n}}{B}}{\app{A^m}{X}} \star Y \stackapp \pi \succ \underline{n} \star B \stackapp \app{A^m}{X} \stackapp Y \stackapp \pi,    
    \end{align*}
    where we have used that $A^m(\app{A}{X}) = \app{A^{m+1}}{X} = \inapp{A}{\app{A^m}{X}}$.

    We shall next prove that $\underline{n} \star B \stackapp \app{A^m}{X} \stackapp Y \stackapp \pi \succ \app{A^m}{X} \star \app{A^n}{Y} \stackapp \pi$. In order to do this, we again need to consider the case $n = 0$ separately;
    \[
    \underline{0} \star B \stackapp \app{A^m}{X} \stackapp Y \stackapp \pi \succ \app{A^m}{X} \star Y \stackapp \pi.
    \]
    Now, for $n > 0$, 
    \begin{align*}
    \underline{n} \star B \stackapp \app{A^m}{X} \stackapp Y \stackapp \pi & \succ \inapp{B^n}{\app{A^m}{X}} \star Y \stackapp \pi \succ B \star \inapp{B^{n-1}}{\app{A^m}{X}} \stackapp Y \stackapp \pi \succ \inapp{B^{n-1}}{\app{A^m}{X}} \star \app{A}{Y} \stackapp \pi \\
    & \succ \dots \succ \app{A^m}{X} \star \app{A^n}{Y} \stackapp \pi.
    \end{align*}
    The proof shall be completed once we have shown that $\app{A^m}{X} \star \app{A^n}{Y} \stackapp \pi \succ \underline{0} \stackapp \pi$ if $n < m$ and $\app{A^m}{X} \star \app{A^n}{Y} \stackapp \pi \succ \underline{1} \stackapp \pi$ if $n \geq m$. 
    
    Firstly, if $m = 0$ then
    \[
    X \star \app{A^n}{Y} \stackapp \pi \equiv \lambda u \lambdaapp \underline{1} \star \app{A^n}{Y} \stackapp \pi \succ \underline{1} \star \pi.
    \]
    Next, if $n = 0$ and $m > 0$ then
    \[
    \app{A^m}{X} \star Y \stackapp \pi \succ A \star \app{A^{m-1}}{X} \stackapp Y \stackapp \pi \succ Y \star \app{A^{m-1}}{X} \stackapp \pi \equiv \lambda u \lambdaapp \underline{0} \star \app{A^{m-1}}{X} \stackapp \pi \succ \underline{0} \star \pi. 
    \]
    Finally, if $n$ and $m$ are both greater than $0$ then
    \begin{align*}
    \app{A^m}{X} \star \app{A^n}{Y} \stackapp \pi & \succ A \star \app{A^{m-1}}{X} \stackapp \app{A^n}{Y} \stackapp \pi \succ \app{A^n}{Y} \star \app{A^{m-1}}{X} \stackapp \pi \\ 
    & \succ A \star \app{A^{n-1}}{Y} \stackapp \app{A^{m-1}}{X} \stackapp \pi \succ \app{A^{m-1}}{X} \star \app{A^{n-1}}{Y} \stackapp \pi.
    \end{align*}
    Continuing in this way, for any $0 \leq k \leq \ff{min}(n, m)$ we have
    \[
    \app{A^m}{X} \star \app{A^n}{Y} \stackapp \pi \succ \app{A^{m-k}}{X} \star \app{A^{n-k}}{Y} \stackapp \pi.
    \]
    Therefore, if $n < m$ then 
    \[
    \app{A^m}{X} \star \app{A^n}{Y} \stackapp \pi \succ \app{A^{m-n}}{X} \star Y \stackapp \pi \succ A \star \app{A^{m-n-1}}{X} \star Y \stackapp \pi \succ Y \star \app{A^{m-n-1}}{X} \stackapp \pi \succ \underline{0} \stackapp \pi
    \]
    and if $n \geq m$ then
    \begin{align*}
    \app{A^m}{X} \star \app{A^n}{Y} \stackapp \pi & \succ \app{A^{m-(m-1)}}{X} \star \app{A^{n-(m-1)}}{Y} \stackapp \pi \succ A \star X \stackapp \app{A^{n-(m-1)}}{Y} \stackapp \pi \succ \app{A^{n-(m-1)}}{Y} \star X \stackapp \pi \\
    & \succ A \star \app{A^{n-(m-1)-1}}{Y} \stackapp X \stackapp \pi \succ X \star \app{A^{n-(m-1)-1}}{Y} \stackapp \pi \succ \underline{1} \star \pi,
    \end{align*}
    completing the proof.
\end{proof}

\begin{lemma}
    Fix $\theta$ as in the previous lemma. Let 
    \[
    \chi \coloneqq \lambda i \lambdaapp \lambda j \lambdaapp \lambda e_t \lambdaapp \lambda  e_s \lambdaapp \lambda e_r \lambdaapp \fapp{\twoapp{\fapp{\app{\theta}{i}}{j}}{\lambda u \lambdaapp \fapp{\twoapp{\fapp{\app{\theta}{j}}{i}}{\lambda v \lambdaapp e_s}}{e_r}}}{e_t}.
    \]
    Then for any $n, m \in \omega$, terms $t, s$ and $r$ and $\pi \in \Pi$,
    \begin{thmlist}
        \item if $n < m$ then $\chi \star \underline{n} \stackapp \underline{m} \stackapp t \stackapp s \stackapp r \stackapp \pi \succ t \stackapp \pi$;
        \item if $n = m$ then $\chi \star \underline{n} \stackapp \underline{m} \stackapp t \stackapp s \stackapp r \stackapp \pi \succ s \stackapp \pi$;
        \item if $n > m$ then $\chi \star \underline{n} \stackapp \underline{m} \stackapp t \stackapp s \stackapp r \stackapp \pi \succ r \stackapp \pi$.
    \end{thmlist}
\end{lemma}

\begin{proof}
    Fix $n, m \in \omega$, $t, s, r \in \Lambda$ and $\pi \in \Pi$. Then
    \[
    \chi \star \underline{n} \stackapp \underline{m} \stackapp t \stackapp s \stackapp r \stackapp \pi \succ \theta \star \underline{n} \stackapp \underline{m} \stackapp (\lambda u \lambdaapp \fapp{\twoapp{\fapp{\app{\theta}{\underline{m}}}{\underline{n}}}{\lambda v \lambdaapp s}}{r}) \stackapp t \stackapp \pi.
    \]
    If $n < m$ then $\theta \star \underline{n} \stackapp \underline{m} \stackapp \sigma \succ \underline{0} \star \sigma$ for any $\sigma \in \Pi$. Therefore,
    \[
    \chi \star \underline{n} \stackapp \underline{m} \stackapp t \stackapp s \stackapp r \stackapp \pi \succ \underline{0} \star (\lambda u \lambdaapp \fapp{\twoapp{\fapp{\app{\theta}{\underline{m}}}{\underline{n}}}{\lambda v \lambdaapp s}}{r}) \stackapp t \stackapp \pi \succ t \star \pi.
    \]
    Next, if $n \geq m$ then $\theta \star \underline{n} \stackapp \underline{m} \stackapp \sigma \succ \underline{1} \star \sigma$ for any $\sigma \in \Pi$. Therefore,
    \begin{align*}
    \chi \star \underline{n} \stackapp \underline{m} \stackapp t \stackapp s \stackapp r \stackapp \pi & \succ \underline{1} \star (\lambda u \lambdaapp \fapp{\twoapp{\fapp{\app{\theta}{\underline{m}}}{\underline{n}}}{\lambda v \lambdaapp s}}{r}) \stackapp t \stackapp \pi \succ \lambda u \lambdaapp \fapp{\twoapp{\fapp{\app{\theta}{\underline{m}}}{\underline{n}}}{\lambda v \lambdaapp s}}{r} \star t \stackapp \pi \\
    & \succ \fapp{\twoapp{\fapp{\app{\theta}{\underline{m}}}{\underline{n}}}{\lambda v \lambdaapp s}}{r} \star \pi \succ \theta \star \underline{m} \stackapp \underline{n} \stackapp (\lambda v \lambdaapp s) \stackapp r \stackapp \pi.
    \end{align*}
    Now, if $n = m$ then $\theta \star \underline{m} \stackapp \underline{n} \stackapp \sigma \succ \underline{1} \star \sigma$ for any $\sigma \in \Pi$. Therefore,
    \[
    \chi \star \underline{n} \stackapp \underline{m} \stackapp t \stackapp s \stackapp r \stackapp \pi \succ \underline{1} \star (\lambda v \lambdaapp s) \stackapp r \stackapp \pi \succ \fapp{\lambda v \lambdaapp s}{r} \star \pi \succ \lambda v \lambdaapp s \star r \stackapp \pi \succ s \star \pi.
    \]
    Finally, if $n > m$ then 
    \[
    \chi \star \underline{n} \stackapp \underline{m} \stackapp t \stackapp s \stackapp r \stackapp \pi \succ \underline{0} \star (\lambda v \lambdaapp s) \stackapp r \stackapp \pi \succ r \star \pi.
    \]
\end{proof}

\bibliography{surveybib}
\bibliographystyle{alpha}

\end{document}